\PassOptionsToPackage{dvipsnames,svgnames}{xcolor}
\documentclass[12pt,reqno]{amsart}
\usepackage{amsmath, amssymb, amsfonts, tikz, enumerate, graphicx, textcomp, caption, amsthm, listings, verbatim, mathrsfs, url, tipa, datetime}
\usepackage[sort,round]{natbib}
\setcitestyle{numbers}
\usepackage[colorlinks=true,hyperindex, linkcolor=magenta, pagebackref=false, citecolor=cyan,pdfpagelabels]{hyperref}
\usepackage{array}
\usepackage{booktabs}
\usepackage[margin=3cm]{geometry}

\newcolumntype{L}{>{$}l<{$}}
\usetikzlibrary{calc, arrows}
\newcommand{\stackyq}{\text{ }/\kern-0.3cm /\text{ }}

\definecolor{rinpink}{RGB}{214,30,130}
\newtheoremstyle{pinkq}{1.1\baselineskip plus .3\baselineskip minus .1\baselineskip}{1.1\baselineskip plus .3\baselineskip minus .1\baselineskip}{\addtolength{\leftskip}{2em}\addtolength{\rightskip}{2em}}{}{\bfseries\color{rinpink}}{.}{ }{\thmname{#1}\thmnote{ (#3)}}

\newtheorem{manuallemmainner}{Lemma}
\newenvironment{manuallemma}[1]{%
  \renewcommand\themanuallemmainner{#1}%
  \manuallemmainner
}{\endmanuallemmainner}

\DeclareMathOperator{\Aut}{Aut}
\DeclareMathOperator{\End}{End}

\DeclareMathOperator{\Ind}{Ind}
\newcommand{\mr}{\mathrm}
\DeclareMathOperator{\Hom}{Hom}
\DeclareMathOperator{\Jac}{Jac}
\DeclareMathOperator{\Sym}{Sym}
\DeclareMathOperator{\Gal}{Gal}

\DeclareMathOperator{\Ext}{Ext}
\DeclareMathOperator{\colim}{colim}

\DeclareMathOperator{\QCoh}{QCoh}
\DeclareMathOperator{\Mod}{Mod}

\DeclareMathOperator{\Spf}{Spf}
\DeclareMathOperator{\Spa}{Spa}
\DeclareMathOperator{\Shv}{Shv}
\DeclareMathOperator{\Def}{Def}

\DeclareMathOperator{\Pic}{Pic}
\DeclareMathOperator{\Div}{Div}
\DeclareMathOperator{\Spec}{Spec}

\usepackage{tikz-cd}\tikzset{node distance=2cm, auto}

\newcommand{\Z}{\mathbb{Z}}

\newcommand{\F}{\mathbb{F}}
\newcommand{\G}{\mathbb{G}}
\newcommand{\Q}{\mathbb{Q}}

\newcommand{\A}{\mathbb{A}}

\newcommand{\Lie}{\mathrm{Lie}}

\newcommand{\OO}{\mathcal{O}}

\newcommand{\mc}{\mathcal}

\newcommand{\bb}{\mathbb}
\newcommand{\Art}{\ms{Art}}

\newcommand{\ms}{\mathsf}
\newcommand{\mf}{\mathfrak}

\newcommand{\pt}{\mr{pt}}

\definecolor{codegray}{gray}{0.9}

\newtheorem{defn}{Definition}[section]
\newtheorem{theorem}[defn]{Theorem}
\newtheorem{lemma}[defn]{Lemma}
\newtheorem{cor}[defn]{Corollary}

\newtheorem{example}[defn]{Example}
\newtheorem{prop}[defn]{Proposition}

\newtheorem{theorem*}{Theorem}
\newtheorem{conv}{Convention}
\newtheorem*{proposition}{Proposition}
\newtheorem*{construction}{Construction}

{\theoremstyle{pinkq}}
{\theoremstyle{pinkq}}
{\theoremstyle{pinkq}}
{\theoremstyle{pinkq}}

\theoremstyle{remark}
\newtheorem*{remark}{Remark}
\newtheorem*{notation}{Notation}

{\theoremstyle{pinkq}}
{\theoremstyle{pinkq}}
{\theoremstyle{pinkq}\newtheorem*{question}{Question}}

\usepackage{xcolor}

\title{Modeling Group Actions on Stacks \\ (Especially the Lubin--Tate Action)}
\author{Rin Ray}
\date{}

\tikzset{curve/.style={settings={#1},to path={(\tikztostart)
    .. controls ($(\tikztostart)!\pv{pos}!(\tikztotarget)!\pv{height}!270:(\tikztotarget)$)
    and ($(\tikztostart)!1-\pv{pos}!(\tikztotarget)!\pv{height}!270:(\tikztotarget)$)
    .. (\tikztotarget)\tikztonodes}},
    settings/.code={\tikzset{quiver/.cd,#1}
        \def\pv##1{\pgfkeysvalueof{/tikz/quiver/##1}}},
    quiver/.cd,pos/.initial=0.35,height/.initial=0}
\tikzset{tail reversed/.code={\pgfsetarrowsstart{tikzcd to}}}

\newcommand{\cA}{\mathcal{A}}
\newcommand{\fm}{\mathfrak{m}}
\newcommand{\fp}{\mathfrak{p}}
\newcommand{\fS}{\mathfrak{S}}

\DeclareMathOperator{\rad}{rad}

\newcommand{\PP}{\mathbb{P}}
\DeclareMathOperator{\Orb}{Orb}
\DeclareMathOperator{\supp}{supp}

{\theoremstyle{pinkq}}
{\theoremstyle{pinkq}}
{\theoremstyle{pinkq}\newtheorem*{answer}{Answer}}

\makeatletter
\renewcommand\part[1]{%
  \par\vspace{2.2\linespacing plus \linespacing}%
  \refstepcounter{part}%
  \addcontentsline{toc}{part}{\tocpart{\partname}{\thepart}{#1}}%
  \begingroup\centering\normalfont
  {\large\scshape\partname\ \thepart}\\[0.45\linespacing]
  {\LARGE\bfseries #1}\par
  \endgroup
  \vspace{1.3\linespacing}\@afterindentfalse\@afterheading}
\makeatother
\newcommand{\partgloss}[1]{\begin{center}\itshape #1\end{center}\vspace{0.5\baselineskip}}

\begin{document}

\begin{abstract} Suppose we are given a profinite group $G$ acting on a formal moduli stack $\mc{M}$, and we want to understand the group action, and compute cohomology related to this group action. How can we do it?

This prolegomenon conveys two methods of pinning down such an action: geometric modeling and the two-tower method. We highlight their use on a specific action --- the automorphisms of a formal group acting on its deformation space, called the Lubin--Tate action.

We provide the toolkit that these two methods run on. The pieces of the kit are standard in their own subjects, but the Lubin--Tate action requires us to sew them together and ask more of them than their usual statements give, and several are set down here in that stronger form: equivariantly, in families, and with the grading kept. The tools are put to work in the two sequels, which build geometric models at height $p-1$ and at the heights $p^{k-1}(p-1)$ above it; the two-tower material is a companion to the proof of the coheight one weak chromatic splitting conjecture.
\end{abstract}
\maketitle

\begin{center}\includegraphics[width=11cm]{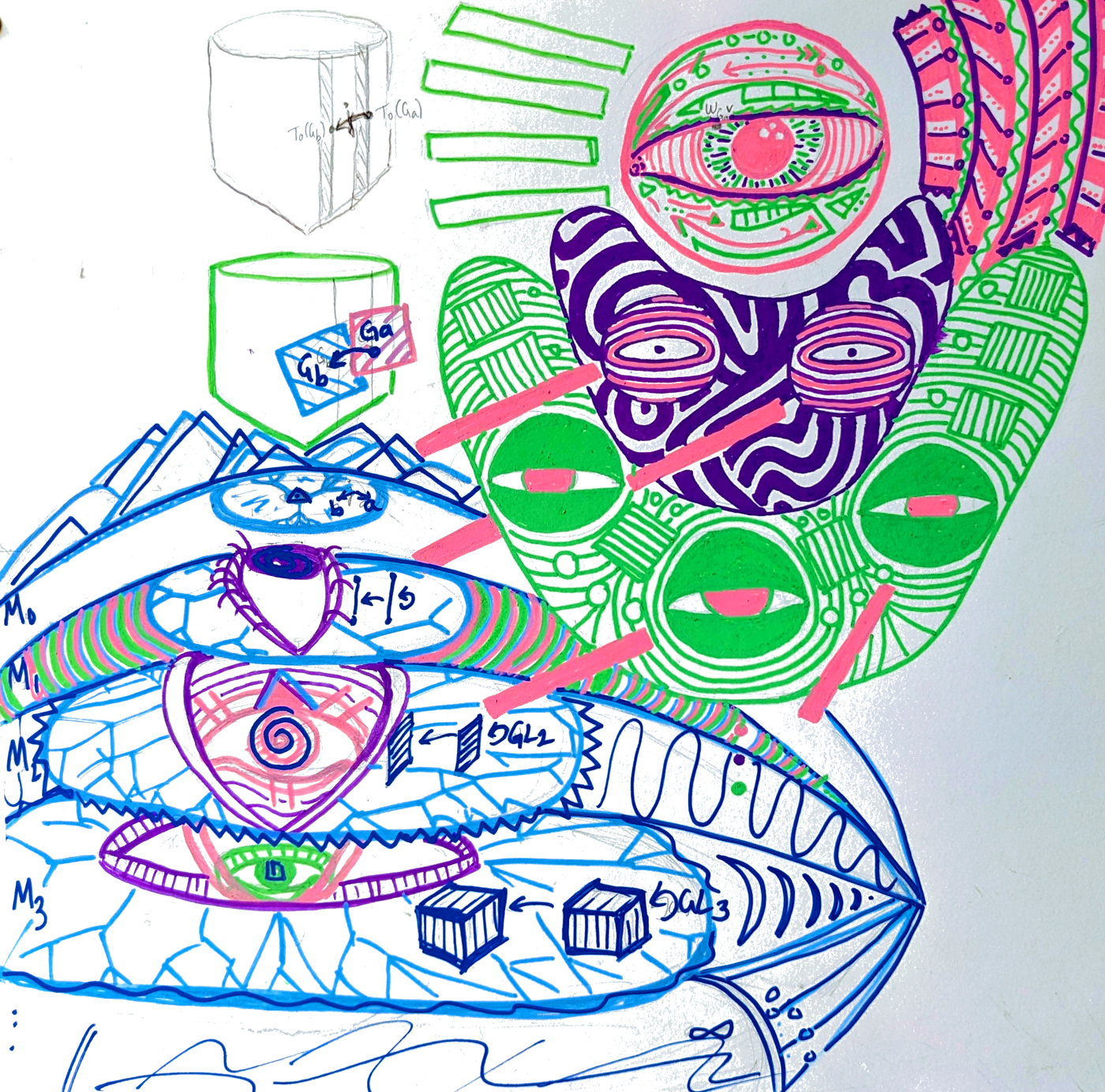}\end{center}

\newpage
\tableofcontents

\section*{Introduction} 
Suppose we are given a pro-finite group $G$ acting on a formal moduli stack $\mc{N}$, and we want to understand the group action, and compute cohomology related to this group action. How can we do it? When and how can we capture information about a group action on a moduli stack $\mc{N}$ by using a more understandable group action on a different moduli stack?

In this survey article, I will exposit two main methods, which for lack better terms, I refer to as geometric modelling and the two-tower method. 

\begin{enumerate}
\item Geometric Modelling:  Given $M, N$ prestacks, both carrying a $G$-action, and a functor $\mc{F} \colon M \to N$ which is a $G$-equivariant equivalence.
\[\begin{tikzcd}
	M & N
	\arrow["G", from=1-1, to=1-1, loop, in=145, out=215, distance=10mm]
	\arrow["{\mc{F}}", from=1-1, to=1-2]
	\arrow["\simeq"', from=1-1, to=1-2]
	\arrow["G", from=1-2, to=1-2, loop, in=325, out=35, distance=10mm]
\end{tikzcd}\] \noindent then, there is an isomorphism of quotient stacks
$$M/ G \simeq N\stackyq G$$
\item Two Tower Method: Given $N, N', \text{and } M$ prestacks, such that $M$ carries commuting actions of $G$ and $G'$ (that is, an action of $G \times G'$) making it a $G$-torsor over $N$ and a $G'$-torsor over $N'$,
\[\begin{tikzcd}
	& M \\
	N & {} & N'
	\arrow["{G \times G'}", from=1-2, to=1-2, loop, in=55, out=125, distance=10mm]
	\arrow["G"', from=1-2, to=2-1]
	\arrow["{{{G'}}}", from=1-2, to=2-3]
	\arrow["{{G'}}", from=2-1, to=2-1, loop, in=190, out=260, distance=10mm]
	\arrow["G", from=2-3, to=2-3, loop, in=280, out=350, distance=10mm]
\end{tikzcd}\]

\noindent then, the quotients of objects by the residual actions are isomorphic, \noindent $$N \stackyq G' \simeq N'\stackyq G.$$
\end{enumerate}

We are primarily concerned with and motivated by one action in particular which arises in several rich and related contexts. It turned out to be the case that in order to establish a clean framework to discuss that action, it arises as an example of a general framework which works for all stacks which are locally quotients of affines by profinite groups.  We fill a gap in the literature in the discussion of such stacks, which lie between Deligne--Mumford stacks, which are locally quotients by finite groups, and Artin stacks, which are locally quotients by algebraic groups.

We set up the sites (pro-\'etale, Convention \ref{conv:proet}) so that the coherent cohomology of our stacks gives us continuous group cohomology of the global sections of their structure sheaves; the finite case for usual group cohomology being a special case. Along the way we explain how the \emph{smooth} actions seen by stacks compare with the \emph{condensed} actions used in chromatic homotopy theory, following \cite[Appendix A]{chromsplit}. %The reader with interest only in finite groups can happily cover their eyes on this point and skip the stack background section. 

\begin{manuallemma}{A} Let $G$ be a profinite group acting on an affine scheme $X$ through the constant group scheme $\underline{G}$. Given a quasi-coherent sheaf $\mc{F} \in \QCoh(X\stackyq G)$, we have an isomorphism
$$H^*(X\stackyq G, \mc{F})\simeq H^*_{cts}(G, \mc{F}(X)),$$ 
with $\mc{F}(X)$ discrete (Lemma \ref{howitbe}). For $X = \Spf A$ an affine formal scheme with continuous $G$-action, the same holds for $\mc{F} = \OO$ with $A$ given its adic topology (Corollary \ref{continuousboogie}). \end{manuallemma}

Lemma A computes \emph{smooth} group cohomology, with locally constant cochains, while chromatic homotopy theory needs \emph{continuous} cohomology for the adic topology on $\pi_0 E_h = W(k)[[u_1, \dots, u_{h-1}]]$. The two agree whenever the coefficients are a limit of discrete modules, which is the case for every deformation ring and every closed subgroup of $\G_h$ (Section \ref{continuouscohom}). They differ for transchromatic coefficients such as $\pi_0 L_{K(h-1)}E_h$, whose residue field $k(\!(u_{h-1})\!)$ carries a topology that is not adic for any ideal of the ring; there one must work with solid coefficients and condensed actions from the start, which is what \cite{chromsplit} does (Remark \ref{rem:solidneeded}).

The action of primary interest is that of the automorphisms  $\Aut_k(F)$ of a one-dimensional formal group $F$ on its deformation stack. This action is colloquially referred to as the ``Lubin--Tate action," after the seminal work of Lubin and Tate, who introduced formal groups to give a $p$-adic analogue of the theory of complex multiplication in local class field theory \cite{LT65} and then constructed the universal deformation on which this action lives \cite{LT}.

This action crucially appears in the crux of computing homotopy groups of spheres \cite{quill, goerss} and in the Jacquet--Langlands correspondence \cite{carayol83, carayol90}. Further, understanding this action bears on the remaining unitary case of the Hecke orbit conjecture regarding the density of the Hecke action, as the stabilizer of the Hecke action at a point is $\Aut_k(F)$ \cite{chai}.

\subsection{Five Angles on One Action}
\label{sec:angles}

The Lubin--Tate action is one object seen from five directions.

\subsubsection{Arithmetic}
The action predates its homotopical use. Lubin--Tate formal groups were
introduced to make local class field theory explicit \cite{LT65}, and the
non-abelian version is the Lubin--Tate tower, whose $\ell$-adic cohomology
realizes the local Langlands and Jacquet--Langlands correspondences for
$\operatorname{GL}_h(\mathbb{Q}_p)$ \cite{HT01,boyer}. The stabilizer group
$\mathbb{G}_h$ acts compatibly with the Gross--Hopkins period map to
$\mathbb{P}^{h-1}$ \cite{HG94b} and with the duality between the
Lubin--Tate and Drinfeld towers \cite{schollywein}; the $p$-adic holography
models \cite{GubserEtAl,HMSS} are built on the Drinfeld side of this
duality via Bruhat--Tits buildings and Mumford curves. A statement about the action proved on either side is a statement about the same object.

\subsubsection{Physical}
A quantum field theory on a genus-$g$ worldsheet has partition function a
section of a line bundle on $\mathcal{M}_g$; the curve $X$ is a point of
that moduli space, not a worldsheet. At $p=3$ and $h=2$, we consider the supersingular
elliptic curve $y^2=x^3-x$, and $E_2^{hG_{24}}\simeq L_{K(2)}\operatorname{tmf}$:
the deformation space of the curve is Lubin--Tate space, and the spectrum
that sees it is topological modular forms. Stolz--Teichner \cite{StolzTeichner}
conjecture that $\operatorname{tmf}$ classifies $2$-dimensional
$\mathcal{N}{=}(0,1)$ supersymmetric theories up to deformation --- the
Witten genus \cite{WittenEllipticGenera} is its partition function, and
what it misses is torsion. That torsion, $3$-locally, is the cohomology of
the curve's automorphism group acting on its own deformation space: the
object this paper computes at every $p$. It is realized by sigma models
\cite{GJFW}, detected by a secondary genus \cite{GJF}, and encodes global
anomalies of heterotic strings \cite{TachikawaYamashita}. For $p\geq 5$ no field theory is known whose partition function sees $E_{p-1}$; the Stolz--Teichner ladder has no rung above $\operatorname{tmf}$. This paper does not build that rung. It builds the geometric data any such
theory would have to reproduce: the formal neighborhood of a single curve
in $\mathcal{M}_g$, with its full symmetry group acting. If the ladder
extends, it extends here.

\subsubsection{Chromatic}
The moduli stack of one-dimensional formal groups controls the stable
homotopy category \cite{quill,morava,ravperiod}: height stratifies the
stack, the chromatic tower filters the sphere, and chromatic convergence
\cite{hoprav} makes the filtration exhaustive. By Devinatz--Hopkins
\cite{devhop} the $K(h)$-local category is the $\mathbb{G}_h$-equivariant
homotopy theory of $E_h$, and $E_h$ is a pro-$\mathbb{G}_h$-Galois
extension of $L_{K(h)}\mathbb{S}$ \cite{devhop,rognes}. Layer by layer,
the Lubin--Tate action is the entire input to the stable homotopy groups of
spheres; a finite subgroup $G'\leq\mathbb{G}_h$ cuts out a sub-Galois
extension --- the one this paper computes at $h=p-1$.

\subsubsection{Hecke}
Behrens \cite{Beh06} resolves the $K(2)$-local sphere at $p=3$ from
$\operatorname{TMF}$ and $\operatorname{TMF}_0(\ell)$ via the Hecke
operator $T_\ell$; Behrens--Lawson \cite{taf} run the same program at
every height through unitary Shimura varieties, expressing $K(h)$-locally
that topological automorphic forms are products of fixed points
$E_h^{h\operatorname{Aut}(x)}$ over a dense supersingular locus
\cite{Nau08}. The maximal finite subgroup stabilizes the deepest building
vertex, is the automorphism group of the distinguished supersingular point,
and is the torsion of the maximal order through which Hecke operators act
\cite{Rez06,Zhu20}. Running this program at height $p-1$ needs exactly two
local inputs: the group $G'$, and $\pi_*E_{p-1}$ as a $G'$-module. The
second is supplied by this paper.

\subsubsection{Computational}
$E_\infty$-rigidity gives fixed points $E_h^{hF}$ for every finite
$F\leq\mathbb{G}_h$: higher real $K$-theories, beginning with
$KO^\wedge_p = E_1^{hC_2}$ and $K(2)$-local $\operatorname{TMF}$. These
are the most computable approximations to $L_{K(h)}\mathbb{S}$: a form of
$E_4^{hC_8}$ detects the Kervaire classes \cite{HHR}, and
$E_{p-1}^{hC_p}$ resolved the Smith--Toda program negatively \cite{nave}.
The return from finite subgroups to $\mathbb{G}_h$ is systematic by Henn
\cite{henn}, with finite resolutions realized at $h=2$ \cite{ghmr,BG18};
the layers are governed by Tate cohomology $\widehat{H}^*(F,\pi_*E_h)$.
The sequel computes that cohomology at height $p-1$ \cite[Thm.\ E]{sequel}.

\subsection{Examples of Modeling the Lubin--Tate Action}\label{sec:lubintateexamples}
Here is a list of examples of stacks that model the Lubin--Tate action that have been considered and shown to model the action to varying degree.  % (at various levels of understandibility, usability, and rigor). 

Let us consider a formal group $F$ of height $h$ over a field $k$ of positive characteristic, and possibly with decoration (graded formal group, formal group with level structure, formal $R$-module, etc). We call the stack of one dimensional formal groups (possibly with decorations) $\mc{M}_{\mr{fg}_1}$. Formal groups over separably closed fields are classified up to isomorphism by an invariant called height. The deformations and automorphisms of formal groups $F'$ and $F$ of the same height over such a field are thus equivalent. The group $\Aut_k(F)$ is a profinite group; when $k \supseteq \F_{p^h}$ it is the group of units $\OO_D^\times$ of the maximal order in the central $\Q_p$-division algebra $D$ of Hasse invariant $1/h$.
\[\begin{tikzcd}
	{\Aut_k(F)} & {\Aut_k(F')} \\
	{\Def_F} & {\Def_{F'}}
	\arrow["\simeq", from=1-1, to=1-2]
	\arrow[from=2-1, to=2-1, loop, in=55, out=125, distance=10mm]
	\arrow["\simeq", from=2-1, to=2-2]
	\arrow[from=2-2, to=2-2, loop, in=55, out=125, distance=10mm]
\end{tikzcd}\]

Geometric modelling is usually done for finite subgroups $G \subseteq \Aut_k(F)$. The Lubin--Tate action is in some sense too floppy when considered just acting on the formal scheme, so we use a much more rigid setting with algebro-geometric structure which restricts or maps to the Lubin--Tate action in order to compute it. We craft a puppet which shows us the secret dance. 

In this setting, our first task is finding an object $X$ such that $G \subseteq \Aut_k(X)$, and a map $\mc{F}$ such that $\mc{F}(X)$ is a one dimensional formal group of height $h$ (possibly with decoration). Next, we find a prestack $M$, such that $X \in M(k)$, and construct a functor $$\mc{F}: M \to \mc{M}_{\mr{fg}_1},$$ which is injective on $G \subseteq \Aut_k(X)$ and induces an equivalence on the deformation problems (Definition \ref{def:geometricmodel}).
\[\begin{tikzcd}
	{G \subseteq \Aut_k(X)} & {\Aut(\mc{F}(X))} \\
	{\Def_X} & {\Def_{\mc{F}(X)}}
	\arrow[hook, from=1-1, to=1-2]
	\arrow[from=2-1, to=2-1, loop, in=55, out=125, distance=10mm]
	\arrow["\simeq"', from=2-1, to=2-2]
	\arrow["{\mc{F}}", from=2-1, to=2-2]
	\arrow[from=2-2, to=2-2, loop, in=55, out=125, distance=10mm]
\end{tikzcd}\]

\noindent then, there is an isomorphism of quotient stacks $$\Def_X\stackyq G \simeq \Def_{\mc{F}(X)}\stackyq G.$$

Let's speed through some examples of such geometric modelling. Let $G \subseteq \Aut_k(F)$ be a finite subgroup.
\begin{itemize}
\item Let $E$ be a supersingular elliptic curve (with Drinfeld level-$N$ structure) such that $G \subseteq \Aut_k(E)$. All that is said below works with and without level structure, and the level-$N$ chosen depends on the prime of interest. Thanks to Serre--Tate \cite{st, katz}, we know that deformations of an elliptic curve coincide with deformations of its formal group $\mc{F}(E, L)$, which is isomorphic to any other formal group of the same height $F$ (compatibly with level structures). The moduli stack of elliptic curves $\mc{M}_{1, 1}^{N}$ completes at a point $(E, L)$ to a deformation problem of the elliptic curve $(E, L)$.
\[\begin{tikzcd}
	& {\Aut_k(E)} & {\Aut(\mc{F}(E, L))} \\
	{\mc{M}_{1, 1}^{\mr{lvl }N}} & {\Def_{(E, L)}} & {\Def_{\mc{F}(E, L)}} & {}
	\arrow[hook, from=1-2, to=1-3]
	\arrow["{{(-)^\wedge_{(E, L)}}}", from=2-1, to=2-2]
	\arrow[from=2-2, to=2-2, loop, in=55, out=125, distance=10mm]
	\arrow["\simeq", from=2-2, to=2-3]
	\arrow[from=2-3, to=2-3, loop, in=55, out=125, distance=10mm]
\end{tikzcd}\]
\noindent The stack $\mc{M}_{1, 1}^{\text{lvl }N}$ is the underlying stack of the spectral stack of topological modular forms $\mr{TMF}(N)$ with appropriate level structure. This was originally constructed by Hopkins-Miller, Goerss-Hopkins constructed it as an $E_\infty$-ring spectra, and then Lurie gave a conceptual approach using spectral algebraic geometry \cite{luriesurvey, arbeitsgemeinschaft}. $\mr{TMF}$ has been used extensively to explore homotopy groups of spheres at height 1 and 2 \cite{vesna, dylan} \cite{meier}. Points on $\mr{TMF}$ are Morava $K(1)$ and $K(2)$, and neighborhoods of such points are $E(1)$ and $E(2).$ The first Morava $E$-theory was constructed by Morava \cite{moravaforms} by considering the Tate curve as a deformation neighborhood of $\G_m$ in compactified $\mc{M}_{1, 1}$, and using this to deform $KU/p$. 
\item At height $p-1$, the modern perspective on this will be covered in the next paper. The order-$p$ subgroup of $\Aut_k(F)$ at this height was used by Ravenel \cite{Rav} to show the non-existence of the odd-primary Kervaire (Arf) invariant elements for $p \geq 5$; curves modelling the action were introduced later by Gorbounov--Mahowald \cite{gm}. We construct the minimal genus curve $X$ such that $G \subseteq \Aut_k(X)$. Then, we work to construct a functor $\mc{F}$ such that $\mc{F}(X, G)$ is a one dimensional formal group of height $p-1$ and $\mc{F}$ induces a $G$-equivariant isomorphism:
\[\begin{tikzcd}
	{G \subseteq \Aut_k(X)} & {\Aut(\mc{F}(X, G))} \\
	{\Def_{(X, G)}} & {\Def_{\mc{F}(X, G)}} & {}
	\arrow[hook, from=1-1, to=1-2]
	\arrow[from=2-1, to=2-1, loop, in=55, out=125, distance=10mm]
	\arrow["\simeq", from=2-1, to=2-2]
	\arrow[from=2-2, to=2-2, loop, in=55, out=125, distance=10mm]
\end{tikzcd}\] The global spectral stack $\mr{eo}_{p-1}$ for which this is an underlying local neighborhood was constructed by Hill \cite{hillphd}. 
\item The following example works for the full profinite group $\Aut_k(F)$. Consider a stack $\mc{S}$ which is a PEL Shimura variety for $U(1, h-1)$. This is a moduli stack of abelian varieties with extra structure. In particular, their formal groups are $h$ copies of the same height $h$ one-dimensional formal group, thus there is a natural projection from a given $s \in \mc{S}(k)$ to one copy of a height $h$ one dimensional formal group over $k$. \[\begin{tikzcd}
	{\text{Hecke}} & {\Aut(s)} & {\Aut_k(F)} \\
	S & {\Def_s} & {\Def_{F}} \\
	\arrow["{\mr{Stab}}", from=1-1, to=1-2]
	\arrow[from=1-2, to=1-3]
	\arrow[from=2-1, to=2-1, loop, in=55, out=125, distance=10mm]
	\arrow["{(-)^\wedge_{s}}", from=2-1, to=2-2]
	\arrow[from=2-2, to=2-2, loop, in=55, out=125, distance=10mm]
	\arrow["\simeq", from=2-2, to=2-3]
	\arrow[from=2-3, to=2-3, loop, in=55, out=125, distance=10mm]
\end{tikzcd}\] This stack was first considered by Carayol in his exploration of the Jacquet-Langlands correspondence \cite{carayol90}, and later by Rapoport--Zink in their consideration of $p$-adic period morphisms and non-archimedean uniformization theorems for general Shimura varieties \cite{rappoportzink}. The stack $\mc{S}$ is the underlying stack of the spectral stack of topological automorphic forms $\mr{TAF}$ which was constructed and considered by Behrens-Lawson and Hill \cite{taf, taf2} \cite{tafhill}.
\end{itemize}

The two-tower methodology is the main connection between chromatic homotopy theory and the Jacquet-Langlands correspondence which studies the relationship between $GL_h(\Q_p)$ and $D^\times$. 

\begin{itemize} 
\item For the full profinite group $\G_h := \Aut_k(F)$, note that $\OO_D^\times \simeq \Aut_k(F)$ for $D$ a $\Q_p$ division algebra with Hasse invariant $1/h$. Let $(\mc{H}^{h-1}_{\breve{\bb{Q}}_p})^{\diamond}$ be the diamond of the Drinfeld upper half space, and let $(\Def_F^\star)^\diamond$ denote the $\Q_p$-diamond of $\Def_F^\star.$ The functor of points of the torsor $\mc{M}$ (the Lubin--Tate tower at infinite level, which by Scholze--Weinstein \cite{schollywein} is also the Drinfeld tower at infinite level) may be described as the locus of maps $$[S \mapsto \Hom_{\OO_{FF_S}}(\OO^{\oplus h}, \OO(\tfrac{1}{h}))]$$ satisfying a nondegeneracy condition, where $FF_S$ is the relative Fargues--Fontaine curve.
\[\begin{tikzcd}
	& {\mc{M}} \\
	{(\mc{H}^{h-1}_{\breve{\bb{Q}}_p})^\diamond} && {(\Def_F^\star)^\diamond}
	\arrow["{{D^\times \times GL_h(\bb{Q}_p)}}", from=1-2, to=1-2, loop, in=55, out=125, distance=10mm]
	\arrow["{D^\times}"', from=1-2, to=2-1]
	\arrow["{GL_h(\bb{Q}_p)}", from=1-2, to=2-3]
	\arrow["{GL_h(\Q_p)}", from=2-1, to=2-1, loop, in=190, out=260, distance=10mm]
	\arrow["{D^\times}", from=2-3, to=2-3, loop, in=280, out=350, distance=10mm]
\end{tikzcd}\]

\noindent then, there is an isomorphism of quotient stacks $$(\mc{H}^{h-1}_{\breve{\bb{Q}}_p})^\diamond\stackyq GL_h(\Q_p) \simeq (\Def_F^\star)^\diamond\stackyq D^\times.$$ Restricting to compact groups, the same diagram yields Faltings' isomorphism of v-stacks (here, as in \cite{rats, chromsplit}, $\G_h$ denotes the extended stabilizer group $\Aut_k(F) \rtimes \Gal(k/\F_p)$) \cite{faltings, schollywein}
$$\big[\mr{LT}^{\diamond}_{h,0}\stackyq \G_h\big] \simeq \big[(\mc{H}^{h-1})^{\diamond}\stackyq GL_h(\Z_p)\big],$$
where $\mr{LT}_h = \Spf \pi_0 E_h$ and $\mr{LT}^\diamond_{h,0}$ is its generic fiber. Both groups are now profinite, so the formalism of Section \ref{stacks} applies directly. This allows us to get a handle on the rational (torsion-free) information of the action of $\Aut_k(F)$ on $\Def_F.$ It was used to great effect by Barthel--Schlank--Stapleton--Weinstein \cite{rats} to compute the rationalization of the $K(h)$-local sphere, via $\Q \otimes H^*_{cts}(\G_h, \pi_0 E_h) \simeq H^*_{cts}(GL_h(\Z_p), \Q_p)$.
\item A characteristic $p$ version of the two-tower isomorphism is proved in \cite{chromsplit}. Stratify the adic Lubin--Tate space $\mr{LT}^{ad}_h = \coprod_{t=0}^{h} \mr{LT}_{h,t}$, where on $\mr{LT}_{h,t}$ the formal group has \'etale part of height $h-t$; e.g. $\mr{LT}_{h,h-1} = \Spa k((u_{h-1}))$. Over $\mr{LT}^\diamond_{h,t}$ there is a $\G_h$-equivariant pro-\'etale torsor $\mr{LT}^\infty_{h,t}$ for $\G_t \times GL_{h-t}(\Z_p)$ trivializing $\Gamma^{un}_h \simeq \Gamma_t \oplus (\Q_p\stackyq \Z_p)^{h-t}$, and its quotient $\Omega_{h,t} := \mr{LT}^\infty_{h,t}\stackyq \G_h$ is an open subspace of the negative Banach--Colmez space $BC(-\tfrac{1}{t})^{\oplus(h-t)}$, where $BC(-\tfrac1t)(S) = H^1(X_S, \OO(-\tfrac{1}{t}))$ on the relative Fargues--Fontaine curve $X_S$ \cite[Theorem H]{chromsplit}. The resulting span of pro-\'etale torsors
\[\begin{tikzcd}
	& {\mr{LT}^\infty_{h,t}} \\
	{\mr{LT}^\diamond_{h,t}} && {\Omega_{h,t}}
	\arrow["{\G_t \times GL_{h-t}(\Z_p)}"', from=1-2, to=2-1]
	\arrow["{\G_h}", from=1-2, to=2-3]
\end{tikzcd}\]
gives, by Lemma \ref{spanny}, an isomorphism of v-stacks
$$\big[\mr{LT}^\diamond_{h,t}\stackyq \G_h\big] \simeq \big[\Omega_{h,t}\stackyq (\G_t \times GL_{h-t}(\Z_p))\big].$$
In coheight one ($t = h-1$) this controls the continuous cohomology of $\G_h$ with coefficients in the completed Gross--Torii field $\widehat{L}$, and was used in \cite{chromsplit} to prove the coheight one weak chromatic splitting conjecture for $p-1 \nmid h-1$. There, the coefficients are no longer adically topologized, and one must work with condensed (solid) coefficients; see Remark \ref{rem:solidneeded}.
\end{itemize}

\subsection{Which Finite Subgroups Matter, and Why}
\label{sec:finitematters}

We separate three claims of different strength; together they explain what the sequel\textquoteright s main computations \cite[Thms.\ D, E]{sequel} are for.

\emph{Claim 1 (group theory).} $\mathbb{G}_h$ contains $p$-torsion if and
only if $p-1\mid h$: an element of order $p$ generates $\mathbb{Q}_p(\zeta_p)$
inside the division algebra $D$, and $[\mathbb{Q}_p(\zeta_p):\mathbb{Q}_p]=p-1$.
The maximal finite subgroups are classified.

\begin{lemma}\label{bouj}\cite{bouj,Hewett}
For $p>2$ and $h=(p-1)p^{k-1}m$ with $p\nmid m$, $\mathbb{G}_h$ has
exactly $k+1$ conjugacy classes of maximal finite subgroups,
$$\mathbb{Z}/(p^h-1)\qquad\text{and}\qquad
\mathbb{Z}/p^i\rtimes\mathbb{Z}/(p^{h_i}-1)(p-1)\quad(1\le i\le k),$$
where $h_i=h/\varphi(p^i)$. For $i=k$ the subgroup is
$\mathbb{Z}/p^k\rtimes\mathbb{Z}/(p-1)^2$, in a single conjugacy class.
\end{lemma}

\emph{Claim 2 (cohomology).} Torsion becomes cohomological structure.
Subgroups containing $\mathbb{Z}/p$ have periodic Tate cohomology, which
feeds periodicity classes into every descent spectral sequence
$H^*(G',\pi_*E_h)\Rightarrow\pi_*E_h^{hG'}$. The shape of
the sequel\textquoteright s Theorem E \cite[Thm.\ E]{sequel} --- a polynomial generator $\beta$ and an
invertible $\Delta$ --- is this periodicity at height $p-1$.

\emph{Claim 3 (homotopy).} For $2p-2>h^2+h$ the $K(h)$-local category is
algebraic \cite{pstra,BSS20}: no exotic phenomena. So subgroups with
$p$-torsion are the only possible carriers of non-algebraic behavior.
That they always carry something is not a theorem; it is the verdict of
every computed case, most directly Nave's theorem \cite{nave}, which is
Claim~2 for the wild subgroup $C_p$ alone. The sequel \cite[Thms.\ D, E]{sequel} supplies the Claim~2 output for the full $G'$. This is the input
any Claim~3 statement at height $p-1$ must consume.

At $h=p-1$, the first height with $p$-torsion, the maximal finite subgroup
is $G'=\mathbb{Z}/p\rtimes\mathbb{Z}/(p-1)^2$, and $\pi_*(E_{p-1}^{hG'})$
is the goal of the sequel; heights $p^{k-1}(p-1)$ for $k\geq 2$ are the third paper, \emph{Only the Gauge Grows} \cite{higherlevels}.

\emph{The next height with $p$-torsion, and the Kervaire invariant.} At $h = p^{k-1}(p-1)$ with $k \geq 2$ the maximal finite subgroup is $\Z/p^k \rtimes \Z/(p-1)^2$, and the first case, $p = 3$, $k = 2$, $h = 6$, is where the odd-primary Kervaire invariant problem lives: Hill--Hopkins--Ravenel propose that the last open case, $p=3$, be decided by the homotopy fixed point spectral sequence for $\pi_* E_6^{hC_9}$, $C_9 \subset \G_6$ the wild part of the maximal finite subgroup. With E. Belmont, the author proved the detection theorem for this case and computed the $E_2$-page away from homological degree zero, using a conjectural form of the $C_9$-action on $E_6$ \cite{belmontray}; the geometric model of that action, by curves whose Jacobians carry the height $6$ formal group, is the third paper of this series \cite{higherlevels}. In September 2026, OpenAI posted a preprint claiming to resolve the Kervaire invariant problem at $p=3$ \cite{openaikervaire}, building on the results and the method of \cite{belmontray}; the author's account of the relationship between that preprint and \cite{belmontray} is at \cite{rinabout}. At the time of writing the preprint has not been refereed, and nothing in this series depends on it.

The third paper of this series, \emph{Only the Gauge Grows: The Lubin--Tate Action at Height $p^{k-1}(p-1)$ via Surgeries on Curves} \cite{higherlevels}, supplies part of the input that \cite{belmontray} took as conjectural. Its Theorem B \cite[Thm.\ B]{higherlevels}: as a representation of $G' = \Z/p^k \rtimes \Z/(p-1)^2$, the tangent space of Lubin--Tate space at height $\varphi(p^k)$ is the induced module $\Ind_{\Z/(p-1)^2}^{G'}\theta$ modulo the $(h-1)$-st power of its radical --- equivalently $(\fp_L/p\OO_L)^\vee$ for $L = \Q_p(\zeta_{p^k})$ --- and the same module is the tangent of a smooth stack $\fS_k$ of surgeries on the minimal $G'$-curve at the orbit $W_k(\F_p)$ of its branch points (the representation-theoretic half is reproduced in Section \ref{sec:lttangent} below). The identification of the two deformation \emph{rings}, as opposed to their tangents, is Conjecture I of that paper, verified on tangents by machine at $(p,k) = (3,2), (5,2)$; the ``conjectural form of the $C_9$-action'' used in \cite{belmontray} is the $(3,2)$ case.

\begin{center}
\renewcommand{\arraystretch}{1.3}
\begin{tabular}{@{}p{0.17\textwidth}p{0.37\textwidth}p{0.4\textwidth}@{}}
\toprule
 & \textbf{level $k=1$, $h = p-1$} \cite{sequel} & \textbf{level $k \geq 2$, $h = p^{k-1}(p-1)$} \cite{higherlevels} \\
\midrule
group & $G' = \Z/p \rtimes \Z/(p-1)^2$ & $G' = \Z/p^k \rtimes \Z/(p-1)^2$ \\
curve & $X\colon y^p - y = x^{p-1}$, forced by $G'$ & $X_k\colon \wp(y_0, \dots, y_{k-1}) = (z^{p-1}, 0, \dots)$, forced by $G'$ \\
moduli & configurations of the $p+1$ branch points of $X \to X/C_{p-1}$ & surgeries $\fS_k$ at the fixed branch orbit $W_k(\F_p)$ \\
tangent & $P(\theta)/\rad^{p-2}P(\theta) = \overline{\rho}$ & $P(\theta)/\rad^{h-1}P(\theta)$ \\
comparison & $\widehat{\mc F}^1\colon \Def_{(X,G)} \simeq \Def_F$, a theorem (Kodaira--Spencer) & period map $\Phi$ on horizontal surgeries; isomorphism on tangents is Theorem B, on rings is Conjecture I \\
output & $H^*(G', \pi_*E_{p-1})$ and its Tate cohomology & the $G'$-module $T\Def(H_h)$; Cartier chain at $k=2$ \\
\bottomrule
\end{tabular}
\end{center}

\subsection{Context as Preface and User's Guide}
Two lines of work run through this paper, and each is both a thread of a larger project and a collection of techniques.

\subsubsection*{Geometric Models, as Part of a Series}
This paper is the first in a trilogy stemming from the author's PhD thesis. The aim of the thesis is to construct arithmetically interesting and geometrically understandable stacks which model the Lubin--Tate action for all maximal finite subgroups $G$ and at all heights simultaneously (with a focus on those capturing $p$-torsion information). The construction of these stacks is guided by the ideology that the finite subgroup $G$ completely determines the construction of such a stack of decorated curves, and that such an approach lends itself to induction. 

This first paper in the series defines what it means to geometrically model the Lubin--Tate action (Definition \ref{def:geometricmodel}) and establishes the consequences of having such a model. Building such geometric models is the rest of the thesis: the second paper, \emph{Tiny and Unique} \cite{sequel}, builds one at height $p-1$, and the third, \emph{Only the Gauge Grows} \cite{higherlevels}, at the heights $p^{k-1}(p-1)$ above it.

\subsubsection*{Geometric Models, as a Distillation of Techniques}
We add toolkits to this prequel to highlight techniques and background which are useful in building geometric models: the deformation theory of curves with a group action, covers and their ramification, configuration spaces and the presentation of deformations as surgeries, and the representation-theoretic lemmas that turn a curve's cohomology into a group's. They are collected in Part \ref{part:toolkits}, stated once, in the generality the two later papers use, so that those papers can cite rather than redo. 

The techniques are mostly standard, each in its own home subject. What is new is the use they are put to. To our knowledge they have not been assembled before to build geometric models of the Lubin--Tate action, and building one asks more of them than their usual statements give: we need them equivariantly, in families, and with the grading kept. Several are restated here in that stronger form.

\subsubsection*{The Two-Tower Method, as a Companion Guide}
Faltings' isomorphism of the Lubin--Tate and Drinfeld towers \cite{faltings}, in the stacky form given by Scholze and Weinstein \cite{schollywein}, is the two-tower method's original instance and its reason for mattering: it is the geometric shadow of the Jacquet--Langlands correspondence relating $GL_h(\Q_p)$ and $D^\times$, and we recall it, with its characteristic-$p$ analogue over the stratified Lubin--Tate space, in Section \ref{sec:lubintateexamples}. 

This paper is meant as an introductory companion to the recent proof of the coheight one weak chromatic splitting conjecture, joint with Barthel, Mann, Schlank, Senger, Weinstein and Zhou \cite{chromsplit}. This runs on that characteristic-$p$ two-tower isomorphism. The profinite and condensed foundations of the first section --- smooth versus condensed actions, and coherent cohomology of quotient stacks as continuous cohomology --- are exactly what that application needs, and the two-tower criterion is stated here (Section \ref{sec:twotower}, Theorem \ref{thm:twotowerLT}) in the generality it uses. 

We also collect here the foundations such two-tower isomorphisms are read through: the pro-\'etale site and the constant group scheme $\underline{G}$, the comparison of smooth with condensed actions of a profinite group, and the passage from coherent cohomology of a quotient stack to continuous group cohomology of its structure sheaf --- with adic coefficients, or, when the coefficients are no longer adically topologized, with solid ones (Remark \ref{rem:solidneeded}).

The rest of this paper came about because, once we had the ``right language" in the finite case, it was clean to include the general profinite case and the two-tower case alongside the finite one to illuminate facets of comprehending the Lubin--Tate action as a coherent whole.  We have not yet seen a paper distill, relate, and collect these methods. Please enjoy.

\section*{Acknowledgements} 

Thanks to Sasha Shmakov, Lucas Piessevaux, Nathan Wenger, and Luozi Shi for reading over the draft. I would like to thank Grigory Kondyrev for encouraging me to write down the first version of this in 2018. Upon attending the 2025 Masterclass in Copenhagen on Arithmetic and Homotopy Theory and hearing the confusions of the audience, I was convinced this paper would be a helpful resource for the community, and decided to finally complete it and share it with the world. 

This paper is the first of 3 papers on expeditions constituting my PhD thesis, and I would like to thank Paul Goerss for advising me during graduate school and allowing me the great honour of being his last student. During that time I was supported as a fellow by the NSF GRFP under Grant Number DGE 1842165. As a postdoc in the Dynamics–Geometry–Structure group at Mathematics M\"unster, where I finished this paper, I was funded by the DFG under Germany's Excellence Strategy EXC 2044–390685587.

\newpage
\part{What It Means to Model an Action}
\partgloss{Profinite groups acting on deformation stacks; what a geometric model is, what it buys, and what it cannot be.}

\section{Ode to Profinite Group Actions on Deformation Stacks}

\subsection{Ode to Stacks}

\label{stacks} %(Much of this section is directly taken from Mann-Heyer.)

\begin{defn} We denote $\mr{ProFin} := \mr{Pro(Fin)}$ to be the Pro-category of the category of finite sets. This is an ordinary category whose objects are cofiltered diagrams $S = (S_i)_{i \in I}$ of finite sets $S_i$, and whose morphisms are $$\Hom_{\mr{ProFin}}((T_j)_{j \in J}, (S_i)_{i \in I}) = \lim_{i \in I} \underset{j \in J}\colim \Hom_{\mr{Fin}}(T_j, S_i).$$
\end{defn}

\begin{defn} We equip $\mr{ProFin}$ with the site where covers are those sieves $(\mr{ProFin})_{/S}$ that contain a finite family of maps $(S_n \to S)_n$ which are jointly surjective on the underlying sets. We denote by
$\mr{Cond(Ani)} := \mr{HypShv(ProFin)}$
the category of $\mr{Ani}$-valued hypersheaves on $\mr{ProFin}$.
\end{defn}

\begin{remark} We implicitly fix an uncountable strong limit cardinal $\kappa$.  \end{remark}

\begin{remark} This contains the category of ($\kappa$-small) locally compact Hausdorff spaces as a full subcategory, but also allows stacky phenomena, for example, one can form the classifying stack $\pt\stackyq G$ of a locally profinite group $G$. \end{remark}

\begin{defn} (1) Let $\mc{C}$ be an $\infty$-site and $\mc{D}$ be $\mr{Cond}(\mr{Ani})$, then a pre-stack on $\mc{C}$ with values in $\mc{D}$ is a pre-sheaf (i.e. a functor) $$\mc{X}: \mc{C}^{op} \to \mc{D}.$$

\noindent (2) A stack on $\mc{C}$ with values in $\mc{D}$ is a sheaf, i.e. $\mc{X}$ satisfies for every open covering family $\{ U_i \to U\}_{i \in I}$ that the following is a homotopy limit in $\mc{D}$.
\[\begin{tikzcd}[column sep=tiny]
	{\mc{X}(U)} & {\displaystyle\prod_{i \in I} \mc{X}(U_i)} & {\displaystyle\prod_{i_1, i_2 \in I} \mc{X}(U_{i_1} \times_U U_{i_2})} & {\displaystyle\prod_{i_1, i_2, i_3 \in I} \mc{X}(U_{i_1} \times_U U_{i_2} \times_U U_{i_3})} & \cdots
	\arrow[from=1-1, to=1-2]
	\arrow[shift left, from=1-2, to=1-3]
	\arrow[shift right, from=1-2, to=1-3]
	\arrow[shift right=2, from=1-3, to=1-4]
	\arrow[shift left=2, from=1-3, to=1-4]
	\arrow[from=1-3, to=1-4]
	\arrow[shift right, from=1-4, to=1-5]
	\arrow[shift left, from=1-4, to=1-5]
	\arrow[shift right=3, from=1-4, to=1-5]
	\arrow[shift left=3, from=1-4, to=1-5]
\end{tikzcd}\]

\end{defn}

\begin{defn} We say a morphism of stacks $f: X \to Y$, has property $P$ if for all affines $S \to Y$, the map $S \times_Y X \to S$ has property $P.$ \end{defn}

\begin{example} 
We will consider the following sites for a base scheme $S\in\mr{Sch}$:
\begin{itemize} 
\item The (small) \'etale site $S_{et}$ is the full subcategory of $\mr{Sch}_{/S}$ on \'etale morphisms of schemes $f:X\to S$ with covering families given by jointly surjective families $\{f_i:X_i\to S\}_{i\in I}$ of \'etale morphisms.
\item The (small) pro-\'etale site $S_{proet}$ is the full subcategory $\mr{Sch}_{/S}$ on pro-\'etale (or weakly \'etale) morphisms of schemes $f:X\to S$ with topology induced by the fpqc topology.
\item The (big) fppf site $\mr{Sch}_{fppf}/S=(\mr{Sch}/S)_{fppf}$ with covering families given by jointly surjective families $\{f_i:X_i\to S\}_{i\in I}$ of morphisms which are flat and locally of finite presentation (see \cite[Section 021L]{stacks-project}).
\item The (big) fpqc site $\mr{Sch}_{fpqc}/S=(\mr{Sch}/S)_{fpqc}$ with covering families given by jointly surjective families $\{f_i:X_i\to S\}_{i\in I}$ of morphisms which are faithfully flat and quasicompact (see \cite[Section 03NV]{stacks-project}). 
\item The (big) pro-\'etale site $\mr{Aff}_{S, \mr{pro\acute{e}t}}$ of affine schemes over $S$, with covering families given by finite jointly surjective families of weakly \'etale morphisms \cite{bhattscholze}.
\end{itemize}
\end{example}

\begin{conv} \label{conv:proet} From now on, \emph{stack} means a sheaf of anima (or groupoids) on $\mr{Aff}_{S, \mr{pro\acute{e}t}}$. Since fpqc covers by weakly \'etale maps are pro-\'etale covers, every fpqc sheaf (in particular every formal scheme) is a stack in this sense.\end{conv}

\begin{remark} \label{rem:whynotfppf} The fppf topology is too coarse for profinite groups. For an infinite profinite group $G$, the constant group scheme $\underline{G}$ (Construction \ref{constr:constant}) is flat but not of finite presentation, so $\underline{G} \to S$ is not an fppf cover, and $\underline{G}$-torsors need not be fppf-locally trivial. They are, however, pro-\'etale-locally trivial: every pro-\'etale cover of a w-contractible affine splits \cite[\S 2.4]{bhattscholze}. In particular, if $R$ is a local Artinian ring with separably closed residue field, $\Spec R$ is w-contractible, so every $\underline{G}$-torsor over it is trivial. \end{remark}

\begin{defn} (1) An affine formal algebraic space over $S$ is a sheaf $\mc{X}$ on $\mr{Aff}_{S, \mr{pro\acute{e}t}}$ which admits a description as an Ind-scheme $X \simeq \colim_i X_i$, where the $X_i$ are affine schemes and the transition morphisms are thickenings.

\noindent (2) A formal algebraic space over $S$ is a sheaf $\mc{X}$ on $\mr{Aff}_{S, \mr{pro\acute{e}t}}$ which receives a morphism $\amalg U_i \to \mc{X}$ which is representable, \'etale, and surjective, and whose source is a disjoint union of affine formal algebraic spaces $U_i$. 
\end{defn}

\begin{defn} Let $\mc{X}$ be a stack in groupoids on $\mr{Aff}_{S, \mr{pro\acute{e}t}}$. We say that $\mc{X}$ is a formal algebraic stack if it admits a representable pro-\'etale surjection $\mc{U} \to \mc{X}$ from a formal algebraic space $\mc{U}$.

\end{defn}

\begin{remark}
Every \'etale morphism is pro-\'etale, so a formal algebraic space is in particular a formal algebraic stack in this sense (take $\mc{U} = \amalg U_i$).
\end{remark}

\begin{remark} Emerton \cite{emerton2020formal} asks instead for an atlas which is representable by algebraic spaces, smooth and surjective. In other words, he works with the analog of Artin stacks, whereas we are working with the pro-\'etale topology. This lets us work with profinite stacks (locally quotients by a profinite group), which lie between DM stacks (locally quotients by a finite group) and Artin stacks.
\end{remark}

\begin{defn} 
Let $\mc{X}$ be a stack. We define $\mr{QCoh}: \mr{Stk}^{op} \to \mr{Cat}$ as the right Kan extension of the presheaf $\mr{Spec}(R) \mapsto \mr{Mod}_R$ along the inclusion $\mr{Aff} \hookrightarrow \mr{Stk}$, where $\mr{Stk}$ is the $\infty$-category of stacks. In other words: 
\begin{align*} 
\QCoh(\mc{X}) &\simeq \QCoh(\underset{\Spec A \to \mc{X}}\colim \Spec A) \\
& \simeq \lim_{\Spec A \to \mc{X}} \QCoh(\Spec A) \\
& \simeq \lim_{\Spec A \to \mc{X}} \Mod_A \\
\end{align*}

\end{defn}

By definition, a quasi-coherent sheaf $\mc{F}$ on a stack $\mc{X}$ amounts to the following data: 
\begin{itemize} 
\item For every $\Spec A \xrightarrow{x} \mc{X}$, the datum of an $A$-module $x^*(\mc{F})$, 
\item For every ring homomorphism $A \to B$ and every $y \colon \Spec B \to \mc{X}$ which factors as the composite $\Spec B \to \Spec A \xrightarrow{x} \mc{X}$,
\[
\begin{tikzcd}
\Spec B \arrow[r] \arrow[rr, "y"', bend right] & \Spec A \arrow[r, "x"] & \mc{X}
\end{tikzcd}
\]

then the datum of a $B$-module $y^*(\mc{F})$ with a prescribed isomorphism of $B$-modules $x^*(\mc{F}) \otimes_A B \simeq y^*(\mc{F}).$
\end{itemize}

\begin{example} For $\pt = \Spec k$, then we have $$\QCoh(\pt) = \Mod_k.$$ \noindent Our beloved global sections thus follow from a pushforward.

\begin{align*} 
p_*: \QCoh(X) &\to \QCoh(pt)\\
\mc{F} &\mapsto \Gamma(X, \mc{F}) = p_*(\mc{F}) 
\end{align*}
\end{example}

\begin{defn}
We define the functor $$\OO: \mr{Stk}^{op} \to \mr{CRing}$$ as the right Kan extension of $\OO(\Spec A) = A$, such that it preserves limits. 
\end{defn} 

\begin{remark} This is the decategorification of $\QCoh$. \end{remark}

\begin{example} Given $X \xrightarrow{p} \pt$, $\OO_X := p^*\OO_{\pt}.$ Also, $p_*\OO_X := \OO(X) := \Gamma(X, \OO_X).$
\end{example}

\begin{defn} We define the stacky quotient of a stack $X$ by a group $G$ as a colimit in the category of Stacks. Below, the two rightward arrows are the action map $(g, x) \mapsto g\cdot x$ and the projection $(g, x) \mapsto x.$ 
\[
\begin{tikzcd}
	X\stackyq G := & {\colim\Big(\cdots} & {G \times G \times X} & {G \times X} & {X \big)}
	\arrow[shift right, from=1-2, to=1-3]
	\arrow[shift left, from=1-2, to=1-3]
	\arrow[shift right=3, from=1-2, to=1-3]
	\arrow[shift left=3, from=1-2, to=1-3]
	\arrow[from=1-3, to=1-4]
	\arrow[shift left=3, from=1-3, to=1-4]
	\arrow[shift right=3, from=1-3, to=1-4]
	\arrow["{\text{act}}", shift left=2, from=1-4, to=1-5]
	\arrow["{\text{proj}}"', shift right, from=1-4, to=1-5]
\end{tikzcd}
\]
\noindent There is also a section $x \mapsto (e, x)$, we will work with the quotient stack as an action groupoid, and this section is our unit map.
\end{defn}

\begin{remark} Notational choice: in this paper a single slash $X\stackyq G$ always means the stacky quotient defined above, as is standard in parts of algebraic geometry; it agrees with the homotopy quotient $X_{hG}$, often written $X \stackyq G$ in homotopy theory. (The sequel reserves $X\stackyq G$ for the algebraic quotient of a curve and writes $X\stackyq G$ for the stacky one.) \end{remark}

\begin{construction} \label{constr:constant} (From condensed anima to stacks.) Fix a base ring $\Lambda$. For a profinite set $S = \lim_i S_i$, let $C(S, \Lambda) := \colim_i \Lambda^{S_i}$ be the ring of locally constant functions and set $$\underline{S} := \Spec C(S, \Lambda) = \lim_i \textstyle\coprod_{S_i} \Spec \Lambda.$$ The functor $\mr{ProFin} \to \mr{Aff}_{\Lambda}$, $S \mapsto \underline{S}$, preserves finite products and fiber products ($C(S \times_T S', \Lambda) = C(S, \Lambda) \otimes_{C(T, \Lambda)} C(S', \Lambda)$), and sends finite jointly surjective families to pro-\'etale covers. Hence it extends uniquely to a colimit-preserving functor $$\mr{Cond}(\mr{Ani}) \to \mr{Stk}_{\Lambda}, \qquad X \mapsto \underline{X}.$$ Since it commutes with colimits, it sends a profinite group $G$ to the constant group scheme $\underline{G}$, and the condensed classifying stack $\pt\stackyq G$ to the classifying stack $B\underline{G} = \pt\stackyq \underline{G}$ (the stacky quotient defined above). This is the algebraic analogue of the functor $S \mapsto S \times \Spa C^\flat$ of \cite[\S 3.5]{chromsplit}. 

Throughout, an action of a profinite group $G$ on a stack $X$ means an action of $\underline{G}$, and $X\stackyq G := X\stackyq \underline{G}$, formed in pro-\'etale stacks. \end{construction}

\begin{remark}  For $K$ a separably closed field, $\pi_0((X\stackyq G)(K))$ is exactly the set of orbits of $G$ acting on $X(K)$, since every $\underline{G}$-torsor over $\Spec K$ is trivial (Remark \ref{rem:whynotfppf}); over a general $K$ one must also account for nontrivial torsors. The same holds for $\Spec R$ with $R$ local Artinian with separably closed residue field. \end{remark}

\begin{remark} \label{rem:actioniscoaction} An action of $\underline{G}$ on an affine scheme $\Spec B$ is a coaction $B \to C(G, \Lambda) \otimes_\Lambda B = C(G, B)$, i.e., an action of $G$ on $B$ for which every element has open stabilizer. For a formal scheme $\Spf A = \colim_n \Spec A/\mf{m}^n$ with $\mf{m}$ $G$-stable, an action of $\underline{G}$ is therefore exactly a $G$-action on $A$ which is smooth on every $A/\mf{m}^n$, that is, a continuous action for the $\mf{m}$-adic topology. \end{remark}

\begin{defn} A functor $\mc{F}: \mc{C} \to \mathrm{An}$ is representable by an object $R \in \mc{C}$ if for every object $C \in \mc{C}$ we have a natural isomorphism $$\mc{F}(C) \simeq \Hom_\mc{C}(R, C).$$ \end{defn}

\begin{remark} By Yoneda, representable functors $\mc{F}, \mc{G}$ are equivalent if there is a natural transformation $n: \mc{F} \to \mc{G}$ which induces an isomorphism on their representing objects. \end{remark}

\subsection{Cohomology of Stacks gives Continuous Cohomology}

\label{continuouscohom}

We now give a site-theoretic description of continuous group cohomology, so that we may pass directly from an equivalence of stacks to an equivalence of the associated group cohomologies. 

There are two different notions of an action of a profinite group $G$; this subsection compares them, following \cite[Appendix A]{chromsplit}. 
\begin{itemize}
\item A \emph{smooth} action on a module $M$: every element has open stabilizer. Equivalently, $M$ is a comodule over the Hopf algebra $C(G, \Lambda) = \colim_U C(G/U, \Lambda)$ of locally constant functions. By Remark \ref{rem:actioniscoaction}, this is what an action of the group scheme $\underline{G}$ on an affine scheme sees.
\item A \emph{condensed} action on a condensed module $M$: a module over the condensed group ring $\Lambda[G]$. This sees the topology of $M$, e.g. the $\mf{m}$-adic topology on $A = W(k)[[u_1, \dots, u_{h-1}]]$, which is the topology for which the Lubin--Tate action is continuous.
\end{itemize}
Stacks naturally compute the first, while chromatic homotopy theory needs the second. We show they agree in the situations of this paper (Corollary \ref{continuousboogie}).

\subsubsection{Smooth actions and classifying stacks}

\begin{defn} Let $G$ be a profinite group and $\Lambda$ a commutative ring. We denote by $\mr{Rep}^{\mr{sm}}_\Lambda(G)$ the category of smooth $G$-representations on $\Lambda$-modules, i.e., $\Lambda$-modules $V$ with a $G$-action such that the stabilizer of every $v \in V$ is open in $G$. Equivalently, $\mr{Rep}^{\mr{sm}}_\Lambda(G)$ is the category of comodules over the commutative Hopf algebra $C(G, \Lambda)$.
\end{defn}

\begin{defn} \label{firsttry} For a profinite set $S$, let $D(S, \Lambda) := \Mod_{C(S, \Lambda)} \simeq \Shv(S, \Mod_\Lambda)$, the derived category of sheaves of $\Lambda$-modules on (the site of open subsets of) $S$. This satisfies descent along finite jointly surjective families, hence extends by right Kan extension to a hypercomplete sheaf of categories $D(-, \Lambda) \colon \mr{Cond}(\mr{Ani})^{op} \to \mr{Cat}$ \cite[Definition 3.5.17]{6fun}.
\end{defn}

\begin{lemma} \label{bitchitsequivalent} Let $G$ be a profinite group and $\Lambda$ a commutative ring. There are natural equivalences
$$\QCoh(B\underline{G}) \simeq D(\pt\stackyq G, \Lambda) \simeq \widehat{\mc{D}}(\mr{Rep}^{\mr{sm}}_\Lambda(G)),$$
where $B\underline{G}$ is the classifying stack over $\Spec \Lambda$ of Construction \ref{constr:constant}, and $\widehat{\mc{D}}$ denotes the left-completed derived category (which agrees with the usual derived category when $G$ has virtual finite cohomological dimension, e.g. for every compact $p$-adic Lie group \cite[Corollary 4.28]{CM21}). The equivalences are natural with respect to continuous homomorphisms, and the pullback along $B\underline{G} \to \pt$ sends a module to the trivial representation. For an object concentrated in a single degree, such as the structure sheaf, left completion is invisible: the reader may ignore the hat.
\end{lemma}

\begin{proof} By descent along the pro-\'etale cover $e \colon \Spec \Lambda \to B\underline{G}$, whose Čech nerve is $\underline{G^{\bullet}}$, we have $\QCoh(B\underline{G}) \simeq \mr{Tot}\, \QCoh(\underline{G^\bullet}) = \mr{Tot}\, \Mod_{C(G^\bullet, \Lambda)}$. The same Čech nerve computes $D(\pt\stackyq G, \Lambda) \simeq \mr{Tot}\, D(G^\bullet, \Lambda) = \mr{Tot}\, \Mod_{C(G^\bullet, \Lambda)}$, which gives the first equivalence. The second is \cite[Proposition 5.3.10]{6fun}; see also \cite[Proposition A.4.3]{chromsplit}, whose proof identifies $\mr{Tot}\, \Mod_{C(G^\bullet, \Lambda)}$ with comodules over $C(G, \Lambda)$ by Barr--Beck--Lurie. \end{proof}

\begin{remark} By \cite[Proposition A.4.3]{chromsplit}, the same category is equivalent to the left-completed category of sheaves on the site $\mc{T}_G$ of finite continuous $G$-sets. This is the language in which Devinatz--Hopkins formulate the continuous action of the Morava stabilizer group on $E_h$ \cite{DH04}. So the stacky quotients of this paper and the Devinatz--Hopkins formalism describe the same objects. \end{remark}

\begin{cor} \label{cor:ptG} Let $M \in \mr{Rep}^{\mr{sm}}_\Lambda(G)$, with corresponding $\mc{F}_M \in \QCoh(B\underline{G})$, and let $q: B\underline{G} \to \pt$. Then the pushforward computes smooth group cohomology:
$$R\Gamma_{\mr{sm}}(G, M) \simeq q_*\mc{F}_M =: R\Gamma(B\underline{G}, \mc{F}_M).$$
For discrete $M$, this is continuous group cohomology $R\Gamma_{cts}(G, M)$.
\end{cor}

\begin{proof} 
We compute $R\Gamma(B\underline{G}, \mc{F}_M)$ by descent along $e \colon \pt \to B\underline{G}$. The $(n+1)$-fold fiber product $\pt \times_{B\underline{G}} \cdots \times_{B\underline{G}} \pt$ is $\underline{G^{n}}$; as a profinite $G$-set it is $G^{n+1}$ with the diagonal action, and 
$$\mc{F}_M(\underline{G^{n}}) \simeq \Hom_{G}(G^{n+1}, M)^{\mr{lc}} \simeq C(G^{n}, M),$$
the locally constant maps. Since $\underline{G^n}$ is affine, $\mc{F}_M$ has no higher cohomology on it, so the descent spectral sequence collapses onto the complex with these terms. This is the standard (homogeneous) cochain complex computing $H^n_{\mr{sm}}(G, M)$, and the differentials are identified as well (both are alternating sums of the face maps). For discrete $M$, locally constant maps are exactly continuous maps.
\end{proof}

\begin{lemma} \label{howitbe} (site for sore eyes) Let $G$ be a profinite group acting on an affine scheme $X$ (i.e., $\underline{G}$ acts on $X$; Remark \ref{rem:actioniscoaction}), and let $\mc{F} \in \QCoh(X\stackyq G)$. Then $\mc{F}(X)$ is a smooth $G$-representation and
$$R\Gamma(X\stackyq G, \mc{F})\simeq R\Gamma_{\mr{sm}}(G, \mc{F}(X)), \qquad \text{so} \qquad H^*(X\stackyq G, \mc{F})\simeq H^*_{cts}(G, \mc{F}(X)),$$
with $\mc{F}(X)$ given the discrete topology.
\end{lemma}

\begin{proof} 
We consider the following square, in which the vertical maps are $\underline{G}$-torsors:
\[\begin{tikzcd}
	X & {\pt} \\
	{X\stackyq G} & {B\underline{G}} & \pt
	\arrow[from=1-1, to=1-2]
	\arrow[from=1-1, to=2-1]
	\arrow["\lrcorner"{anchor=center, pos=0.125}, draw=none, from=1-1, to=2-2]
	\arrow["e", from=1-2, to=2-2]
	\arrow["p"', from=2-1, to=2-2]
	\arrow["q"', from=2-2, to=2-3]
\end{tikzcd}\]
The square is cartesian, since $X\stackyq G \times_{B\underline{G}} \pt \simeq X$. Hence $p$ is affine: its base change along the cover $e$ is the affine map $X \to \pt$. Since quasi-coherent sheaves on affines have no higher cohomology and pushforward along $p$ commutes with flat base change along $e$, we get $Rp_*\mc{F} \simeq p_*\mc{F}$, and 
$$e^*p_*\mc{F} \simeq \Gamma(X, \mc{F}|_X) = \mc{F}(X),$$
with its $G$-action coming from descent data. That is, under Lemma \ref{bitchitsequivalent}, $p_*\mc{F}$ corresponds to the smooth representation $\mc{F}(X)$. Therefore
$$R\Gamma(X\stackyq G, \mc{F}) \simeq q_* Rp_* \mc{F} \simeq q_* p_*\mc{F} \simeq R\Gamma_{\mr{sm}}(G, \mc{F}(X)),$$
where the last equivalence is Corollary \ref{cor:ptG}.
\end{proof}

\subsubsection{Condensed actions and the comparison}

\begin{defn} \cite[\S A.1]{chromsplit} Let $\mc{D}$ be $\Mod_\Lambda(\mr{Sp}^{\mr{cond}})$ or the category of solid $\Lambda$-modules. The category $\mc{D}^{BG}$ of objects with \emph{condensed} $G$-action is the category of modules over the condensed group algebra $\Lambda[G]$; equivalently, of coalgebras for the comonad $C(G, -)$ given on sheaves by $C(G, M)(S) = M(G \times S)$. The \emph{continuous homotopy fixed points} are $$M^{hG} := \mr{Tot}\, C(G^{\bullet}, M),$$ the right adjoint of the trivial-action functor. 
For $M = \lim_n M_n$ a limit of discrete modules, e.g. $A = \lim_n A/\mf{m}^n$ with its $\mf{m}$-adic topology, $C(G^n, M)$ is the module of continuous maps $G^n \to M$, so $\pi_{-*} M^{hG} = H^*_{cts}(G, M)$ is continuous cohomology in the classical sense.
\end{defn}

\begin{prop} \label{prop:smoothcondensed} Let $G$ be a profinite group.
\begin{enumerate}
\item \cite[Proposition A.3.3]{chromsplit} For a discrete module $M$, smooth and condensed $G$-actions on $M$ are the same, and $M^{hG, \mr{sm}} \simeq M^{hG}$.
\item \cite[Proposition A.5.3, Corollary A.5.4]{chromsplit} If $G$ is a compact $p$-adic Lie group and $M$ is a $p$-adically complete solid module with smooth $G$-action, then the smooth and condensed homotopy fixed points agree: $M^{hG, \mr{sm}} \simeq M^{hG}$.
\end{enumerate}
\end{prop}

\begin{remark} Part (2) fails without $p$-completeness: for $G = \Z_p$ acting trivially on $\Q_p$, the smooth fixed points are $\Q_p$, while the condensed ones are $\Q_p \oplus \Q_p[-1]$ \cite[Remark A.5.7]{chromsplit}. \end{remark}

\begin{cor} \label{continuousboogie} \color{blue}{(continuous boogie) Let $G$ be a profinite group acting on $\Def_X^\star$, and suppose $\Def_X^\star \simeq \Spf A$ is pro-represented by a complete local ring $(A, \mf{m})$. (By Remark \ref{rem:actioniscoaction}, the action is then continuous for the $\mf{m}$-adic topology.) Let $A_{\square} := \lim_n (A/\mf{m}^n)^{\delta}$ be $A$ with its $\mf{m}$-adic condensed structure. Then
$$R\Gamma(\Def_X^\star\stackyq G, \OO_{\Def_X^\star\stackyq G}) \simeq \lim_n R\Gamma_{\mr{sm}}(G, A/\mf{m}^n) \simeq (A_\square)^{hG} \simeq R\Gamma_{cts}(G, A).$$ In particular, $$H^i(\Def_X^\star\stackyq G, \OO_{\Def_X^\star\stackyq G}) \simeq H^i_{\mr{cts}}(G, \OO(\Def_X^\star)).$$} \end{cor}

\begin{proof} Since $\mf{m}$ is $G$-stable, $\Spf A \simeq \colim_n \Spec(A/\mf{m}^n)$ as stacks with $G$-action. Quotients commute with colimits, so $\Def_X^\star\stackyq G \simeq \colim_n \Spec(A/\mf{m}^n)\stackyq G$, and $R\Gamma(-, \OO)$ turns this colimit into a limit:
$$R\Gamma(\Def_X^\star\stackyq G, \OO) \simeq \lim_n R\Gamma(\Spec(A/\mf{m}^n)\stackyq G, \OO) \simeq \lim_n R\Gamma_{\mr{sm}}(G, A/\mf{m}^n),$$
by Lemma \ref{howitbe}. Each $A/\mf{m}^n$ is discrete, so by Proposition \ref{prop:smoothcondensed}(1), $R\Gamma_{\mr{sm}}(G, A/\mf{m}^n) \simeq (A/\mf{m}^n)^{hG}$. Condensed homotopy fixed points are a right adjoint, hence commute with limits, giving $(A_\square)^{hG}$. Concretely, the cochain terms $C(G^k, A/\mf{m}^{n+1}) \to C(G^k, A/\mf{m}^n)$ are surjective (locally constant maps lift), so the limit is computed termwise by the continuous cochains $C(G^k, A)$. 
\end{proof}

\begin{remark} Corollary \ref{continuousboogie} needs no hypothesis on $G$ beyond profiniteness, because the coefficients are pro-discrete, so the comparison happens levelwise on discrete modules. Proposition \ref{prop:smoothcondensed}(2) is needed only for coefficients which are not limits of discrete smooth modules. \end{remark}

\begin{remark} \label{rem:solidneeded} The adic regime of Corollary \ref{continuousboogie} is the regime of deformation stacks, and it covers everything in this paper: $\Def^\star_F = \Spf \pi_0 E_h$ and its finite and profinite quotients. It does not cover transchromatic coefficients. For instance, $\pi_0 L_{K(h-1)} E_h$ is $I_{h-1}$-adically complete, but its residue field $k((u_{h-1}))$ carries the nontrivial $u_{h-1}$-adic topology, which is not $I_{h-1}$-adic. For coefficients such as the completed Gross--Torii field $\widehat{L}$ of \cite{chromsplit}, the relevant topology is likewise not the one a deformation stack sees. There one must work with solid coefficients and condensed actions from the outset, and Proposition \ref{prop:smoothcondensed}(2) is what compares the two notions \cite[\S 1.3, \S 3.5]{chromsplit}. \end{remark}

\section{Ode to Deformations}

Let $k$ be a perfect field of characteristic $p$. Let $\widehat{\Art}_k$ be the category of complete local $W(k)$-algebras $R$ with finitely generated maximal ideal and a specified isomorphism $R/\mf{m}_R \simeq k$. A functor on Artinian such rings is extended to $\widehat{\Art}_k$ by $\Def(R) := \lim_n \Def(R/\mf{m}_R^n)$. In the rest of this document, we denote the base change of an object $\mf{X} \times_{\Spec R} \Spec k$ as $\mf{X}|_k$.

We define a deformation moduli problem where we allow morphisms to reduce to a subset of automorphisms of $X$. 

\begin{defn} \label{deform} Let $\mc{M}$ be a prestack (a functor from rings to groupoids) and $X \in \mc{M}(k)$. Given a subgroup $G \subseteq \Aut(X)$, we consider the functor $\Def^G_X = \Def^{G, \mc{M}}_X: \widehat{\Art}_k \to \mathrm{Grpd}$, where the groupoid \color{blue}{$\Def_X^G(R)$} has
\begin{itemize} 
\item as objects tuples $$\{ \mf{X} \in \mc{M}(R), \iota: \mf{X}|_{k} \simeq X \},$$ 
\item as morphisms: maps $\phi: \mf{X} \to \mf{X}'$ such that there exists $g \in G$ for which the following diagram commutes:
\[
\begin{tikzcd}
\mf{X}|_k \arrow[d, "\iota"] \arrow[r, "\phi|_k"] & \mf{X}'|_k \arrow[d, "\iota'"] \\
X \arrow[r, "g", color={rgb,255:red,92;green,214;blue,214}]                                 & X                             
\end{tikzcd}
\] 
\end{itemize}\end{defn}

\begin{center}\includegraphics[width=4cm]{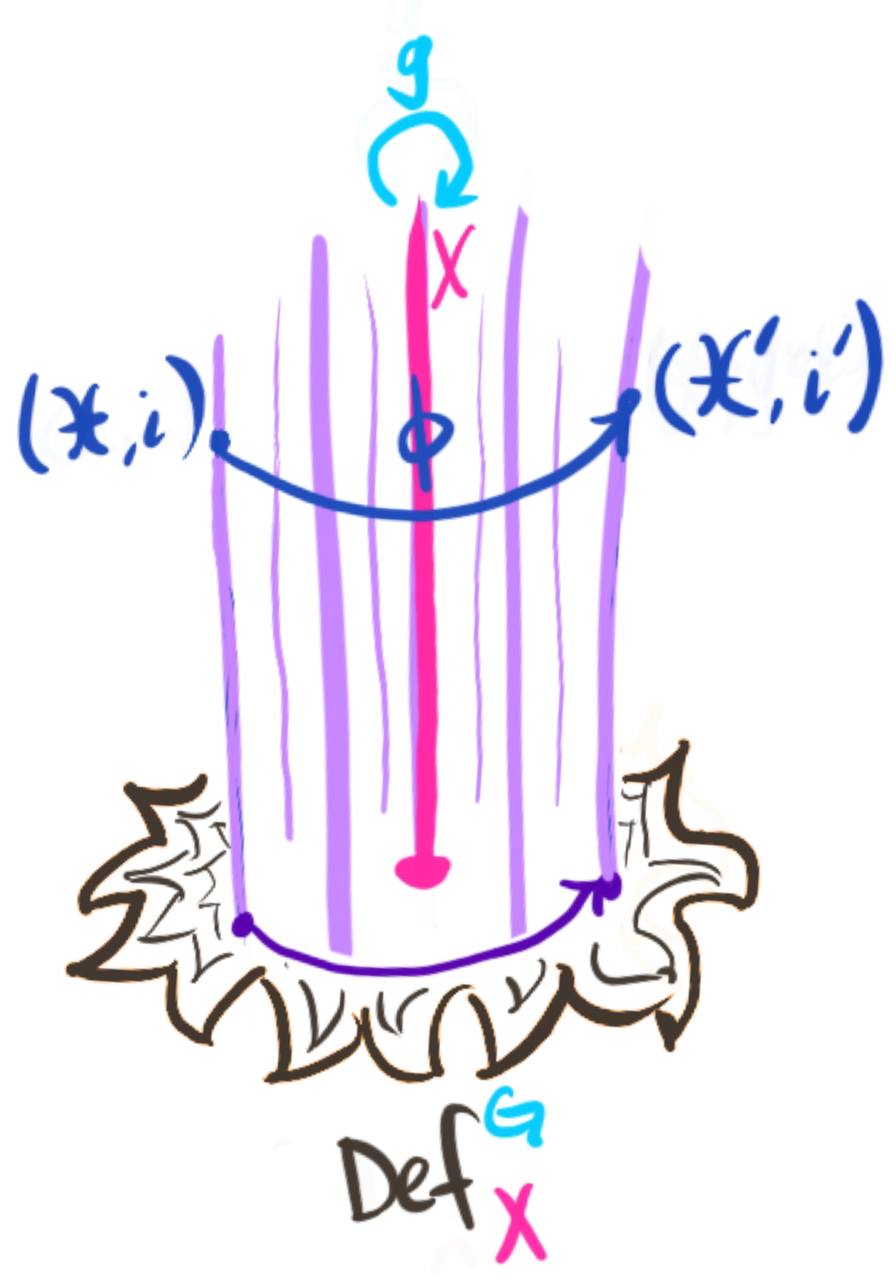}\end{center}

Historically, morphisms which reduce to the identity on the residue field are referred to as star-isomorphisms. As a notational convention, we will refer to the special case of $\Def_X^{\text{id}}$ as $\Def^\star_X$. 

\begin{defn} \label{autxaction} Let $G \subseteq \Aut(X)$ be a subgroup normalized by $g' \in \Aut(X)$ (e.g. any $g'$ when $G$ is normal, or any $G$ when $g' $ centralizes $G$). Then $g'$ acts on $\Def_{X}^{G}$ as follows: 

\begin{itemize} 
\item on objects, it sends $(\mf{X}, \mf{X}|_k \xrightarrow{\iota} X)$ to the object $(\mf{X}, {\color{red}{\mf{X}|_k \xrightarrow{\iota} X \xrightarrow{g'} X}})$,

\item on morphisms, it sends a morphism $\mf{X} \xrightarrow{\phi} \mf{X}'$ to itself; the diagram for $(\mf{X}, g'\iota) \to (\mf{X}', g'\iota')$ commutes with $g'gg'^{-1} \in G$ in place of $g$:
\[\begin{tikzcd}
	{\mf{X}|_k} & {\mf{X}'|_k} \\
	X & X
	\arrow["{\phi|_k}", from=1-1, to=1-2]
	\arrow["{g' \circ \iota}"', color=red, from=1-1, to=2-1]
	\arrow["{g' \circ \iota'}", color=red, from=1-2, to=2-2]
	\arrow["{g'gg'^{-1}}", from=2-1, to=2-2]
\end{tikzcd}\]
\end{itemize}
Writing $\Aut(X, G)$ for the normalizer of $G$ in $\Aut(X)$, this is an action of $\Aut(X, G)$ on $\Def^G_X$, through which $G$ acts trivially on isomorphism classes (see Lemma \ref{lem:Gtrivialclasses}). The two cases we use: $G = \mr{id}$, where all of $\Aut(X)$ acts on $\Def^\star_X$; and $G \subseteq Z(\Aut(X))$ central, where again all of $\Aut(X)$ acts, now on $\Def^G_X$.
\end{defn}

\begin{lemma} \label{lem:Gtrivialclasses} For $g \in G$ and $\xi = (\mf{X}, \iota) \in \Def^G_X(R)$, the identity of $\mf{X}$ is a morphism $\xi \to g \cdot \xi$ in $\Def^G_X(R)$. Hence $G \subseteq \Aut(X,G)$ acts trivially on $\pi_0 \Def^G_X$, and the action of $\Aut(X, G)$ on the pro-representing ring of $\Def^G_X$ (when it exists) factors through $\Aut(X, G)\stackyq G$. \end{lemma}
\begin{proof} The required diagram for $\phi = \mr{id}$ is $g \circ \iota = g \circ \iota$, with $g \in G$. \end{proof}

\begin{defn} A $G$-torsor $M \to M^{\star}$ is a pullback
\[\begin{tikzcd}
	M & \pt \\
	{M^\star} & BG
	\arrow[from=1-1, to=1-2]
	\arrow[from=1-1, to=2-1]
	\arrow["\lrcorner"{anchor=center, pos=0.125}, draw=none, from=1-1, to=2-2]
	\arrow[from=1-2, to=2-2]
	\arrow[from=2-1, to=2-2]
\end{tikzcd}\]
\end{defn}

\begin{lemma} \label{skydive} \color{blue}{(skydive)} Fix $G \subseteq \Aut(X)$, regarded as a discrete group. Then, as functors on $\widehat{\Art}_k$,
$$\Def^G_X \simeq (\Def_X^{\star})\stackyq G.$$  \end{lemma}

\begin{proof} 

Consider the map from 
\begin{align*} 
\Def_{X}^{\Aut_k(X)} &\to B\Aut_k(X) \\
(\mf{X} \xrightarrow{\phi} \mf{X}') & \mapsto  (X \xrightarrow{\phi|_k} X)
\end{align*}

The inclusion $G \hookrightarrow \Aut_k(X)$ induces a map on classifying stacks. The claim reduces to show that the pullback of these two maps is $\Def^G_X.$
\[\begin{tikzcd}
	{\Def_X^G} & BG \\
	{\Def_X^{\Aut_k(X)}} & {\mr{B}\Aut_k(X)}
	\arrow[from=1-1, to=1-2]
	\arrow[from=1-1, to=2-1]
	\arrow["\lrcorner"{anchor=center, pos=0.125}, draw=none, from=1-1, to=2-2]
	\arrow[from=1-2, to=2-2]
	\arrow[from=2-1, to=2-2]
\end{tikzcd}\]

Then, applying this to the groups $G$ and $\mr{id}$ respectively, implies that the following is one big pullback.
\[\begin{tikzcd}
	{\Def^\star_X} & {*} \\
	{\Def_X^G} & BG \\
	{\Def_X^{\Aut_k(X)}} & {\mr{B}\Aut_k(X)}
	\arrow[from=1-1, to=1-2]
	\arrow[from=1-1, to=2-1]
	\arrow["\lrcorner"{anchor=center, pos=0.125}, draw=none, from=1-1, to=2-2]
	\arrow["{\stackyq G}", from=1-2, to=2-2]
	\arrow[from=2-1, to=2-2]
	\arrow[from=2-1, to=3-1]
	\arrow["\lrcorner"{anchor=center, pos=0.125}, draw=none, from=2-1, to=3-2]
	\arrow[from=2-2, to=3-2]
	\arrow[from=3-1, to=3-2]
\end{tikzcd}\]

\noindent Since the upper square is a pullback of the $G$-torsor $\pt \to BG$, the map $\Def^\star_X \to \Def^G_X$ is a $G$-torsor, and the conclusion follows.

\end{proof}

\begin{remark} For finite groups, $\Def_X^G \simeq (\Def_X^\star)_{hG}$ (i.e., the homotopy colimit of the $G$ action), and the reader might be more comfortable seeing $(\Def_{X}^\star)_{hG}$. \end{remark}

\begin{remark}[The profinite case] \label{rem:defGprofinite} When $G$ is an infinite profinite group, Definition \ref{deform} \emph{cannot} remember the topology of $G$. A functor on local rings only ever sees $\underline{G}(R) = G$ as an abstract group, so it cannot distinguish $\underline{G}$ from $G$ with the discrete topology. The topology is only visible on non-local test objects such as $\underline{S} \times \Spec R$ for a profinite set $S$, and continuous and discrete cohomology differ in general (already for $G = \Z_p$). We therefore make the following definition. \end{remark}

\begin{defn} \label{def:defGprofinite} Let $G$ be a profinite group acting on $\Def^\star_X$, and assume $\Def^\star_X \simeq \Spf A$ is pro-representable with $\mf{m}_A$ $G$-stable and the action continuous for the $\mf{m}_A$-adic topology. By Remark \ref{rem:actioniscoaction}, this is an action of the constant group scheme $\underline{G}$ on the formal scheme $\Spf A$, viewed as a pro-\'etale stack over $\Spf W(k)$. We \emph{define} 
$$\Def^G_X := \Spf A / \underline{G},$$ 
the quotient in pro-\'etale stacks (Convention \ref{conv:proet}). \end{defn}

\begin{lemma} \label{lem:defGcompatible} In the situation of Definition \ref{def:defGprofinite}, let $R \in \widehat{\Art}_k$ with $k$ separably closed. Then $(\Spf A\stackyq \underline{G})(R)$ is equivalent to the groupoid $\Def^G_X(R)$ of Definition \ref{deform}. That is, the two definitions agree on all test objects of $\widehat{\Art}_k$, and Definition \ref{def:defGprofinite} additionally records the topology of $G$. \end{lemma}

\begin{proof} It suffices to treat Artinian $R$ and pass to the limit. Then $\Spec R$ is connected and w-contractible, so every $\underline{G}$-torsor over it is trivial (Remark \ref{rem:whynotfppf}), and $\underline{G}(R) = C(\Spec R, G) = G$. Hence $(\Spf A\stackyq \underline{G})(R)$ is the action groupoid of $G$ on $\Hom_{cts}(A, R) = \Def^\star_X(R)$. In $\Def^G_X(R)$ of Definition \ref{deform}, a morphism $(\mf{X}, \iota) \to (\mf{X}', \iota')$ is a $\phi$ with $\iota' \circ \phi|_k = g \circ \iota$ for some $g \in G$, i.e., a $\star$-isomorphism $(\mf{X}, g\circ\iota) \simeq (\mf{X}', \iota')$. Since $\Def^\star_X$ is discrete, such a $\phi$ is unique given $g$ when it exists, so these morphisms are exactly the morphisms $g$ of the action groupoid. \end{proof}

\begin{remark} With Definition \ref{def:defGprofinite}, the statement $\Def^G_X \simeq \Def^\star_X\stackyq G$ for profinite $G$ is a definition, and Lemma \ref{lem:defGcompatible} is its consistency check against Definition \ref{deform}. All cohomological statements below (Corollaries \ref{continuousboogie}, \ref{group robot time}) refer to Definition \ref{def:defGprofinite}. \end{remark}

\section{Geometric Modelling}
\label{sec:geomod}

In this section, we will explore equivalences of deformations of source and target of a functor between stacks.

\begin{defn} \label{def:inducedfunctor}
Given pre-stacks $\mc{M}$ and $\mc{N}$, consider a natural transformation $\mc{F}: \mc{M} \to \mc{N}$. Consider an object $X \in \mc{M}(k)$ and the corresponding object $\mc{F}(X)$ in $\mc{N}(k)$. Then $\mc{F}$ induces a homomorphism $\mc{F}_* \colon \Aut_k(X) \to \Aut_k(\mc{F}(X))$, and a functor $$\widetilde{\mc{F}}: \Def_X^{\star, \mc{M}} \to \Def_{\mc{F}(X)}^{\star, \mc{N}}, \qquad (\mf{X}, \iota) \mapsto (\mc{F}(\mf{X}), \mc{F}(\iota)).$$
\end{defn}

\begin{lemma} \label{lem:automaticequivariance} For every $g \in \Aut_k(X)$ we have $\widetilde{\mc{F}}(g \cdot \xi) = \mc{F}_*(g) \cdot \widetilde{\mc{F}}(\xi)$, on objects and on morphisms. In other words, $\widetilde{\mc{F}}$ is a morphism from the $\Aut_k(X)$-torsor $\Def^\star_X \to \Def^{\Aut(X)}_X$ to the $\Aut_k(\mc{F}(X))$-torsor $\Def^\star_{\mc{F}(X)} \to \Def^{\Aut(\mc{F}(X))}_{\mc{F}(X)}$ lying over the homomorphism $\mc{F}_*$. \end{lemma}

\begin{proof} On objects, $\widetilde{\mc{F}}(\mf{X}, g\circ\iota) = (\mc{F}(\mf{X}), \mc{F}(g)\circ\mc{F}(\iota))$ by functoriality of $\mc{F}$; both actions are the identity on morphisms (Definition \ref{autxaction}). \end{proof}

\begin{defn} \label{def:equivariant} Let $G \subseteq \Aut_k(X)$. We say $\widetilde{\mc{F}}$ is \emph{$G$-equivariant} if $\mc{F}_*$ is injective on $G$, so that $G$ is identified with its image $\mc{F}_*(G) \subseteq \Aut_k(\mc{F}(X))$ and $\widetilde{\mc{F}}$ intertwines the two actions of $G$ by Lemma \ref{lem:automaticequivariance}:
\[\begin{tikzcd}
	G & G \\
	{\Aut_k(X)} & {\Aut_k(\mc{F}(X))} \\
	{\Def^\star_X} & {\Def^\star_{\mc{F}(X)}}
	\arrow["\simeq", from=1-1, to=1-2]
	\arrow[hook, from=1-1, to=2-1]
	\arrow[hook', from=1-2, to=2-2]
	\arrow["{\mc{F}_*}", from=2-1, to=2-2]
	\arrow[from=3-1, to=3-1, loop, in=55, out=125, distance=10mm]
	\arrow["{\widetilde{\mc{F}}}", from=3-1, to=3-2]
	\arrow[from=3-2, to=3-2, loop, in=55, out=125, distance=10mm]
\end{tikzcd}\]

\end{defn}

\begin{remark} Equivariance in the sense of Lemma \ref{lem:automaticequivariance} is automatic; the content of Definition \ref{def:equivariant} is injectivity of $\mc{F}_*$ on $G$. If $K := \ker(\mc{F}_*) \cap G$ were nontrivial, the group acting on $\Def^\star_{\mc{F}(X)}$ would be the quotient $G/K$ rather than $G$. Note that injectivity cannot be read off the induced action on the representing rings, since elements acting trivially on isomorphism classes (e.g. central automorphisms) may still be nontrivial deck transformations. \end{remark}

\begin{defn} \label{def:geometricmodel} \label{geometricmodel} Let $G \subseteq \Aut_k(X)$. If $\widetilde{\mc{F}}$ is $G$-equivariant and an equivalence, we say that the tuple $$(X \in \mc{M}(k),\ G \subseteq \Aut_k(X),\ \widetilde{\mc{F}}: \Def_X^{\star, \mc{M}} \to \Def_{\mc{F}(X)}^{\star, \mc{N}})$$ is a {\color{blue}{geometric model}} for the action of $G$ on $\Def_{\mc{F}(X)}^{\star, \mc{N}}$. \end{defn}

Fabulous, we now have a way of factoring our potentially mysterious action of $G$ on $\Def^\star_{\mc{F}X}$ through a more understandable one, the action of $\Aut(X)$ on $\Def^\star_X.$ If we are greedier, we can ask for even more. 

\begin{lemma} \label{weakgreed} \label{criterion} \color{blue}{(greed; criterion theorem) If $\widetilde{\mc{F}} \colon \Def_X^{\star}\simeq \Def_{\mc{F}X}^\star$ is a geometric model for the action of $G$ (Definition \ref{def:geometricmodel}), then $\Def_X^G \simeq \Def_{\mc{F}X}^G$ are equivalent.}
\end{lemma}

\begin{proof} By Lemma \ref{lem:automaticequivariance} and injectivity of $\mc{F}_*$ on $G$, $\widetilde{\mc{F}}$ is an equivalence in the category of stacks with $G$-action. Taking the stacky quotient $(-)\stackyq G$ is a functor on that category, so $(\Def_X^{\star})\stackyq G \simeq (\Def_{\mc{F}X}^\star)\stackyq G$, and thus by Lemma \ref{skydive} $\Def_X^G \simeq \Def_{\mc{F}X}^G$. \end{proof}

We now wish to consider cohomology of prestacks $\widehat{\mathrm{Art}_k} \to \mathrm{Grpd}$. Here we specifically consider the coherent cohomology of a stack defined over a ringed site, as discussed in section \ref{continuouscohom}. 

\begin{cor}
    \label{robot time} \color{blue}{(robot time) If $\widetilde{\mc{F}} \colon \Def_X^{\star}\simeq \Def_{\mc{F}X}^\star$ is a geometric model for the action of $G$, then $$R\Gamma(\Def_X^G, \OO_{\Def_X^G}) \simeq R\Gamma(\Def_{\mc{F}X}^G, \OO_{\Def_{\mc{F}X}^G}).$$}
\end{cor} 

\begin{proof} By Lemma \ref{weakgreed} the two stacks are equivalent, and an equivalence of stacks induces an equivalence of their structure sheaves and hence of the coherent cohomology of the structure sheaves. \end{proof}

\begin{cor}
    \label{group robot time} \color{blue}{(group robot time) If $\widetilde{\mc{F}} \colon \Def_X^{\star}\simeq \Def_{\mc{F}X}^\star$ is a geometric model for the action of $G$, and $\Def_X^\star$ is pro-represented by a complete local ring with continuous $G$-action, then $$H^*_{cts}(G, \mc{O}({\Def_X^\star})) \simeq H_{cts}^*(G, \mc{O}(\Def_{\mc{F}X}^\star)).$$}
\end{cor} 

\begin{proof} By Lemma \ref{skydive}, $\Def^G_X \simeq \Def^\star_X\stackyq G$, and likewise for $\mc{F}X$; so this follows from Corollary \ref{robot time} and Corollary \ref{continuousboogie} applied on both sides. \end{proof}

\subsection{What a Geometric Model Cannot Be}
\label{sec:nogo}
Definition \ref{def:geometricmodel} asks for a $G$-equivariant equivalence of rigidified deformation problems. The following obstruction, from \cite{higherlevels}, says that at height $h = \varphi(p^k)$, $k \geq 2$, with $G' = \Z/p^k \rtimes \Z/(p-1)^2$, no geometric model can be a \emph{tower} over a model one level down: the level-$k$ tangent directions cannot be carried by wild-extension data over a level-$(k-1)$ curve. Here a \emph{level-$(k-1)$ model} is a geometric model of the action of $\Z/p^{k-1} \rtimes \Z/(p-1)^2$ on Lubin--Tate space at height $\varphi(p^{k-1})$, and ``fibred by wild-extension data'' means that the functor sends a curve to a $\Z/p$-cover of it. The ``socle'' is the subgroup $\langle \sigma^{p^{k-1}}\rangle$ of order $p$.

\begin{prop}[no tower]\label{prop:nogo}
For $k\ge2$, no functor fibred over a level-$(k-1)$ model by
wild-extension data is $G'$-equivariantly equivalent to $\Def(H_h)$;
and every $\sigma$-normalized map from a curve with a one-fixed-point
$C_{p^k}$-action to a base of genus $\le1$ has degree divisible by
$p^{k-1}$.
\end{prop}
\begin{proof}
Classify the relative tangent directions of such a functor. Cover
directions deform the wild extension datum; the deck group is intrinsic
and abelian, so the socle lies in the decorated group and acts
trivially on deformations. Flag directions deform a tame decoration on
a fixed curve, and Kodaira--Spencer kills them. Base directions: the
socle covers the identity, and an order-$p$ unipotent shear between
base and fibre forces $pn=0$ on a torsion-free tangent. On the target
the socle-nontrivial tangent has dimension $(h-1)-(p^{k-1}-1)=
p^{k-1}(p-2)>0$, and every relative direction fails one of the two
tests. For the second statement, $\mathrm{PGL}_2(\overline\F_p)$ has no
element of order $p^2$ and point stabilizers in genus $\le1$ have
$p$-exponent $p$, so $\bar\sigma^p$ acts fibrewise and freely off the
fixed point.
\end{proof}

\section{Two Tower Method}
\label{sec:twotower}

In the story above, we kept $G$ fixed and compared $\Def_{X}^G$ only to other $G$-torsors. We can also widen our scope of comparison by constructing and comparing it to $G'$-torsors, where $G$ and $G'$ are different. 

Let us now consider the situation where $N$, $N'$, and $M$ are prestacks, such that $M$ carries commuting actions of $G$ and $G'$, making it a $G$-torsor over $N$ and a $G'$-torsor over $N'$. 
\[\begin{tikzcd}
	& M \\
	N & {} & N'
	\arrow["{G \times G'}", from=1-2, to=1-2, loop, in=55, out=125, distance=10mm]
	\arrow["G"', from=1-2, to=2-1]
	\arrow["{{{G'}}}", from=1-2, to=2-3]
	\arrow["{{G'}}", from=2-1, to=2-1, loop, in=190, out=260, distance=10mm]
	\arrow["G", from=2-3, to=2-3, loop, in=280, out=350, distance=10mm]
\end{tikzcd}\]

\noindent then, we have an isomorphism of quotient stacks $$N \simeq M / G \qquad \text{and} \qquad N' \simeq M / G'.$$  Note that $N$ has a residual $G'$-action, and $N'$ has a residual $G$-action.

\begin{lemma} \label{spanny} Let us consider the same span of torsors, then we have an isomorphism of quotient stacks
$$N\stackyq G' \simeq N'\stackyq G$$
\end{lemma}

\begin{proof} This follows from 
\[\begin{tikzcd}
	& {G \stackyq  M\stackyq G'} \\
	{G \stackyq (M\stackyq G')} & {} & {(G \stackyq  M)\stackyq G'} \\
	{G \stackyq  N'} && {N\stackyq G'}
	\arrow["\simeq"', from=1-2, to=2-1]
	\arrow["\simeq", from=1-2, to=2-3]
	\arrow["\simeq"', from=2-1, to=3-1]
	\arrow["\simeq", from=2-3, to=3-3]
\end{tikzcd}\]
\end{proof}

\begin{prop}[Faltings, Fargues, Scholze--Weinstein] \label{prop:faltings} Let $F$ be the Honda formal group of height $h$ over $k = \bar{\F}_p$, $\mr{LT}_h := \Def^\star_F = \Spf W(k)[[u_1, \dots, u_{h-1}]]$, and $\mr{LT}^\diamond_{h,0}$ the diamond of its generic fiber. Let $\mr{LT}^\infty_{h,0}$ be the diamond classifying trivializations $(\Q_p\stackyq \Z_p)^h \simeq F^{un}[p^\infty]_\eta$ of the generic fiber of the universal deformation, and let $(\mc{H}^{h-1})^\diamond$ be the diamond of Drinfeld's upper half space over $\breve{\Q}_p$. Then
\[\begin{tikzcd}
	& {\mr{LT}^\infty_{h,0}} \\
	{\mr{LT}^\diamond_{h,0}} && {(\mc{H}^{h-1})^\diamond}
	\arrow["{GL_h(\Z_p)}"', from=1-2, to=2-1]
	\arrow["{\OO_D^\times = \G_h}", from=1-2, to=2-3]
\end{tikzcd}\]
is a span of pro-\'etale torsors with commuting actions, and therefore (Lemma \ref{spanny})
$$\big[\mr{LT}^\diamond_{h,0}\stackyq \G_h\big] \simeq \big[(\mc{H}^{h-1})^\diamond\stackyq GL_h(\Z_p)\big]$$
as $v$-stacks. Explicitly, $\mr{LT}^\infty_{h,0}(S)$ for a perfectoid space $S$ over $\breve{\Q}_p$ is the set of maps $\OO^{\oplus h} \to \OO(\tfrac1h)$ of vector bundles on the relative Fargues--Fontaine curve $X_S$ which are isomorphisms at the point $S^\sharp$, i.e., modifications of type $\mu = (1, 0, \dots, 0)$; the two torsor structures are the actions of $GL_h(\Q_p) = \Aut(\OO^{\oplus h})$ and $D^\times = \Aut(\OO(\tfrac1h))$, restricted to their maximal compact subgroups. \end{prop}

\begin{proof} The identification of $\mr{LT}^\infty_{h,0}$ with modifications of vector bundles on the Fargues--Fontaine curve, and the resulting duality with the Drinfeld tower, is \cite[Thm.\ 6.5.4, Thm.\ 7.2.3]{schollywein}; the original isomorphism of towers is due to Faltings \cite{faltings} and Fargues. That both maps are pro-\'etale torsors for the indicated groups is \cite[\S 6.3]{schollywein}. The commutation of the two actions is clear from the description. \end{proof}

\begin{remark} Taking $\OO^+$-cohomology on both sides of Proposition \ref{prop:faltings} and rationalizing, Barthel--Schlank--Stapleton--Weinstein \cite{rats} obtain
$$\Q \otimes H^*_{cts}(\G_h, \pi_0 E_h) \simeq H^*_{cts}(GL_h(\Z_p), \Q_p),$$
which is Lazard's computation; this is Corollary \ref{bi group time} with $N = (\mc{H}^{h-1})^\diamond$, up to the passage from $\OO$ to $\OO^+$ and rationalization, which is the content of that paper. The characteristic $p$ analogue, with $\mr{LT}_{h,t}$ in place of $\mr{LT}_{h,0}$, is \cite[Thm.\ H]{chromsplit}, stated in the introduction. \end{remark}

So, given $Y$, we have another approach to modelling $\Def_{Y}^G$. Rather than factoring the $G$-action through a more understandable $G$-action on $\Def_X$ for $X$ with a map to $Y$, which  involves finding $X$ and $\mc{F}$ such that $\mc{F}(X) = Y$, we can instead think about how to consider $\Def_{Y}^G$ as a $G'$-torsor, and compute with $\Def_{Y}^G\stackyq G'$ instead.

\begin{cor} \label{bi robot time} \color{blue}{(bi robot time) Given $N, N', M$ as above, then $$R\Gamma(N\stackyq G', \OO_{N\stackyq G'}) \simeq R\Gamma(N'\stackyq G, \OO_{N'\stackyq G}).$$}
\end{cor}

\begin{proof} This follows directly from Lemma \ref{spanny}: equivalent stacks have equivalent cohomology of their structure sheaves. \end{proof}

\begin{cor}  \label{bi group time} \color{blue}{(bi group time) Given $N, N', M$ as above, with $G, G'$ profinite and $N, N'$ affine (or pro-represented by complete local rings with continuous actions, as in Corollary \ref{continuousboogie}), then $$H^*_{cts}(G', \OO(N)) \simeq H^*_{cts}(G, \OO(N')).$$}
\end{cor}
\begin{proof} Apply Lemma \ref{howitbe} (or Corollary \ref{continuousboogie}) to both sides of Corollary \ref{bi robot time}.
\end{proof}

\newpage
\part{Especially the Lubin--Tate Action}
\partgloss{Formal groups and their deformations, the action of one on the other, the criterion that lets a curve acting on its deformations stand in.}

\section{Graded Formal Groups Painted As Flower Buds}

\begin{center}\includegraphics[width=7cm]{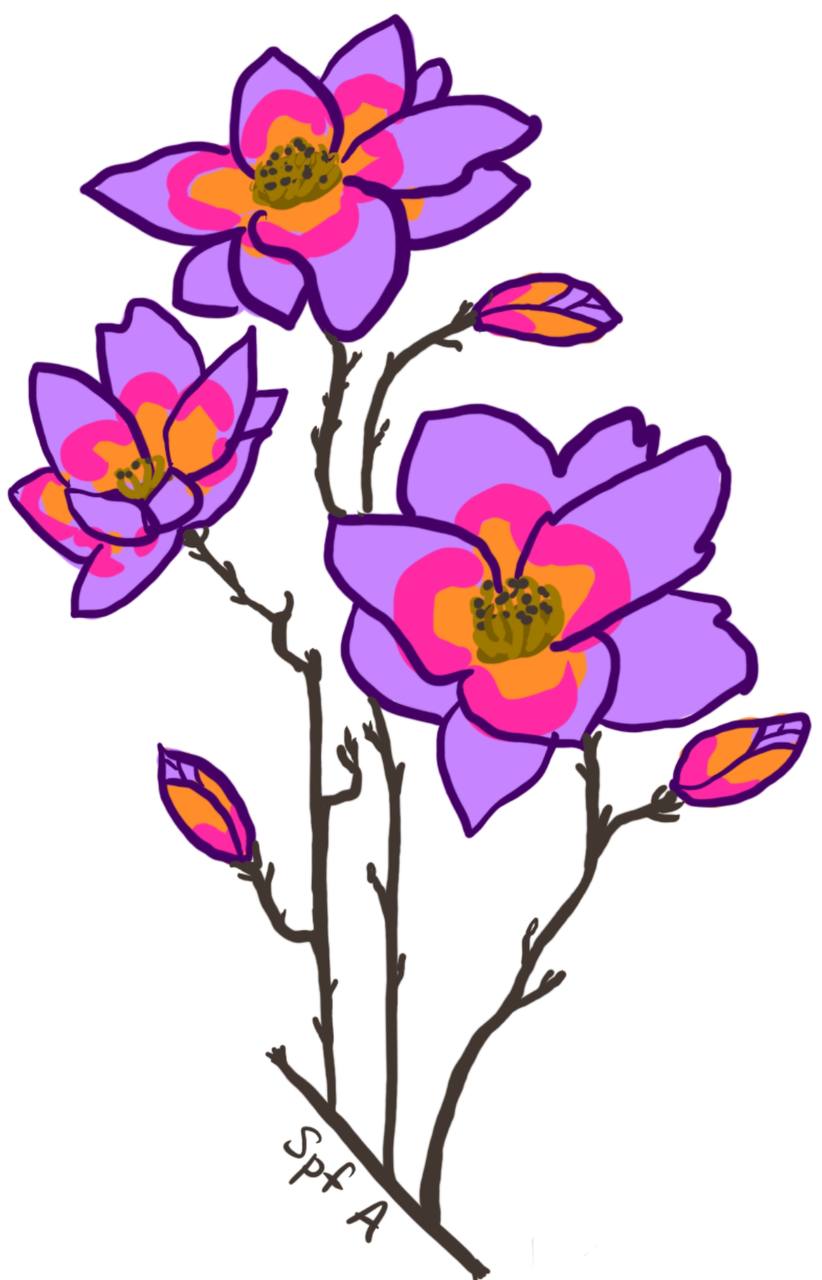}\end{center}

In this section, we define formal groups, formal group laws, and twisted versions of both. This ``twisting" is required to work with formal groups endowed with a natural grading which correspond to even periodic cohomology theories, as chern classes come equipped with a grading. We discuss how this twisting relates to choices of morphisms for the category of formal groups, and codify both in terms of the Lie algebra. 

We caution the reader that we discuss commutative formal group laws of \textit{all} dimensions, not just dimension one.

\begin{defn} 
Consider the category $\mr{CRing}^{top}_{R}$ of commutative $R$-algebras which are linearly topologized. The topology on $A$ is linear if there exists a fundamental system of neighborhoods of $0$ consisting of ideals.
\end{defn}

\begin{defn} An element $f$ of $A \in \mr{CRing}^{top}_{R}$ is called topologically nilpotent if $f^n \to 0$ as $n \to \infty$. We use the notation $A^{\circ \circ}$ to denote the ring of topologically nilpotent elements of $A$. \end{defn}

\begin{defn} A formal group over $R$ of dimension $n$ is a functor $$G: \mr{CRing}^{top}_R \to \mathrm{AbGrp}$$
such that its forgetful functor $$U(G): \mr{CRing}^{top}_R \to \mathrm{Set}$$ is Zariski locally in $R$ isomorphic to the functor which sends a ring to its topologically nilpotent elements $$A \mapsto (A^{\circ\circ})^n.$$ 
\end{defn}

\begin{remark} Note that $\Hom_{cts}(\Spf(A), \widehat{\A}^n_R) \simeq (A^{\circ\circ})^n$, in other words, Zariski locally a formal group has an isomorphism $U(G) \simeq  \widehat{\bb{A}}^n_R$. \end{remark} 

\begin{defn} A formal group law is a formal group $G$ together with a global isomorphism $\phi: U(G) \simeq \widehat{\bb{A}}^n_R$ of underlying set-valued functors. 
\end{defn} 

\begin{remark}  Note that being an isomorphism on underlying sets also guarantees that multiplication will be given by $$((b_1, ..., b_n), (c_1, ..., c_n)) \mapsto F((b_1, ..., b_n), (c_1, ..., c_n))$$ where $F \in \mr{Grp}(\widehat{\bb{A}}^n_R)$. \end{remark}

\begin{remark} Any abelian group structure on $\widehat{\A}^n_R$ as a sheaf over $\Spec(R)$ with 0 as a unit comes from a unique formal group law over $R$. \end{remark}

\begin{remark} Morphisms of formal groups are natural transformations of functors valued in abelian groups. In coordinates, morphisms of one-dimensional formal groups are locally of the form of a power series $f(t) := a_1t+a_2t^2+\cdots$ in $R[[t]]$ (an isomorphism exactly when $a_1 \in R^\times$), and in dimension $d$ they are of the form $f(t_1, ..., t_d) = (f_1(t_1, ..., t_d), \cdots, f_d(t_1, ..., t_d))$ with $f_i(0) = 0$.
\end{remark}

\subsection{Graded Formal Groups}

\begin{defn} Given $\mc{L}$ an invertible $R$-module, i.e. a map $[\mc{L}] : \Spec R \to B\G_m$, an $\mc{L}$-twisted formal group over $R$ of dimension $n$ is a functor $$G: \mr{CRing}^{top}_R \to \mathrm{AbGrp}$$
such that its forgetful functor to $\mathrm{Set}$ is Zariski locally in $R$ isomorphic to the functor which sends a ring to 
$$A \mapsto \Hom_R(\mc{L}, (A^{\circ\circ})^n).$$

\noindent For $n=1$ this is equivalent to putting an abelian group structure, with unit the zero section, on the formal scheme $\Spf \big(\widehat{\Sym}_R(\mc{L})\big) = \Spf \big(\prod_{i \geq 0} \mc{L}^{\otimes i}\big).$ 
\end{defn}

In the 1-dimensional case, this gives a formal group law of the form $\sum c_{ij} x^i y^j$ where $c_{ij} \in \mc{L}^{\otimes (i + j-1)}$. 

In topology, $x$ and $y$ are considered to be of degree $-2$ to reflect that they are first chern classes of line bundles. For example, in complex K-theory, the bott class $\beta = c_{11}$ in $KU$ a la Snaith.

\begin{defn} The Lie algebra of a formal group $G$ over $R$ is the kernel $$\mr{Lie}(G) = \ker \big(G(R[\epsilon]/\epsilon^2) \to G(R) \big).$$
In other words $\mr{Lie}(G)$ is the tangent space of $G$ along the identity section, with $R$-module structure induced by the scaling action $\epsilon \mapsto r\epsilon$. Since $G$ is commutative, its Lie bracket vanishes.
\end{defn} 

\begin{defn} 
The dualizing line $\omega_G$ of a formal group $G$ is the dual of the Lie algebra. Equivalently, it is the cotangent space at the identity section of the formal group.
\end{defn} 

\begin{remark} In the higher-dimensional case, we replace the dualizing lines $\omega_G$ by the cotangent space at the identity. \end{remark}

\begin{lemma} Given a formal group $G$ of dimension $n$, if its Lie algebra admits a trivialization $$f: \mr{Lie}(G) \simeq k^n,$$ then the formal group admits a trivialization, that is, there exists an isomorphism $$\widetilde{f}: U(G) \simeq \widehat{\bb{A}}^n_k.$$ 
\end{lemma}
\begin{proof} Zariski locally $U(G) \simeq \widehat{\A}^n$ by definition, so $\OO(U(G)) \simeq k[[x_1, \dots, x_n]]$ for some coordinates $x_i$, and $\mr{Lie}(G)^\vee = \mf{m}/\mf{m}^2$ has basis $dx_i$. Given $f$, let $y_1, \dots, y_n \in \mf{m}$ be lifts of the dual basis of $f^\vee$; the $y_i$ generate $\mf{m}$ modulo $\mf{m}^2$, hence (complete Nakayama) $k[[x]] = k[[y]]$, and $\widetilde{f} := (y_1, \dots, y_n)$ is the required coordinate. The lift is unique only modulo $\mf{m}^2$. See also \cite[\S 11]{finiteheight}, \cite[\S 1.1]{ell2}. \end{proof}

\begin{remark} Given a trivialization $f$ of the Lie algebra $\mr{Lie}(G)$ of a formal group $G$, there exists a lift $\widetilde{f}$ which gives a coordinate system for $U(G)$. The trivialization $f$ does \textit{not} uniquely determine $\widetilde{f}$. \end{remark}

\begin{defn} The moduli of one dimensional formal groups  $\mc{M}_{\mr{fg}_1}$ is the \'etale sheaf that associates to any ring $R$ the groupoid $G$ where $G \to \Spec R$ is a formal group. \end{defn}

\begin{defn} The moduli of one dimensional formal groups with trivialized Lie algebra $\mc{M}_{\mr{fg}_1}^{\mr{Lie} \simeq \mr{triv}}$ is the \'etale sheaf that associates to any ring $R$ the groupoid of pairs $(G, \phi)$ where $G \to \Spec R$ is a formal group and $\phi: \omega_G \simeq R$ is the trivialization of its sheaf of invariant differentials. \end{defn}

\begin{remark} The trivialization of a sheaf of invariant differentials is the same as a choice of globally non vanishing differential. \end{remark}

\begin{defn} \label{invariant} Let  $\G_{inv}: \mr{CRing} \to \mr{Grp}$ be the affine group scheme of invertible power series defined on points as $$\G_{inv}(R) := \big\{ \phi(x) := \sum_{i \geq 0} b_ix^{i+1} \in R[[x]] \text{ } | \text{ } b_0 \in R^\times \big\}.$$ \end{defn} 

Notice that $\G_{inv}$ admits a semi direct product decomposition as $\G_{inv} \simeq \G_{inv}^s \rtimes \G_m,$ where $\G_{inv}^s(R) := \big\{ \phi(x) := \sum_{i \geq 0} b_ix^{i+1} \in R[[x]] \text{ } | \text{ } b_0 = 1\}.$ When we consider a moduli stack of formal groups, we may either take general morphisms, $\G_{inv}$, or restrict ourselves to morphisms $h: G_0\to G_1$ that are the identity on the Lie algebra of the formal groups $\Lie(h) = \mr{id}: \mr{Lie}(G_0) \to \mr{Lie}(G_1)$ (i.e., $b_0 = 1$). These are also called strict morphisms, denoted above as $\G_{inv}^s$. 

\begin{lemma} \cite[p.\,55]{finiteheight} \label{trivvy}  If $G$ comes from a formal group law $F$, then there's a map
$$\mr{Fgl} \to \mc{M}_{\mr{fg}_1}^{\mr{Lie} \simeq \mr{triv}}.$$ \noindent This map is not $\G_{inv}$-invariant, but it is $\G^{s}_{inv}$-invariant, as isomorphisms of formal group laws do not have to preserve our chosen distinguished invariant differentials. The ones that do are the strictly invertible power series $\G^s$. 
\end{lemma}

\begin{lemma} \cite[p.\,55]{finiteheight}
There's a map $\mc{M}_{\mr{fg}_1}^{\mr{Lie} \simeq \mr{triv}} \to \mc{M}_{\mr{fg}_1}$ which is $\G_m$-invariant, because locally any two trivializations differ by an action of $\G_m$, 
$$\mc{M}_{\mr{fg}_1}^{\mr{Lie} \simeq \mr{triv}}\stackyq \G_m \simeq \mc{M}_{\mr{fg}_1}$$
\end{lemma}

Even though it is slightly evil, emboldened by Lemma \ref{trivvy} we will use the notation $\mc{M}_{\mr{fg}_1}^s$ for $\mc{M}_{\mr{fg}_1}^{\mr{Lie} \simeq \mr{triv}}$.

\begin{remark} Let $G$ be the formal group specified by a map $[G]: \Spec R \to \mc{M}_{\mr{fg}_1}$. The graded ring $\Spec (\bigoplus_{n \in \Z} \mc{L}^{\otimes n })$ has a natural interpretation as the coordinate ring of a principal $\G_m$-bundle corresponding to the Lie-algebra of $G$, which is also the universal scheme over which the latter admits a trivialization. 
\[\begin{tikzcd}
	{\Spec (\bigoplus_{n \in \Z} \mc{L}^{\otimes n })} & {\mc{M}_{\mr{fg}_1}^{\mr{Lie }\simeq \mr{triv}}} \\
	{\Spec(R)} & {\mc{M}_{\mr{fg}_1}}
	\arrow[from=1-1, to=1-2]
	\arrow[from=1-1, to=2-1]
	\arrow["{\stackyq \G_m}", from=1-2, to=2-2]
	\arrow["{[G]}", from=2-1, to=2-2]
\end{tikzcd}\]
\end{remark}

We conclude with an informal discussion on the role in topology of the graded element $\beta$ (coming from the Lie algebra). This will be revisited in Section \ref{representability}.

\begin{lemma} \label{lem:gradedungraded} Let $E_*$ be an evenly graded ring of the form $E_* \simeq E_0[\beta^{\pm 1}]$, where $\beta \in E_2$ (homotopical grading). Then the groupoid of graded formal groups over $E_*$ (i.e., $\G_m$-equivariant formal groups over $\Spec E_*$) is equivalent to the groupoid of ungraded formal groups over $E_0$. The equivalence is canonical once $\beta$ is fixed. \end{lemma}

\begin{proof} The grading is a $\G_m$-action on $\Spec E_*$, and the unit $\beta$ of degree $2$ gives a $\G_m$-equivariant isomorphism $\Spec E_* \simeq \Spec E_0 \times \G_m$. Hence $(\Spec E_*)\stackyq \G_m \simeq \Spec E_0$, and $\G_m$-equivariant formal groups on $\Spec E_*$ descend to formal groups on $\Spec E_0$. Since all elements involved sit in even degrees, no Koszul signs intervene. 
\end{proof}

\begin{remark}[other gradings] Hopkins, and Hill \cite{hillphd}, present $R_*$ on the generators $(uu_1,\dots,uu_{h-1},u^{\pm1})$, all in degree $-2$; the cohomology of $\mr{Conf}(\A^1,p)$ in Section \ref{sec:symmetriccohom}, with the standard representation in degree $-2$, which shows that $\Def_{(X,G)}\simeq\Def_{(\A^1,p)}$ uses this convention implicitly. Gorbounov--Mahowald \cite{gm} ignore the grading, act by a power series $g$ on $(u_1,\dots,u_{h-1})$ in degree $0$, and recover the action on $u$ separately as $g'(0)$ (this methodology is discussed in detail in the introduction of \cite{gorbounovsymonds}); \cite[Cor.\ C.2]{sequel} is that $g'(0)$ without a coordinate. Hovey--Strickland \cite{hoveystrickland} take a similar approach to Hopkins, working with $(vv_1, ..., vv_{h-1}),$ but they assume the grading of the variables comprising the graded ring representing the moduli problem is $|vv_i| = -(p^i-1)$ rather than $-2,$ which can be understood as them working with the cohomology of tensor powers of the line bundle defined by the universal differential $\omega^{\otimes p^i-1}$.
\end{remark}

\begin{remark} 
Here we take $E$ to be a complex orientable cohomology theory so that every complex line bundle $L$ on a space $X$ admits a first Chern class $c_1(L) \in E^2(X)$. If one considers a formal group law $c_1(L) +_F c_1(L')$, nonlinear terms such as $c_1(L)c_1(L')$ in the power series correspond to a cup product of first chern classes. Such a cup product which would take us straight out of $E^2(X)$ and into $E^4(X)$. However, we can \textit{maintain a consistent grading} by multiplying nonlinear factors by a class $\beta^{-1}$ in order to shift the degree to consider the power series entirely internal to $E^2(X).$  
\end{remark}

\subsection{Height of a Formal Group Law and Their Classification}

This section is devoted to the consideration of formal groups in characteristic $p$.

\begin{lemma} (Tate, Messing) Over a ring $R$ in which $p$ is nilpotent, the category of formal groups of finite height (those for which $[p]$ is an isogeny) is equivalent to the category of connected $p$-divisible groups over $R$, via $G \mapsto G[p^\infty]$. \end{lemma}

\begin{defn} Given $R$ a commutative ring in characteristic $p$, there is a map $\varphi_R: R \to R$ such that $\varphi_R(x) = x^p.$ For a commutative $R$-algebra $A,$ with structure map $R \xrightarrow{f} A,$ we denote $A^{1/p^h}$ as the corresponding $R$-algebra defined via the structure map $$R \xrightarrow{(\varphi_{R})^h} R \xrightarrow{f} A.$$ Given a functor $X$ with source category $\mr{CAlg}_R$, we define $X^{(p^h)}(A) := X(A^{1/p^h}).$
There is a natural map $\varphi_{X/R}^h: X \to X^{(p^h)}$ called the \textbf{relative Frobenius map}.
\end{defn}

We now introduce a key property of formal groups over characteristic $p$ fields. 

\begin{lemma} \cite[Prop. 4.4.5]{ell2} Given a map $f: G \to G'$ in $\mc{M}_{\mr{fg}}(R)$, the following conditions are equivalent: 
\begin{itemize}
\item The pullback map $f^* \colon \omega_{G'} \to \omega_{G}$ vanishes.
\item The morphism $f$ factors as a composition $G \xrightarrow{\phi_G} G^{(p)} \xrightarrow{g} G'$.
\end{itemize}
If these conditions are satisfied, the map $g$ is uniquely determined.
\end{lemma}

\begin{defn}
A formal group $F$ over an $\F_p$-algebra is of height at least $h$ if the multiplication by $p$ map factors through the $h$-th relative Frobenius, as in
\[\begin{tikzcd}
	F & {F^{(p^h)}} \\
	& F
	\arrow["{\varphi^h_{F/R}}", from=1-1, to=1-2]
	\arrow["{[p]_F}"', from=1-1, to=2-2]
	\arrow[dashed, from=1-2, to=2-2]
\end{tikzcd}\]

\noindent A formal group $F$ is of height exactly $h$ if the map factoring Frobenius is an isomorphism.
\[\begin{tikzcd}
	F & {F^{(p^h)}} \\
	& F
	\arrow["{{\varphi^h_{F/R}}}", from=1-1, to=1-2]
	\arrow["{[p]_F}"', from=1-1, to=2-2]
	\arrow["\simeq", from=1-2, to=2-2]
\end{tikzcd}\] 
\end{defn}

\begin{remark} Equivalently, a one-dimensional formal group $F$ over a field $k$ of characteristic $p$ has height $h$ if the kernel $F[p]$ of multiplication by $p$ is a finite (connected) group scheme of order $p^h$, i.e., 
$\text{rank}_k \OO(F[p]) = p^{h}.$
\end{remark}

\subsection{Isomorphism Scheme is Ind-Etale and Height Classifies}
\begin{defn} An $R$-algebra $S$ is ind-\'etale if it is isomorphic to a filtered colimit of \'etale $R$-algebras. Every ind-\'etale map $\Spec S \to \Spec R$ is formally \'etale: for every diagram as below there is a unique lift of $q$
\[\begin{tikzcd}
	{\Spec k} & {\Spec S} \\
	{\Spec A} & {\Spec R}
	\arrow[from=1-1, to=1-2]
	\arrow[from=1-1, to=2-1]
	\arrow["\begin{array}{c} \text{ind-}\\ \text{etale} \end{array}", from=1-2, to=2-2]
	\arrow["{{\exists !}}"{description}, dashed, from=2-1, to=1-2]
	\arrow["q", from=2-1, to=2-2]
\end{tikzcd}\]
where $\Spec k \to \Spec A$ is a nil-thickening.
\end{defn}

\begin{remark} A map is \'etale if and only if it is formally \'etale and locally of finite presentation. The unique lifting property is the only consequence of ind-\'etaleness we use, in Lemma \ref{lem:constantaut}. \end{remark}

\begin{defn} Given $G_0 \to \Spec(R_0)$ and $G_1 \to \Spec(R_1)$ are formal groups, we have an isomorphism scheme which fits into a pullback diagram 
\[\begin{tikzcd}
	{\mathrm{Iso}(G_0, G_1)} & {\Spec(R_0)} \\
	{\Spec(R_1)} & {\mc{M}_{\text{fg}}}
	\arrow[from=1-1, to=1-2]
	\arrow[from=1-1, to=2-1]
	\arrow["\lrcorner"{anchor=center, pos=0.125}, draw=none, from=1-1, to=2-2]
	\arrow[from=1-2, to=2-2]
	\arrow[from=2-1, to=2-2]
\end{tikzcd}\]

The $S$-points of $\mathrm{Iso}(G_0, G_1)$ are given by triples consisting of maps $f_i: R_i \to S$ together with an isomorphism $f_0^*G_0 \simeq f_1^*G_1$ of formal groups. 
\end{defn}

If the formal groups come from formal group laws $F_0, F_1$, the resulting scheme is affine, $\mathrm{Iso}(G_0, G_1) \simeq \Spec(A_{F_0, F_1}).$ This is the $R_0 \otimes_\Z R_1$ algebra generated by symbols $b_i$ for $i \geq 0$ subject to the relations which state that the power series $\phi(x) = \sum_i b_ix^{i+1}$ is an isomorphism from $F_0$ to $F_1$.  We introduce the notation of $A_{F_0, F_1}(m)$ to mean the $R_0 \otimes R_1$-subalgebra of $A_{F_0, F_1}$ generated by $b_i$ for $i < m.$

\begin{lemma} \label{indetale} \cite[15.2]{finiteheight} Let $F_0$ and $F_1$ be formal group laws of dimension 1 over $\F_p$-algebras $R_0$ and $R_1$, both of height exactly $h > 0$. Then 
\begin{enumerate} 
\item $A_{F_0, F_1}(0) \simeq R_0 \otimes_{\Z} R_1$
\item each of the maps $A_{F_0, F_1}(m) \hookrightarrow A_{F_0, F_1}(m+1)$ is finite etale.
\end{enumerate} 
In particular $A_{F_0, F_1}$ is ind-etale over $R_0 \otimes R_1$.
\end{lemma}

\begin{lemma} \cite[15.6]{finiteheight}
Let $F_0, F_0', F_1$ be formal group laws over $R,$ then any choice of isomorphism $\phi \colon F_0 \simeq F_0'$ induces an isomorphism of $R$ algebras $A_{F_0, F_1} \simeq A_{F_0', F_1}$ compatible with the filtration. 
\end{lemma}

\begin{remark} Dually, the filtration by $m$ exhibits
$$\mr{Iso}(F_0, F_1)(\bar{k}) = \Hom(A_{F_0, F_1}, \bar{k}) = \lim_m \Hom(A_{F_0, F_1}(m), \bar{k})$$
as a limit of finite sets, giving it the structure of a profinite set, and $\Aut_{\bar{k}}(F)$ the structure of a profinite group. \end{remark}

\begin{theorem}  \cite[15.4]{finiteheight}(Lazard)  Let $K$ be an algebraically closed field of characteristic $p$. Then any two formal groups $F_0, F_1$ of dimension 1 over $K$ of the same height are isomorphic. \end{theorem}

\begin{question} Are iso-schemes for formal group laws of dimension $n$ still ind-\'etale? \end{question}

\begin{answer} For $n = 1$, yes: Lemma \ref{indetale}, which is the case we use throughout. We do not treat $n > 1$ here; nothing below needs it. \end{answer}

\subsection{Automorphisms of a Formal Group}

In this section we establish that the group of interest to us is a constant profinite group scheme, this allows us to freely apply the machinery we developed in the stack section to our case.

\begin{defn} Given a one-dimensional formal group $F$ of height $h < \infty$ over $k$, let $\underline{\Aut}(F)$ be the group scheme over $k$ whose $S$-points are the automorphisms of $F \otimes_k S$; it is the base change of $\mr{Iso}(F, F)$ along the diagonal $\Spec k \to \Spec(k \otimes k)$. \end{defn}

\begin{lemma} \label{lem:constantaut} If $k$ is separably closed, then $\underline{\Aut}(F)$ is the constant profinite group scheme associated to the profinite group $\Aut_k(F)$. In particular, for every $k$-algebra $R$ which is local Artinian with residue field $k$, restriction gives a bijection $\Aut_R(F \otimes_k R) \simeq \Aut_k(F)$. \end{lemma}

\begin{proof} By Lemma \ref{indetale}, $\underline{\Aut}(F) = \Spec B$ with $B = A_{F,F} \otimes_{k\otimes k} k$ a filtered colimit of finite \'etale $k$-algebras $B(m)$. As $k$ is separably closed, each $B(m)$ is a finite product of copies of $k$, so $\Spec B$ is the profinite set $\lim_m \Hom_k(B(m), k)$, viewed as a constant pro-finite scheme. For the second claim, $\Spec k \to \Spec R$ is a nilpotent thickening, so by the unique lifting property of the ind-\'etale map $\Spec B \to \Spec k$,
\[\begin{tikzcd}
	{\Spec k} & {\underline{\Aut}(F)} \\
	{\Spec R} & {\Spec k}
	\arrow[from=1-1, to=1-2]
	\arrow[hook, from=1-1, to=2-1]
	\arrow["\begin{array}{c} \substack{ \text{ind}- \\ \text{\'etale}} \end{array}", from=1-2, to=2-2]
	\arrow["{\exists !}"{description}, dashed, from=2-1, to=1-2]
	\arrow[from=2-1, to=2-2]
\end{tikzcd}\]
every automorphism of $F$ over $k$ extends uniquely to $F \otimes_k R$. \end{proof}

\begin{lemma} \label{lem:rigidity} (rigidity) Let $R \in \widehat{\Art}_k$ and let $\widetilde{F}_1, \widetilde{F}_2$ be deformations of $F$ over $R$. Then reduction to the residue field $$\mr{Iso}_R(\widetilde{F}_1, \widetilde{F}_2) \to \mr{Iso}_k(\widetilde{F}_1|_k, \widetilde{F}_2|_k)$$ is injective. \end{lemma}

\begin{proof} It suffices to treat Artinian $R$, in which $p$ is nilpotent. Then $\widetilde{F}_i$ are connected $p$-divisible groups (the height is finite), and homomorphisms of $p$-divisible groups over a ring in which $p$ is nilpotent are determined by their reduction modulo a nilpotent ideal \cite[Lemma 1.1.3]{katz}. \end{proof}

\begin{remark} Lemma \ref{lem:rigidity} does \emph{not} say that every automorphism of $F$ lifts to a given deformation $\widetilde{F}$: it does not, and the failure to lift is exactly what makes the action of $\Aut_k(F)$ on $\Def^\star_F$ nontrivial. Lemma \ref{lem:constantaut} concerns the constant deformation $F \otimes_k R$ only. \end{remark}

\begin{cor} Given a formal group $F$ of dimension one and finite height, $\Def_F^\star \simeq \pi_0(\Def_F^\star)$. That is, $\Def_F^\star$ is discrete. \end{cor}

\begin{proof} An automorphism of an object $(\widetilde{F}, \iota)$ of $\Def_F^\star(R)$ is an automorphism of $\widetilde{F}$ reducing to the identity; by Lemma \ref{lem:rigidity} it is the identity. \end{proof}

\begin{cor} \label{cor:autacts} The profinite group $\Aut_k(F)$ acts on $\Def_F^{\star} \simeq \Spf A$ (Definition \ref{autxaction}), continuously for the $\mf{m}_A$-adic topology \cite{LT, DH04}. Hence this is an action of the constant profinite group scheme $\underline{\Aut_k(F)}$ (Remark \ref{rem:actioniscoaction}), so Definition \ref{def:defGprofinite} and Corollary \ref{continuousboogie} apply to every closed subgroup $G \subseteq \Aut_k(F)$. \end{cor}

\begin{remark} We needed the group to be constant in order to form the quotient $\Def_F^\star\stackyq \Aut(F)$ in the sense of Section \ref{stacks}. If $k$ is not separably closed, one may instead work with the group scheme $\underline{\Aut}(F)$, or first base change to $k^{sep}$; for $k \supseteq \F_{p^h}$ and $F$ the Honda formal group, all automorphisms over $\bar{k}$ are already defined over $k$. \end{remark}

\subsection{Representability of Deformations of Formal Groups}
\label{representability}
Given a characteristic $p$ field $k$, we consider the deformations of a $k$-point in the prestack of ungraded formal groups of dimension one $\mc{M}_{\mr{fg}_1}$ (ungraded case), and the prestack of formal groups with $\mc{M}_{\mr{fg}_1}^{s}$ (graded case). This section establishes the co-representability of the deformation moduli problems of graded and ungraded formal groups $\Def_F^\star$ and $\Def_{F'}^\star$ in the sense of Definition \ref{deform}. Let $W := W(k)$ denote the Witt vectors of $k$.

\begin{lemma} \label{reppingbaby} $\text{ }$
\begin{itemize} 
\item Given $F \in \mc{M}_{\mr{fg}_1}^{=h}(k)$, then $\Def_{F}^\star$ is pro-represented by a complete local ring $A$ which is noncanonically isomorphic to $W[[u_1, ..., u_{h-1}]]$ \cite{LT}. In other words, there's an isomorphism of groupoids (which are discrete) $$\Hom_{cts, W}(A, R) \simeq \Def_{F}^\star(R).$$
\item Given $F' \in \mc{M}_{\mr{fg}_1}^{=h,s}(k)$, the graded deformation problem $\Def_{F'}^\star$ --- the deformation problem of Definition \ref{deform} for the pair $(F, \phi_0) \in \mc{M}^{s}_{\mr{fg}_1}(k)$ with morphisms allowed to reduce to the subgroup $\G_m(k) \subseteq \Aut(F, \phi_0)$ scaling $\phi_0$, i.e., $\Def^{\G_m}_{(F, \phi_0)}$ in the notation of Definition \ref{deform} --- is the $\G_m$-torsor over $\Spf A$ of trivializations of $\omega$ of the universal deformation. It is represented by the formal scheme $\Spf(A) \times \G_m$, with graded ring of functions $B$ noncanonically isomorphic to $W[[u_1, ..., u_{h-1}]][\beta^{\pm 1}]$, $|\beta| = 2$.
\end{itemize} 
\end{lemma}

\begin{construction}
Let $G$ be a formal group with height $\geq h$, then $[p]$ factors as a composition 
$G \xrightarrow{\phi_G} G^{(p^h)} \xrightarrow{T} G$. Since $T$ is uniquely determined, it therefore induces a pullback map $$T^* \colon \omega_G \to \omega_{G^{(p^h)}} \simeq \omega^{\otimes p^h}_G$$ \noindent which we can identify with an element $v_h \in \omega_{G}^{\otimes (p^h-1)}.$ This is often called the Hasse invariant.
\end{construction}

We now remark on the special case of Morava $E$-theory, and encourage the reader to take a look Example 5.3.7 and Section 3.3 of \cite{ell2} for context and revelation.

\begin{lemma} \cite[Cor. 5.4.2]{ell2}  Suppose that there exists an element $\beta \in \pi_2(E)$ which is invertible in $\pi_*(E)$. Pick elements $\overline{v}_m \in \pi_{2(p^m-1)}(E)$ representing the Hasse invariants $v_m \in \pi_{2(p^m-1)}(E)/I_m$ and set $u_m = \overline{v}_m/\beta^{p^m-1} \in \pi_0(E)$. Then we have noncanonical isomorphisms  
$$\pi_0(E) \simeq W(k)[[u_1, ..., u_{h-1}]] \qquad \pi_*(E) \simeq W(k)[[u_1, ..., u_{h-1}]][\beta^{\pm}].$$ \end{lemma} 

\begin{remark} For the formal group $G$ of a complex orientable $E$, there is a canonical isomorphism $\omega_G \simeq \pi_2(E)$ (more generally $\omega_G^{\otimes n} \simeq \pi_{2n}(E)$ for even periodic $E$). An invertible $\beta \in \pi_2(E)$ is therefore exactly a trivialization of $\omega_G$, and we can identify the tensor powers of $\omega_G$ with powers of $\beta$. \end{remark}

\begin{cor} $\text{ }$ \color{blue}{  
\begin{itemize} 
\item Given $F \in \mc{M}_{\mr{fg}_1}^{=h}(k)$, and fixing a closed subgroup $G \subseteq \Aut(F)$, then $\Def_{F}^G \simeq \Spf(A)\stackyq G$ with $A \simeq W[[u_1, ..., u_{h-1}]]$ noncanonically; in particular $$H^0(\Def^G_F, \OO) \simeq A^G, \qquad H^*(\Def^G_F, \OO) \simeq H^*_{cts}(G, A).$$
\item Given $F' \in \mc{M}_{\mr{fg}_1}^{=h, s}(k)$, and fixing a closed subgroup $G \subseteq \Aut(F')$, then $\Def_{F'}^G \simeq (\Spf(A) \times \G_m)\stackyq G$, and $$H^*(\Def^G_{F'}, \OO) \simeq H^*_{cts}(G, B), \qquad B \simeq W[[u_1, ..., u_{h-1}]][\beta^{\pm 1}].$$
\end{itemize}}
\end{cor}

\begin{proof} By Lemma \ref{skydive}, $\Def_{F}^G \simeq \Def_F^\star\stackyq G$, and $\Def^\star_F \simeq \Spf A$ by Lemma \ref{reppingbaby}. The cohomology statement is Corollary \ref{continuousboogie}. The same applies to $F'$. Note that $\Def^G_F$ is a quotient stack and is not itself (co)represented by $A^G$ unless $G$ is trivial; $A^G$ is only its ring of global functions.
\end{proof}

\section{Criterion Theorem}
\label{sec:criterion}

There are many variants of moduli problems related to formal groups. This section applies to both graded and ungraded one-dimensional formal groups equally, and we will use the notation of $\mc{M}_{\mr{fg}_1}^{\spadesuit} := \mc{M}_{\mr{fg}_1} \text{ or }\mc{M}_{\mr{fg}_1}^{s}$. Let $\Def_F^\star$ denote the deformation of $F \in \mc{M}_{\mr{fg}_1}^{\spadesuit}(k)$ in $\mc{M}_{\mr{fg}_1}^{\spadesuit}$.

\begin{question} Consider a formal group $F \in \mc{M}_{\mr{fg}_1}^{\spadesuit}(k)$ of height $h$ which is over an algebraically closed field. What is required for a stack $\mc{M}$ with an action of $G$ to model the action of $G \subseteq \Aut(F)$ on $\Def_F^\star$?
\end{question}

\begin{answer} Exactly what Definition \ref{def:geometricmodel} requires, specialized to $\mc{N} = \mc{M}^{\spadesuit}_{\mr{fg}_1}$: this is Theorem \ref{thm:criterionLT}, and the two-group version is Theorem \ref{thm:twotowerLT}. \end{answer}

We opted to develop the background so thoroughly that the answer to our question falls directly into our lap: it is the lemmas of Sections \ref{sec:geomod}--\ref{sec:twotower}. We state the specialization once, as a single theorem in each of the one-group and two-group cases, and then give the two things that are new to the formal-group setting: a worked example and the conjugacy principle. 

\subsection{The Criterion Theorem, Lubin--Tate Form}

\begin{theorem} \label{thm:criterionLT} \label{robot time fg} \label{group robot time fg} \color{blue}{(criterion theorem: greed, robot time and group time, for formal groups) Let $F := \mc{F}(X) \in \mc{M}_{\mr{fg}_1}^{\spadesuit}(k)$ have height $h$, with $k \supseteq \F_{p^h}$ separably closed, and let $G \subseteq \Aut_k(F)$ be a closed subgroup. Suppose $(X \in \mc{M}(k), G, \widetilde{\mc{F}} \colon \Def^\star_X \simeq \Def^\star_F)$ is a geometric model for the action of $G$ (Definition \ref{def:geometricmodel}), i.e., $\mc{F} \colon \mc{M} \to \mc{M}^{\spadesuit}_{\mr{fg}_1}$ is injective on $G \subseteq \Aut(X)$ and induces an equivalence of rigidified deformation problems. Then:
\begin{enumerate}
\item[(i)] (greed) $\Def^G_X \simeq \Def^G_F \simeq \Spf(A)\stackyq G$, with $A \simeq W(k)[[u_1, \dots, u_{h-1}]]$ (resp. $A[\beta^{\pm 1}]$ in the graded case);
\item[(ii)] (robot time) $R\Gamma(\Def^G_X, \OO) \simeq R\Gamma(\Def^G_F, \OO)$;
\item[(iii)] (group time) $H^*_{cts}(G, \OO(\Def^\star_X)) \simeq H^*_{cts}(G, A)$, the $E_2$-page of the descent spectral sequence for $E_h^{hG}$ when $G$ is finite or, more generally, closed.
\end{enumerate}}
\end{theorem}

\begin{proof} (i) is Lemma \ref{weakgreed} with $\mc{N} = \mc{M}^{\spadesuit}_{\mr{fg}_1}$, together with Lemma \ref{reppingbaby} for the identification of $\Def^\star_F$; the action of $G$ on $\Spf A$ is continuous by Corollary \ref{cor:autacts}, so Definition \ref{def:defGprofinite} applies for profinite $G$. (ii) is Corollary \ref{robot time}. (iii) is Corollary \ref{group robot time}, i.e., Corollary \ref{continuousboogie} applied to both sides of (ii). \end{proof}

\subsection{Worked Example: Height 2}
The following is the $p=3$, height $2$ instance of Definition \ref{def:geometricmodel}; the sequel treats height $p-1$ for all odd $p$ and this is its $p=3$ member.

Elliptic curves appear in homotopy theory because they model the Lubin--Tate action of finite subgroups at heights 1 and 2
\cite{forms,moore2,chromsplit}. 

This is the case $p=3$ of the sequel\textquoteright s construction: the curve $X\colon y^p-y=x^{p-1}$ of the general construction has genus $\frac{(p-1)(p-2)}{2}$, so at $p=3$ it \emph{is} the supersingular elliptic curve $E\colon y^3-y=x^2$ with $\Aut_k(E)\simeq C_3\rtimes C_4$. Everything in this subsection is the $p=3$ member of the family built in the sequel, and the reader may keep it as a running example.

\begin{itemize}
\item Minimality of genus with automorphism group
$G'\simeq\Z/3\rtimes\Z/4$ forces $E$ to be supersingular --- formal group of height 2 --- and quotienting gives a span of Galois covers $$\PP^1\xleftarrow{\ C_2\ }E\xrightarrow{\ C_3\ }\PP^1.$$
\item By Serre--Tate \cite{st}, deformations of $E$ coincide with deformations of its formal group.
\item {\color{blue}{The ramification divisor of the $C_2$-cover $p_y\colon E\to\PP^1$ coincides with the level-$2$ structure on $E$; the action of $GL_2(\F_2)\simeq\Sigma_3$ on level structures coincides with the permutation action of $\Sigma_3$ on the branch points of $p_y$, and with the action of $\Aut(E)/C_2$ on the representing ring of $\Def_{(E,C_2)}$.}}
\item Since $C_2$ is central, $\Def_{(E,C_2)}$ carries an action of the full $G'\simeq\Aut(E)=\Aut(E,C_2)$.
\end{itemize}

From this: the automorphisms of the elliptic curve acting on its deformation space model the Lubin--Tate action restricted to $G'$, $$\Def^\star_{(E, C_2)}\stackyq G'\ \simeq\
\Def^\star_{\mc{F}(E)}\stackyq G',\qquad\text{hence}\qquad
H^*\big(G',\OO(\Def^\star_{(E,C_2)})\big)\simeq
H^*\big(G',\pi_*(E_2)\big).$$

\begin{remark} The reader might expect the $C_4$-cover rather than $C_2$; but the quotient of $E$ by $C_4$ is not Galois. The choice of the \emph{central} tame subgroup is not a convenience --- it is the selection principle that generalizes.
\end{remark}

\begin{remark} A global stack appears naturally here:
$\mc{M}_{1,1}^{N}$ completes at $(E,L)$ to this deformation problem, and the global form was used beautifully by Stojanoska \cite{vesna}. We leave the global setting to another place and time.
\end{remark}

\subsection{Conjugate Embeddings Give Isomorphic Actions}
\begin{lemma} \label{lem:conjugate} Let $\iota_1, \iota_2 \colon G \to \Aut_k(F)$ be two homomorphisms with $\iota_2 = u\,\iota_1\,u^{-1}$ for some $u \in \Aut_k(F)$. Then the automorphism $f_u \colon (\widetilde{F}, \alpha) \mapsto (\widetilde{F}, u \circ \alpha)$ of $\Def^\star_F$ intertwines the $G$-actions through $\iota_1$ and $\iota_2$, hence induces an equivalence $\Def^\star_F\stackyq _{\iota_1} G \simeq \Def^\star_F\stackyq _{\iota_2} G$ and an isomorphism $H^*_{cts}(G, A)_{\iota_1} \simeq H^*_{cts}(G, A)_{\iota_2}$. \end{lemma}
\begin{proof} $f_u \circ f_{\iota_1(g)} = f_{u\iota_1(g)} = f_{\iota_2(g) u} = f_{\iota_2(g)} \circ f_u$ since $a \mapsto f_a$ is a homomorphism (Definition \ref{autxaction}). \end{proof}
\begin{remark} Consequently, to identify the action of a finite subgroup $G \subseteq \Aut_k(F)$ obtained from a geometric model with ``the'' Lubin--Tate action of $G$, it suffices to know that all embeddings of the abstract group $G$ into $\Aut_k(F) = \OO_D^\times$ are conjugate. For the maximal finite subgroups of Lemma \ref{bouj} this is the uniqueness of the conjugacy class; we verify it by hand at height $p-1$ below, where it is the identity $\pi^{p-1} = -p$ for the uniformizer of the maximal totally ramified subfield. \end{remark}

\subsubsection{The maximal finite subgroup at height $p-1$} \label{sec:identify}
Write $\Aut(F)=\OO_D^\times$ for $D$ the division algebra of invariant
$1/(p-1)$, presented as $D=\Q_{p^{p-1}}\langle\Pi\rangle$ with
$\Pi a=a^{\varphi}\Pi$, $\Pi^{p-1}=p$.

\begin{lemma}[the normalizer]\label{lem:normalizer}
Choose $\omega_0\in\mu_{p^{p-1}-1}$ with
$\omega_0^{(p^{p-1}-1)/(p-1)}=-1$ and set $\pi:=\omega_0\Pi$, so
$\pi^{p-1}=-p$ and $L:=\Q_p(\pi)\cong\Q_p(\zeta_p)$ is a maximal totally ramified
subfield. For $\rho\in\mu_{p^{p-1}-1}$,
\[
\rho\,\pi\,\rho^{-1}=\rho^{\,1-p}\,\pi,
\]
which lies in $L$ if and only if $\rho\in\mu_{(p-1)^2}$. The subgroup
$j(G'):=\langle\zeta_p\rangle\rtimes\mu_{(p-1)^2}\subset\OO_D^\times$ is
isomorphic to $G'$, with surjective conjugation character
$\rho\mapsto\overline{\rho^{-(p-1)}}$.
\end{lemma}

\begin{proof}
The commutation rule follows from $\Pi\rho^{-1}=\rho^{-p}\Pi$; normalization
requires $\rho^{p-1}\in\mu(L)_{p'}=\mu_{p-1}$ ($L$ totally ramified), i.e.\
$\rho\in\mu_{(p-1)^2}$, using $(p-1)^2\mid p^{p-1}-1$. Conjugation sends
$\zeta_p\mapsto\zeta_p^{\,q}$ with $q\equiv\rho^{-(p-1)}$ modulo $\pi$;
$(p-1)$-th powering maps $\mu_{(p-1)^2}$ onto $\mu_{p-1}=\F_p^\times$.
\end{proof}

\begin{proposition}[uniqueness of the class]\label{prop:conjugacy}
Any subgroup $\Gamma\subset\OO_D^\times$ abstractly isomorphic to
$C_p\rtimes_\chi C_{(p-1)^2}$ with $\chi$ surjective is conjugate to $j(G')$ by an
element of $\OO_D^\times$.
\end{proposition}

\begin{proof}
Let $s$ generate the Sylow $p$-subgroup, $t$ a complement generator.

\emph{Step 1.} Skolem--Noether provides $u\in D^\times$ with $usu^{-1}=\zeta_p$;
since $v_D(\pi)=1$, multiplying $u$ by a power of a uniformizer of $\Q_p(s)$
(which centralizes $s$) makes $u$ a unit. Assume
$\langle s\rangle=\langle\zeta_p\rangle$.

\emph{Step 2.} $t$ normalizes $L$; by surjectivity of $\chi$, after replacing $t$
by a power we may assume $t\zeta_pt^{-1}=\zeta_p^{\,q}$ with $q$ the conjugation
value of a fixed generator $\rho_1$ of $\mu_{(p-1)^2}$. Then $c:=t\rho_1^{-1}$
centralizes $\zeta_p$, hence $c\in\OO_L^\times$ ($L$ maximal).

\emph{Step 3.} Let $g\in\Gal(L/\Q_p)$ be induced by $t$ and $\rho_1$ alike.
Telescoping $t^{p-1}=N^g(c)\,\rho_1^{p-1}$, both outer terms lie in
$\mu_{p-1}\subset\Z_p^\times$, so $N^g(c)\in\mu_{p-1}$; total ramification makes
$g$ trivial on residues, so $N^g(c)\equiv\bar c^{\,p-1}=1$ in $\F_p^\times$, and
$\mu_{p-1}\cap(1+p\Z_p)=1$: $N^g(c)=1$.

\emph{Step 4.} Hilbert 90: $c=\ell\,g(\ell)^{-1}$, and writing $\ell=\pi^nv$ with
$v\in\OO_L^\times$, $c=\delta^n\,vg(v)^{-1}$ with
$\delta=\pi g(\pi)^{-1}\in[q]^{-1}(1+\mathfrak m_L)$ of norm one. Principal units
form a pro-$p$ module over the order-$(p-1)$ group $\langle g\rangle$, hence are
cohomologically trivial; so $[q]\delta=w\,g(w)^{-1}$ with $w\in1+\mathfrak m_L$.
Conjugating by the unit $v':=w^nv$ carries $t$ to $[q]^{-n}\rho_1$, whose
$(p-1)$-th power is $\rho_1^{p-1}$: a generator of $\mu_{(p-1)^2}$. Done.
\end{proof}

\begin{remark}[level $k$] At height $\varphi(p^k)$, $k \geq 2$, the same construction embeds $G' = \Z/p^k \rtimes \Z/(p-1)^2$: $L = \Q_p(\zeta_{p^k})$ is a maximal totally ramified subfield of the division algebra of invariant $1/\varphi(p^k)$, $\sigma \mapsto \zeta_{p^k}$, and $\tau$ maps to a lift of a generator of $\Gal(L/\Q_p) \simeq (\Z/p^k)^\times$ of order $(p-1)^2$ normalizing $L$. That this is the unique conjugacy class of subgroups of this type is \cite{bouj}, Lemma \ref{bouj}; the argument above does not extend verbatim since $L/\Q_p$ is no longer tamely ramified over its $(p-1)$-st layer. \end{remark}

\begin{cor} \label{cor:identifygeneral} Let $\iota \colon G' = C_p \rtimes_\chi C_{(p-1)^2} \hookrightarrow \OO_D^\times = \Aut_k(F)$ be any embedding with $\chi$ surjective. Then there exist $u \in \OO_D^\times$ and $\alpha \in \Aut(G')$ with $u^{-1}\iota(g')u = j(\alpha(g'))$, so by Lemma \ref{lem:conjugate} the action of $G'$ on $\Def^\star_F$ through $\iota$ is isomorphic, via an automorphism of $\Def^\star_F$, to the canonical Lubin--Tate action of $j(G')$ restricted along $\alpha$. \end{cor}

\subsubsection{The tangent of Lubin--Tate space as a $G'$-module at height $\varphi(p^k)$}
\label{sec:lttangent}
Let $k \geq 1$, $h = \varphi(p^k) = p^{k-1}(p-1)$, $L = \Q_p(\zeta_{p^k})$, $\fp = (\zeta_{p^k}-1)$, $\pi$ a uniformizer of $\OO_L$, and $G' = \langle\sigma\rangle \rtimes \langle\tau\rangle = \Z/p^k \rtimes \Z/(p-1)^2$ the maximal finite subgroup of $\G_h$, embedded in $\OO_D^\times$ through $\sigma \mapsto \zeta_{p^k}$ and $\tau \mapsto$ a lift of a generator $\gamma$ of $\Gal(L/\Q_p)$ (the case $k=1$ is Lemma \ref{lem:normalizer}; for $k \geq 2$ the embedding is the same construction and the uniqueness of its conjugacy class is Lemma \ref{bouj}). Write $H = \langle \tau^{p-1}\rangle$ for the center, $\bar G = G'/H$, $\theta = \bar\chi(\tau) \in \F_p^\times$ the conjugation character, and $\bar\sigma, \bar\tau$ for the images in $\bar G$. The Honda formal group $H_h$ with this $\zeta_{p^k}$-action is a formal $\OO_L$-module of $\OO_L$-height one, with Dieudonn\'e module $M = \bb{D}(H_h) \simeq \OO_L \hat\otimes W$. The following two results compute the tangent space of $\Def^\star_{H_h}$ as a $G'$-representation, with no curve in sight; they are the target-side input to the third paper of this series \cite{higherlevels}, where they are used at level $k \geq 2$, and at $k=1$ they recover the action on the grading computed in \cite[Lemma 6.28]{sequel}.

\begin{theorem}[$u$-cancellation]\label{thm:norm}
Let $M=\mathbb D(H_h)\cong\OO_L\hat\otimes W$ with $F=\pi^{h-1}\varphi$,
$\tau=u\tilde\gamma$, $\varphi(u)/u=(\gamma(\pi)/\pi)^{h-1}$.
(a) The tame-graded tangent structure is independent of $u$ and of every
Hilbert-90 normalization. (b) Orientation is convention-free once the
action is defined on marked objects and transported; on tangents this
is $X\mapsto\bar\tau X\bar\tau^{-1}$, the pushforward convention.
(c) The top of the tangent is $\theta$ and the socle $\theta^{-1}$, at
every level. (d) $\tau^{p-1}=\omega_{\mathrm{Teich}}(\theta)^{-1}$ after
the tame-norm rescaling, an invariant generating $Z(G')$.
\end{theorem}
\begin{proof}
(a) On the $\pi^j$-graded line of $\bar M$,
$\tau(\lambda\pi^j)\equiv\theta^j\bar u\lambda\pi^j$, and on the
tangent lines $\Hom(\pi^{j'}\text{-line},\Lie\text{-line})$ the $\bar u$
cancel: weight $\theta^{-j'}$. (b) $\ker\delta_T=\{X:X\circ\pi=0\}$ in
either sign convention. (c) The top is dual to $\pi^{h-1}$:
$\theta^{-(h-1)}=\theta$ since $(p-1)\mid h$. (d) $\tau^{p-1}=N_\gamma(u)
\in\OO_L^\times$ centralizes $L$ and $F$; the tame norm is onto the
principal units, so a rescaling puts it in $\mu_{p-1}$, and mod $\pi$,
$\bar u^{p-1}=\theta^{h-1}=\theta^{-1}$.
\end{proof}

\begin{prop}[the target tangent]\label{prop:target}
$T:=T\Def(H_h)\otimes k\cong(\fp/p\OO_L)^\vee\otimes_{\F_p}k$, with
$\bar\sigma$ acting through multiplication by $\zeta_{p^k}$ on $\OO_L$
and $\bar\tau$ through $\Gal(L/\Q_p)$. It is uniserial of length $h-1$
with top $\theta$ and socle $\theta^{-1}$, and self-dual.
\end{prop}
\begin{proof}
$H_h$ with its $\zeta_{p^k}$-action is a formal $\OO_L$-module of
$\OO_L$-height one, so $\bar M=M/pM\cong\OO_L/p\otimes k=k[\varepsilon]/
(\varepsilon^h)$, $\varepsilon=\zeta_{p^k}-1$, on which $\bar\sigma-1$
is multiplication by $\varepsilon$ and $\bar\tau$ acts on
$\varepsilon^j/\varepsilon^{j+1}$ by $\theta^j$. The Hodge filtration
of the one-dimensional group $H_h$ is the line $\ell=(\varepsilon^{h-1})$,
and $T=\Hom(\ell,\bar M/\ell)=(\bar M/\ell)\otimes\ell^\vee$. Now
$\bar M/\ell=k[\varepsilon]/(\varepsilon^{h-1})$ is uniserial under
$\varepsilon$ with top weight $1$ and socle weight $\theta^{h-2}$, and
$\ell^\vee$ has weight $\theta^{-(h-1)}=\theta$ (Theorem
\ref{thm:norm}(a),(c)): so $T$ is uniserial of length $h-1$ with top
$\theta$ and socle $\theta^{h-1}=\theta^{-1}$. The same holds for
$(\fp/p\OO_L)^\vee$, the dual of $(\varepsilon)/(\varepsilon^h)$, and
two uniserial $k\bar G$-modules with the same top and length are
isomorphic: $k\bar G$ is a Nakayama algebra, so a uniserial module is determined by its top and its length \cite[Thm.\ B]{higherlevels}.
\end{proof}

\subsection{Two Tower Isomorphisms, Lubin--Tate Form}

We now consider the case of two groups. Fix $F := \mc{F}X \in \mc{M}^{\spadesuit}_{\mr{fg}_1}(k)$ of height $h$ over $k \supseteq \F_{p^h}$.

\begin{question} What is required for a stack $N$ with an action of $G'$ to model the action of $G \subseteq \Aut_k(F)$ on $\Def^\star_F$?
\end{question}

\begin{answer} A span: a stack $M$ carrying commuting actions of $G$ and $G'$, which is a $G$-torsor over $N$ and a $G'$-torsor over $\Def^\star_F$. Then the residual actions match, and so do all their cohomologies --- Theorem \ref{thm:twotowerLT}. \end{answer}

\begin{theorem} \label{thm:twotowerLT} \color{blue}{(two-tower criterion: bi robot time and bi group time, for formal groups) Let $N$ and $M$ be prestacks such that $M$ carries commuting actions of $G$ and $G'$ making it a $G$-torsor over $N$ and a $G'$-torsor over $\Def^\star_F$:
\[\begin{tikzcd}
	& M \\
	N & {} & \Def_{F}^\star
	\arrow["G"', from=1-2, to=2-1]
	\arrow["{G'}", from=1-2, to=2-3]
\end{tikzcd}\]
Then $N\stackyq G' \simeq \Def^G_F$ (quotients by the \emph{residual} actions), hence $R\Gamma(N\stackyq G', \OO) \simeq R\Gamma(\Def^G_F, \OO)$, and, when $G, G'$ are profinite and $N$ is pro-represented by a complete local ring with continuous $G'$-action, $H^*_{cts}(G', \OO(N)) \simeq H^*_{cts}(G, \OO(\Def^\star_F))$.}
\end{theorem}

\begin{proof} Lemma \ref{spanny} and Corollaries \ref{bi robot time}, \ref{bi group time} with $N' = \Def^\star_F$. \end{proof}

\part{Toolkits for Building Models}\label{part:toolkits}
\partgloss{Techniques, collected where constructions of such models can reach them.}

\section{Toolkit: Deformations of Curves with Group Action}
\label{gstacks}

This section is the general deformation theory of a curve with a finite group action, following Bertin--M\'ezard \cite{bm} and Wewers \cite{wewers}; it is the input for the sequel and for the higher-height paper. It uses the equation of no particular curve.

\subsection{Deck Transformations}
Recall (Definition \ref{deform}) $\Def^\star_X$, $\Def^G_X$ and $\Def_X := \Def_X^{\Aut(X)}$, and the action of $\Aut(X,G)$ on $\Def^G_X$ (Definition \ref{autxaction}). We now identify that action with the deck transformations of $\Def^\star_X$ over $\Def_X$.

\begin{defn} Let $q: Y \to X$ be a morphism of categories fibered in groupoids over some base category. The group $\Aut_Y(X)$ of automorphisms of $Y$ over $X$ consists of equivalence classes of pairs $(f, \psi)$, where $f: Y \to Y$ is a $1$-morphism of groupoids and $\psi: q \to qf$ is a $2$-morphism. Two such pairs $(f, \psi)$ and $(f', \psi')$ are equivalent if there is a $2$-morphism $\phi: f \to f'$ so that $\psi'\phi = \psi.$ The composition law reads $$(g, \psi)(f, \phi) = (gf, (f^*\psi)\phi).$$
\end{defn}

\begin{defn} \label{homo} There is a homomorphism $\Aut(Y) \to \Aut_{\Def_Y}(\Def_Y^\star)$ sending $\sigma$ to $f_{\sigma}$, where $f_{\sigma}$ is the transformation $(G, i) \to (G, i\sigma),$ as in Defn \ref{autxaction}.
\end{defn}

\begin{lemma} \label{halo} (halo) Given a pro-representable moduli problem $\Def_Y^\star$, pro-represented by a ring $R$, then the homomorphism in Defn \ref{homo} induces an isomorphism of automorphism groups $\Aut(Y) \simeq \Aut_{\Def_Y}(\Def_Y^\star).$ \end{lemma} 

\begin{proof} By being pro-representable, the deformation problem is discrete. The claim follows from the argument given in Lemma 7.18 in \cite{goerss}, which uses only Lemma 7.13, both of which work word for word in more general discrete deformation moduli problems. \end{proof}

\subsection{$G$-Stacks of Curves}
\label{gstacksofcurves}
\begin{defn} We define a curve as a smooth proper scheme of relative dimension one over a base scheme.\end{defn} 

\begin{defn}
We define the prestack $$\mr{Curve}_{G}: \mr{CRing} \to \mathrm{Grpd},$$ such that $\mr{Curve}_{G}(R)$ is the groupoid with 
\begin{itemize} 
\item objects: tuples $(\mf{X}_{/R}, G)$ consisting of a curve $\mf{X}$ over $R$ with a $G$-action $G \to \Aut_R(\mf{X}).$ 
\item morphisms: $G$-equivariant maps between curves. 
\end{itemize}
\end{defn}

\begin{remark} If $f: R \to S,$ and $C$ is a curve over $\Spec(R),$ then $f^\ast C$ is a curve over $\Spec(S)$. Since $f^\ast(G \times_{\Spec(R)} C) = G \times_{\Spec(S)} f^\ast C$, the pull-back curve has $G$-action. Thus, we have a prestack. In the previous sentence, $G$ is shorthand for $\Spec(\mr{map}(G,R))$ or $\Spec(\mr{map}(G,S))$ respectively. \end{remark}

\begin{defn} Deformations over this prestack, in the sense of Definition \ref{deform} with $\mc{M} = \mr{Curve}_G$, are $\Def_{(X,G)}(\Lambda)$; explicitly, this is the prestack with 
\begin{itemize} 
\item objects: triples $\xi :=(\mc X,\rho,\iota)$: a smooth proper curve $\mc X/\Lambda$, a lift $\rho\colon G\to\Aut_\Lambda(\mc X)$ of the action $\rho_0\colon G\hookrightarrow\Aut_k(X)$, and a $G$-equivariant isomorphism $\iota\colon\mc X_s\xrightarrow{\sim}X$ of the special fiber,
\[\begin{tikzcd} G\arrow[r,"\rho"]\arrow[rd,"\rho_0"'] & \Aut_\Lambda(\mc X)\arrow[d,"\mathrm{res}"]\\ & \Aut_k(X). \end{tikzcd}\]
\item morphisms: $\xi \to \xi'$ is a $G$-equivariant isomorphism $\phi\colon\mc X\to\mc X'$ over $\Lambda$ with $\iota'\circ\phi_s=\iota$.
\end{itemize}
\end{defn}

\begin{example} \label{ex:autxgaction} (of Defn \ref{autxaction})
Here $G$ is part of the deformed datum $(X, G) \in \mr{Curve}_G(k)$, so Definition \ref{autxaction} applied to $\Def^\star_{(X,G)}$ gives an action of the automorphisms of $(X, G)$, i.e., of the centralizer $\Aut(X, k, G)$ of $G$ in $\Aut(X,k)$ (the elements $g'$ with $g' \circ \iota$ again $G$-equivariant). Let's see how it changes as $G$ changes.
\begin{itemize} 
\item If $G \simeq \Aut(X, k)$, then $\Aut(X, k, \Aut(X, k)) \simeq Z(\Aut(X,k))$, where $Z$ denotes the center of the group. 
\item If $G \simeq \text{id}$, then $\Aut(X, k, \text{id}) \simeq \Aut(X, k)$.
\item If $G \subseteq Z(\Aut(X, k)),$ then $\Aut(X, k, G) \simeq \Aut(X, k)$, since $g \in G$ commutes with every element of $\Aut(X, k).$
\end{itemize}

Thus, there is an $\Aut(X, k)$ action on $\Def_{(X, G)}$ when $G$ is in the center of $\Aut(X, k).$ This will come in handy for us. (An element $g'$ of the normalizer of $G$ which does not centralize it also acts, as in Definition \ref{autxaction}, but only after replacing the lift $\rho$ by $\rho^{g'}$; we never need this, since in our examples $G$ is central.)
\end{example}

\subsection{Local and Marked Moduli Problems}
\label{localgstacks} 

As usual for the source of our formal moduli problems, we work with $\widehat{\Art}_k$, the category of complete $k$-algebras with given reduction to the residue field $(R, \iota: R/\mf{m} \to k)$. 

We define the moduli problem of divisors of a curve. 

\begin{defn} Given a curve $X$ and with a divisor $D\in\Div(X)$, define the prestack $$\Def_{(X ; D)}: \widehat{\Art}_k \to \mathrm{Grpd},$$ such that $\Def_{(X ; D)}(R)$ is the groupoid with
\begin{itemize} 
\item objects $\colon$ tuples $(\mf{X},\mf{D})$ consisting of
\begin{itemize} 
\item a curve $\mf{X}$ over $R$ 
\item a divisor $\mf{D}\in \Div(\mf{X})$ 
\item a map $\iota: \mf{X}|_k \simeq X$ witnessing the lift of the curve which induces $\iota^*: D \simeq \mf{D}|_k$ witnessing the lift of the divisor.
\end{itemize}
\item morphisms are pairs $\colon$ $(\mf{X},\mf{D}) \to (\mf{X'},\mf{D}')$ consisting of a morphism $f: \mf{X} \to \mf{X}'$ such that the induced morphism on divisors $f^*: \Div(\mf{X}') \to \Div(\mf{X})$ sends $\mf{D}' \mapsto \mf{D}$, and such that $f|_k$ reduces to the identity on $X$. 
\end{itemize}
\end{defn}

\begin{defn} \label{LightYagami} The functor $$L: \Def_{(X, G)} \to \Def_{(X/G; B)}$$ is described as follows. Given a $G$-equivariant deformation of $X$, $(\mf{X}, G)$, consider its corresponding algebraic quotient $q: \mf{X} \to \mf{X}/G$. Let $B$ be the reduced branch points of the map $q$, i.e., the reduced branch divisor. Regard the pair $(\mf{X}/G ; B)$ of the quotient curve and its branch points.
\end{defn}

\begin{center}\includegraphics[width=8cm]{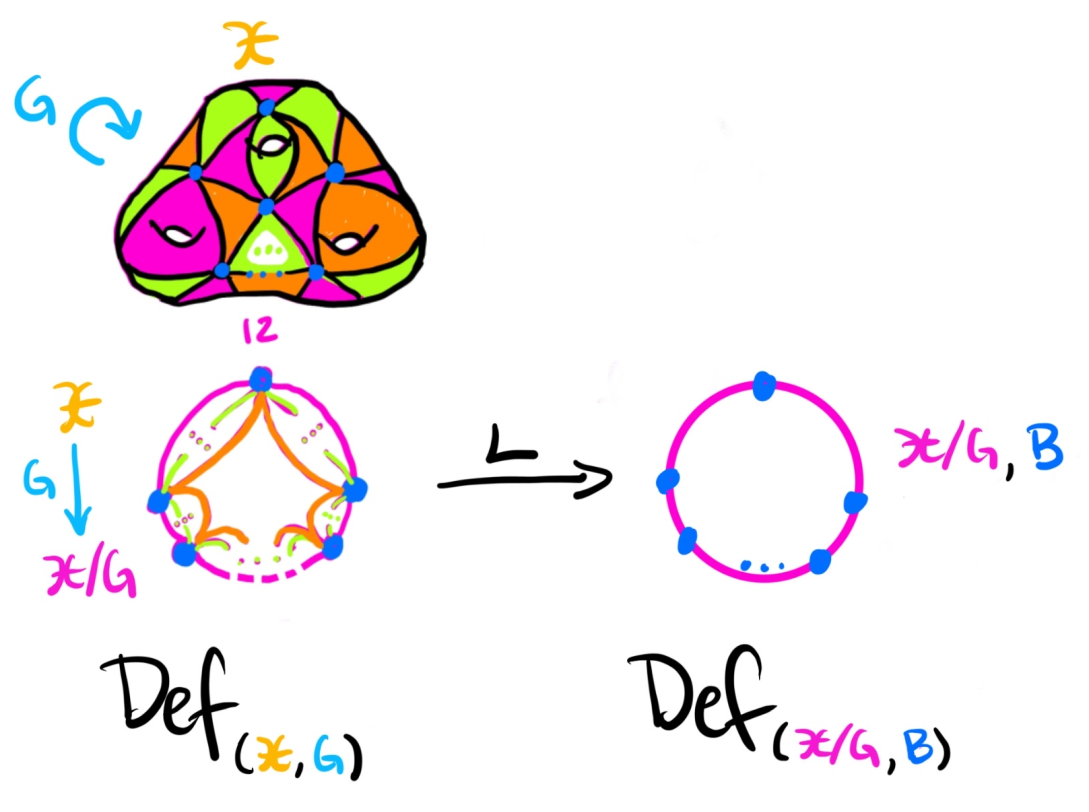}\end{center}

Next, we define a moduli problem of $G$-representations of discs in char $p$.

\begin{defn} Fix a finite subgroup $G$ of $\Aut(k[[T]])$ (finiteness is assumed throughout this subsection), this inclusion can be considered as a representation $$
\rho_G \colon G \hookrightarrow \Aut(k[[T]]).$$ Consider the deformation problem $$\Def_{\rho_G} \colon \widehat{\Art}_k \to \mathrm{Grpd},$$ \noindent such that $\Def_{\rho_G}(A)$ has as 
\begin{itemize} 
\item objects $\colon$ representations $$\widetilde{\rho}_G \colon G \to \Aut (A[[T]])$$ \noindent in $\mr{Rep}(A[[T]])$, such that the composition of $\widetilde{\rho}_G$ with $\Aut(A[[T]]) \to \Aut(k[[T]])$ is the representation $\rho_G.$
\item morphisms $\colon$ isomorphisms of representations.
\end{itemize}
\end{defn}

\begin{construction} \label{stabilizer} Given a curve $X$ over $k$ with an action of $G$, consider every point $x$ with a nontrivial stabilizer group $G_x$. For each such point, the restriction of the action to the disc around $x$ gives a map $$G_x \to \Aut(\OO_{X, x}) \simeq \Aut(k[[T]]).$$
\noindent We define the associated composite local deformation problem as $$\Def_{(X, G)}^{\mr{loc}} :=\prod_i \Def_{\rho_{G_{x_i}}}$$.
\end{construction}

\begin{defn} \label{localization} The localization functor $$Loc: \Def_{(X, G)} \to \Def_{(X, G)}^{\mr{loc}}.$$ is defined by considering for every deformation $(\mf{X}_{/A}, G \subseteq \Aut_A(\mf{X}), \iota)$, all points $\mf{x} \in X$ with nontrivial stabilizer. For each such point, the restriction of the action to the disc around $\mf{x}$ gives a map $$\rho_{G_\mf{x}}: G_{\mf{x}} \to \Aut(A[[T]]),$$ where the lift of $G$ and the information of $\iota$ ensure we have the appropriate reduction property to a corresponding representation $\rho_{G_x}: G_x \hookrightarrow \Aut(k[[T]])$ for a point $x \in X$ with nontrivial stabilizer $G_x \simeq G_{\mf{x}}$. Taking the product of all points with nontrivial stabilizer gives us an element of $\Def_{(X, G)}^{\mr{loc}}$ as desired. 
\end{defn}

\begin{center}\includegraphics[width=8cm]{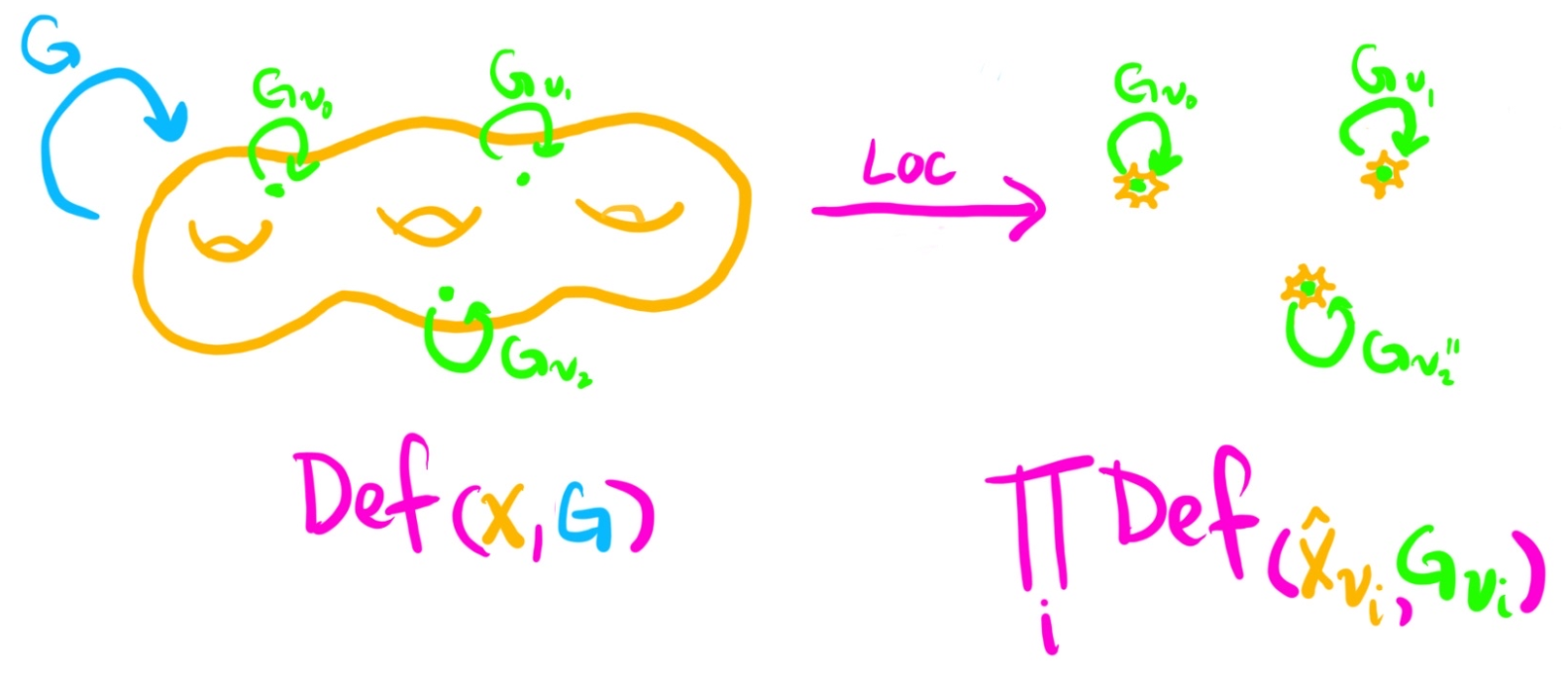}\end{center}

\begin{lemma} \cite[Thm.\ 3.3.5]{bm}, \cite[Lemma 2.9]{itsyg}
For $G$ finite, $\mr{Loc}$ is formally smooth. \end{lemma}

\subsection{Graded Version} To access $\pi_*(E_{(p-1)})$ rather than just $\pi_0(E_{(p-1)})$, we will work with deformations of formal groups with a Lie algebra structure. We mention the corresponding stack on the curve side. 

Let $G$ be a finite group. Given a $G$-action on a curve $\mf{X}$, this will induce a $G$-action on $\mf{g} := H^0(\mf{X}, \Omega^1_\mf{X})$, and thus an idempotent decomposition $$\mf{g} \simeq \bigoplus_{\chi \in \Hom(G, \G_m)} \mf{g}^{\chi^i}.$$ \noindent The notion of idempotent decomposition is discussed in more detail in Section \ref{idempotent}.

\begin{defn} \label{gradedboi}
We define the prestack $$\mr{Curve}_{G}^{\mr{Lie }\chi^i}: \mr{CRing} \to \mathrm{Grpd},$$ such that $\mr{Curve}_{G}^{\mr{Lie }\chi^i}(R)$ is the groupoid with 
\begin{itemize} 
\item objects: tuples $(\mf{X}_{/R}, \mf{g}^{\chi^i}, G)$ consisting of a curve $\mc{X}$ over $R$, $\mf{g}^{\chi^i}$, and a $G$-action $G \to \Aut(R,\mf{X})$
\item morphisms: $G$-equivariant maps between curves. 
\end{itemize}
\end{defn}
\subsection{Our Tangent Complexes}
\label{tangentcomplexes}

We will only work with smooth objects $X$ with an action of a finite group $G$ in this paper, in which case the equivariant cotangent complex collapses to $\pi_0(\mc{L}^G_{X/k}) \simeq \Omega_X,$ where the latter is considered as an $G-\OO_X$-module equipped with the $G$-action inherited from $X$. We refer the reader to Illusie \cite{illusie} and Adeel Khan's notes on the general equivariant cotangent complex \cite{khan}. 

\begin{defn} An equivariant extension of $\OO_Y$ by $\mc{I}$ is a short exact sequence of sheaves $$0 \to \mc{I} \to \mc{E} \to \OO_Y \to 0$$ \noindent where we consider $\mc{I}$ to be sheaf of ideals with square zero and a subsheaf of $\mc{E}$. 
\end{defn}

These $G$-equivariant extensions allow us to classify deformations of $G$-equivariant schemes. We write $\Hom_G$ as short-hand for $\Hom_{G-\OO_X}$ internal to $G$-equivariant quasi-coherent sheaves on $X$.

\begin{cor} \label{tangentstuff} \cite[Cor.\ 2.3]{wewers} (i) The group of automorphisms of any fixed equivariant extension of $Y$ by $\mc{I}$ is canonically isomorphic to $\Hom_G(\Omega_{Y/S}, \mc{I}).$ \newline
\noindent (ii) The group of isomorphism classes of equivariant extensions of $Y$ by $\mc{I}$ is canonically $\Ext^1_G(\mc{L}_{Y/S}, \mc{I})$. \newline \noindent (iii) \cite[Thm.\ 3.3]{wewers} The obstruction to deforming $Y$ equivariantly by $\mc{I}$ is canonically $\Ext^2_G(\mc{L}_{Y/S}, \mc{I})$. 
\end{cor}

\begin{defn} We define the {\color{Bittersweet}{tangent sheaf}} $T_X$ as $Hom_{\OO_X}(\Omega_X, \OO_X),$ and the {\color{Bittersweet}{equivariant tangent sheaf}} as $Hom_{G}(\Omega_X, \OO_X)$.
\end{defn}

As stated in Lemma \ref{tangentstuff}, $G$-equivariant extensions of the morphism $X \to S$ by a quasicoherent $G$-$\mc{O}_Y$-module $\mc{F}$ are classified by the group $\Ext^1_G(\mc{L}_{X/S}, \mc{F})$. This is the basis for all results on equivariant deformations of $X \to S$. 

\begin{theorem} (i) \cite[Lemmas 2.5, 2.6]{itsyg} \label{discreteanddecorated} $(\Def_{(X, G)}^{\mr{Curve}_G})^\star$ is discrete and pro-representable when $g \geq 2$ and $X$ is smooth. \newline \noindent 
(ii) \cite[Prop.\ 2.15]{itsyg} $\Def_{(X, G, Q_1)^\star}$ is discrete and pro-representable when $X$ is a smooth genus one curve with marked point $Q_1,$ i.e., $(X, Q_1)$ is an elliptic curve. 
\end{theorem}

\begin{theorem} [Prop 3.2.1 \cite{bm}] If $X$ is non-singular, there is a canonical identification of $k$-vector spaces $$T(\Def_{(X, G)}) \simeq R^1\Gamma^G(X, T_X) \simeq \Ext^1_G(\Omega_X, \OO_X).$$ \end{theorem}

\begin{lemma} \label{lowbadboys} \cite[Prop.\ 3.3.2]{bm} Let $R$ be the points in $X$ with nontrivial stabilizer group under the action of a finite group $G$. Let $\widehat{T}_{X, x} \simeq k[[T]]\frac{d}{dT}$ be the tangent sheaf completed at a point $x$, where $T$ is the uniformizer at that point. The tangent space of $\Def_{(X, G)}^{\mr{loc}}$ is $$T(\Def_{(X, G)}^{\mr{loc}}) \simeq \bigoplus_{x \in R} H^1(G_{x}, \widehat{T}_{X, x}).$$ 
\end{lemma}

\subsection{Equivariant Sheaves and Equivariant Sheaf Cohomology}
\label{equivariant sheaves}

In order to define the tangent complexes of our deformation problems, we first must introduce equivariant sheaves and their cohomology. For a finite group $G$ this is the special case of Section \ref{continuouscohom} in which $G$ is finite but $X$ is a curve rather than an affine scheme: $\QCoh^G(X) \simeq \QCoh(X \stackyq G)$ is Lemma \ref{bitchitsequivalent} relative to $X$, and the Grothendieck spectral sequence below replaces the affineness used in Lemma \ref{howitbe}. We write $X \stackyq G$ for the stacky quotient (written $X/G$ in Sections \ref{stacks}--\ref{sec:criterion}) because in this section $X/G$ must denote the algebraic quotient curve.

\begin{defn}
Given an action $a: G \times_S X \to X$ of a group scheme $G$ on a scheme $X$ an equivariant sheaf $\mc{F}$ on $\mc{X}$ is a sheaf of $\OO_X$-modules together with an isomorphism of $\OO_{G \times X}$ modules 
$$a^*\mc{F} \simeq p_2^*\mc{F}.$$
\end{defn}

\begin{defn} We define the $G$-equivariant global sections as the right adjoint of the functor which constructs trivial $G$-sheaves.
\[\begin{tikzcd}
	{\QCoh(*)} & \perp & {\QCoh^G(X)}
	\arrow["{\Gamma^G}"', curve={height=18pt}, tail reversed, no head, from=1-1, to=1-3]
	\arrow["{\mr{triv}}", curve={height=-18pt}, from=1-1, to=1-3]
\end{tikzcd}\]

The right adjoint of the forgetful functor from $\QCoh^G(X) \to \QCoh(X)$ forgetting the $G$-action is the coinduced module of the trivial action, $\widetilde{(-)} \colon \mc{F} \mapsto a_*p_2^*\mc{F} =: \widetilde{\mc{F}}$ (this coincides with the induced module since $G$ is finite).
\[\begin{tikzcd}
	{\QCoh(X)} & \perp & {\QCoh^G(X)}
	\arrow["{\mc{U}}"', curve={height=18pt}, tail reversed, no head, from=1-1, to=1-3]
	\arrow["{\widetilde{(-)}}", curve={height=-18pt}, from=1-1, to=1-3]
\end{tikzcd}\]
\end{defn}

We denote the algebraic quotient of $X$ by a finite $G$ action as $Y := X/G$. The stacky quotient we refer to as $X\stackyq G$. 

\begin{construction}
   Given a finite $G$ group scheme acting on $X$, the quotient map $q: X \to X/G$ factors through $X \stackyq G$ uniquely by the universal property that all maps from $X$ to objects with trivial $G$ action must factor uniquely through $X\stackyq G.$
\[\begin{tikzcd}
	X \\
	& {X\stackyq G} \\
	{X/G}
	\arrow["g", from=1-1, to=2-2]
	\arrow["f"', from=1-1, to=3-1]
	\arrow["h", from=2-2, to=3-1]
\end{tikzcd}\]

Equivariant sheaves on $X$ are sheaves on the stacky quotient of $X$ by $G$, that is, $\QCoh(X\stackyq G) \simeq \QCoh^G(X)$. These maps induce via pushforward a map from $G$-sheaves on $X$ to sheaves on $Y := X/G.$
\[\begin{tikzcd}
	{\QCoh^G(X)} & {\QCoh(X\stackyq G)} & {\QCoh(X/G)}
	\arrow["{\simeq}", from=1-1, to=1-2]
	\arrow["{q^G_*}", from=1-2, to=1-3]
\end{tikzcd}\]

Given a $G$-Galois covering $f: X \to Y$, and $\mc{F}$ a $G-\OO_X$ module, we may construct a sheaf of $\OO_Y$ modules $q^G_*\mc{F}.$
\end{construction}

\begin{remark} If $X \to X/G$ is an \'etale map, then $\QCoh(X \stackyq G) \simeq \QCoh(X/G)$. \end{remark}
\[\begin{tikzcd}
	{\QCoh^G(X)} & {\QCoh(X/G)} & {\QCoh(*)} \\
	{\mc{F}} & {q_*^G\mc{F}} & {\Gamma(X/G, q_*^G\mc{F})} \\
	{\mc{F}} && {\Gamma^G(X, \mc{F})}
	\arrow["{q_*^G}", from=1-1, to=1-2]
	\arrow["{\Gamma^G}", curve={height=-24pt}, from=1-1, to=1-3]
	\arrow["\Gamma", from=1-2, to=1-3]
	\arrow[maps to, from=2-1, to=2-2]
	\arrow[no head, from=2-1, to=3-1]
	\arrow[shift left, no head, from=2-1, to=3-1]
	\arrow[maps to, from=2-2, to=2-3]
	\arrow["\simeq", from=2-3, to=3-3]
	\arrow[maps to, from=3-1, to=3-3]
\end{tikzcd}\]

The composite map is the constant sheaf on $X$ with trivial action, and its right adjoint is equivariant global sections $\Gamma^G$. Right adjoints are commutative, so we may consider the right adjoint to the map $f$, which gives us a sheaf on $Y$ whose global sections are the equivariant global sections.

\begin{defn} We define $R^q\Gamma^G(X, \mc{F})$ to be the equivariant cohomology of  $\mc{F} \in \QCoh^G(X)$. Sometimes we denote this by $H^q(G, \mc{F})$. \end{defn}

\noindent Since $$\Gamma^G(X, \mc{F}) \simeq \Gamma(X/G, q_*^G(\mc{F})),$$ i.e., $\Gamma^G_X \simeq \Gamma_{X/G} \circ q_*^G$ we may use the Grothendieck spectral sequence for compositions of functors. We denote $Y := X/G$. The spectral sequence collapses to give us the following short exact sequence:

$$H^1(Y,q_*^G(\mc{F})) \to R^1\Gamma^G_X(\mc{F}) \to H^0(Y, R^1(q_*^G\mc{F}))$$

\subsection{Local-to-Global for Equivariant Deformations of Curves}
\begin{lemma} \label{moodstabilizer} (mood stabilizer) Let $X$ be a curve over a field $k$ carrying a $G$-action. Let $R$ be the set of points in $x \in X$ with nontrivial stabilizer group $G_x$. Let $\widehat{T}_{X, x}$ denote the tangent sheaf of $X$ completed at the point $x$, this is a module of the form $\widehat{T}_{X, x} \simeq k[[T]]\frac{d}{dT}$ where $T$ is the uniformizer at the point $x.$ Then,
(i) $$R^qq_*^GT_X \simeq \bigoplus_{x \in R} H^q(G_x, \widehat{T}_{X, x}).$$ 
(ii) For $|G_x|$ invertible in $k$, $H^q(G_x, \widehat{T}_{X, x})$ is zero for $q > 0.$
\end{lemma}

\begin{proof} 
(i) Consider an affine $V$ on $X$, and $f(V)$ on $X/G$, then restrict ourselves to this affine set, then consider the collapse of this spectral sequence from earlier. Given a subset $\iota: V \hookrightarrow X$, and a sheaf $\mc{G} \in \mr{QCoh}(X),$ then by adjunction we have a canonical map $\mc{G} \to \iota_*\iota^*\mc{G} \simeq \iota_* \mc{G}|_V.$

Note that since $\Hom_G(\Omega_X, \OO_X) \simeq \Hom_{\OO_X}(\Omega_X, \OO_X)^G$, we have another spectral sequence converging to $R^i\Gamma_X^G(T_X)$, since $\Gamma_X^G \simeq \Gamma_{X/G} \circ q_*^GT_X \simeq \Gamma_X \circ (-)^G$. Considering these two spectral sequences, we find a collapse:

\begin{align*} 
H^p(X/G|_{f(V)}, R^qq_*^GT_X|_{f(V)})) = 0 &\text { if } p > 0 \\
H^q(G, H^0(X|_V, T_X|_V)) = 0 &\text { if } q > 0\\
\end{align*} \noindent This collapse of spectral sequences gives us an equivalence connecting the two:
$$R\Gamma_X^G, T_X|_V \simeq H^0(X/G|_{f(V)}, R^qq_*^GT_X|_{f(V)}) \simeq H^q(G, \Gamma(X|_V, T_X|_V))$$

Consider the closed subset $\iota:V\hookrightarrow X$ of points on which $G$ acts nonfreely, and consider the complement $V^c:=X-V$. At discs around points $y$ such that $f^{-1}(y)$ lay in $V^c$ (that is, points $y$ which are not branch points), $(q^G_*\mc{F})_y \simeq \bigoplus_{x \in f^{-1}(y)} \mc{F}$ is a free summand of $f^{-1}(y)$ copies of the sheaf $\mc{F}$, permuted cyclically according to the global action of $G$. Free constructions are exact, thus for $q \geq 0$, $$R^qq_*^G(\mc{F}|_{V^c}) \simeq 0.$$ 

Thus, $R^qq_*^G(\mc{F})$ is a skyscraper sheaf with values on $y \in f(V)$ such that $G$ has nontrivial stabilizer on $f^{-1}(y)$. Let $x$ be a point on $X$ in which $G$ acts with nontrivial stabilizer, then $\Gamma(X|_x, T_X|_x) \simeq \widehat{T}_{X, x}),$ where the latter is a $G_x$-module. 

The cohomology group of interest at a point $y_i \in f(V)$ is of the form $$H^0(X/G|_{y_i}, R^qq_*^GT_X|_{y_i}) \simeq H^q(G, \prod_{x \mapsto y_i} \widehat{T}_{X, x}).$$ Note that for any $x' \in \mr{Orb}_G(x)$ $$\mr{Ind}_{G_{x'}}^G \widehat{T}_{X, x'} \simeq \mr{Ind}_{G_{x}}^G \widehat{T}_{x, x},$$ by the orbit-stabilizer theorem thus it matters not which point from the orbit we choose. Further, $$\mr{Ind}^{G}_{G_x} \widehat{T}_{X, x} \simeq \bigoplus_{x \in \Orb_G} \widehat{T}_{X,x}.$$  Shapiro's lemma further tells us $$H^q(G, \mr{Ind}^{G}_{G_x} \widehat{T}_{X, x}) \simeq H^q(G_x, \widehat{T}_{X, x}).$$ The clean formulation claimed follows from $$\bigoplus_{x \in R} H^q(G_x, \widehat{T}_{X, x}) \simeq \bigoplus_{y \in \supp(B_f)} \bigoplus_{x_k \in \Orb_G(f^{-1}(y))} H^q(G_{x_k}, \widehat{T}_{X, x_k}).$$
(ii) Consider a cyclic group $C_m$ acting on an $R$-module, the group cohomology $H^i(C_m, R)$ will vanish when $i > 0$ if $m$ is invertible in $R$ as the transfer map is invertible.
\end{proof}

\subsection{Equivalence of Deformation Problems}

\begin{theorem} \label{breakitup} (break it up) For $X$ a smooth curve over a field $k$ with an action of $G$, if $n = \dim \Ext^1_G(\Omega_X, \OO_X)$ is finite. Then, $\Def_{(X, G)}$ is representable by a ring $R$ which carries an $\Aut(X, G)$-action by Defn \ref{autxaction}. \newline \noindent 
(i) without further assumptions, $$R \simeq W(k)[[t_1, ..., t_n]]/(f_1, ...f_i)$$ \newline \noindent  (ii) Let $R_\sigma$ be the ring representing $\Def_{(X, G)}^{\mr{loc}}$, and $m = \dim H^0(X/G, R^1q_*^GT_X)$ then $$R \simeq R_\sigma[[t_1, ..., t_m]]$$ \newline \noindent  (iii) If $\Ext^2_G(T_X, \OO_X) = 0$ then $\Def_{(X, G)}$ is formally smooth and represented by a ring $$R \simeq W(k)[[t_1, ..., t_n]].$$
\end{theorem}

\begin{proof} 
(i) and (iii) follow immediately from Lemma \ref{tangentstuff}. (ii) This proof takes the strategy of identifying the map between the cotangent complexes of these formal moduli problems $Loc: \Def_{(X, G)} \to \Def_{(X, G)}^{\mr{loc}}$ with the surjective map arising from the short exact sequence computing $R^1\Gamma^G(T_X)$ given by Grothendieck spectral sequence by $\Gamma^G_X \simeq \Gamma_{X/G} \circ q^G_*$. 

$$H^1(Y, q_*^GT_X) \to \Ext^1_G(\Omega_X, \OO_X) \to H^0(Y, R^1q_*^GT_X).$$

Lemma \ref{moodstabilizer} shows that $H^0(X/G, R^1q^G_*T_X) \simeq \bigoplus_{x \in R} H^1(G_{x}, T_{X, x})$, the latter being the tangent space of $\Def_{(X, G)}^{\mr{loc}}$ by Lemma \ref{lowbadboys}.  The map $\delta(\mr{Loc})$ is the surjective map above by Lemma 3.3.2 \cite{bm}. We summarize the result in the diagram below.
\[\begin{tikzcd}
	& {T(\Def_{(X, G)})} & {T(\Def_{(X, G)}^{loc})} \\
	&& {\bigoplus_{x \in R} H^1(G_{x}, T_{X, x})} \\
	{H^1(X/G, h_*T_X)} & {R^1\Gamma_X^G(T_X)} & {H^0(X, R^1h_*T_X)}
	\arrow["{\delta(Loc)}", from=1-2, to=1-3]
	\arrow["\simeq", from=1-2, to=3-2]
	\arrow["\simeq"', from=1-3, to=2-3]
	\arrow["\simeq"', from=2-3, to=3-3]
	\arrow[hook, from=3-1, to=3-2]
	\arrow[two heads, from=3-2, to=3-3]
\end{tikzcd}\]
\end{proof}

\subsection{Covers, Ramification, and Algebraic Quotients}
\label{sec:coversram}
\subsubsection{Cyclic Covers of Curves}
\label{sec:cycliccovers}

\begin{defn}
A map of curves $f\colon X\to Y$ is a {\color{Bittersweet}{$G$-Galois cover}}
if it is a $G$-Galois extension of function fields,
$\Gal_k(k(X)/k(Y))\simeq G$.
\end{defn}

\begin{remark} If we were virtuous of heart we would say ``generically
$G$-Galois,'' as the map is not \'etale at the ramification points. We drop
``generically'' in this paper.
\end{remark}

Given $Y$ with function field $k(Y)$, the $\Z/n$-covers of $Y$ are classified
according to whether $n$ is prime to $\mr{char}\,k=p$: if $p\nmid n$ the cover
is a Kummer extension $x^n=a$ with $a\notin k(Y)^{\times n}$; if $n=p$ it is an
Artin--Schreier extension $y^p-y=a$ with $a\notin\wp(k(Y))$, $\wp(b)=b^p-b$.
Covers of $\A^1$ correspond to covers of $\PP^1$ ramified over infinity, and we
switch between them freely. The two basic examples, and the two projections
our curve will carry:
\[\begin{tikzcd}
	X & {\A^1} & x && X & {\A^1} & x & {} \\
	Y & {\A^1} & {x^n} && Y & {\A^1} & {x^p-x}
	\arrow[from=1-1, to=1-2]
	\arrow["{\Z/n}"', from=1-1, to=2-1]
	\arrow["\lrcorner"{anchor=center, pos=0.125}, draw=none, from=1-1, to=2-2]
	\arrow["{(-)^n}", from=1-2, to=2-2]
	\arrow[maps to, from=1-3, to=2-3]
	\arrow[from=1-5, to=1-6]
	\arrow["{\Z/p}"', from=1-5, to=2-5]
	\arrow["\lrcorner"{anchor=center, pos=0.125}, draw=none, from=1-5, to=2-6]
	\arrow["{\mr{Fr}-\mr{Id}}", from=1-6, to=2-6]
	\arrow[maps to, from=1-7, to=2-7]
	\arrow["a"', from=2-1, to=2-2]
	\arrow["a"', from=2-5, to=2-6]
\end{tikzcd}\]

\subsubsection{Ramification Divisors and Algebraic Quotients of Curves}
\label{rammymammybranchywanchy}

We may consider $G$-actions on $X$ as equivalent to $G$-covers with source
$X$, the latter defined through the algebraic quotient; under this
equivalence, ramification points are the points with nontrivial stabilizer.

Let $X$ be a scheme of finite type over $k$ with a finite group $G$ acting by
$k$-automorphisms $\alpha_g$.

\begin{defn}\label{algquotient}
A {\color{Bittersweet}{quotient}} of $X$ by $G$ is a morphism
$f\colon X\to Y=:X/G$ which is $G$-invariant ($f\circ\alpha_g=f$ for all $g$)
and universal for this property: every $G$-invariant $X\to Z$ factors uniquely
through $f$. Its fibers are the $G$-orbits.
\end{defn}

\begin{lemma}\label{actionsarecovers}
Let $X$ be a smooth proper curve over $k=\bar k$ and $G$ finite. Every
$G$-action on $X$ induces a $G$-Galois branched cover of smooth proper curves
$f\colon X\to X/G$.
\end{lemma}
\begin{proof}
Under the equivalence between smooth proper curves and function fields of
transcendence degree one, a $G$-action on $X$ is a $G$-action on $\kappa(X)$;
the invariants $\kappa(X)^G$ again have transcendence degree one and
$\kappa(X)/\kappa(X)^G$ is $G$-Galois; the curve corresponding to
$\kappa(X)^G$ is $X/G$ and the extension is $f$.
\end{proof}

For a finite cover $f\colon X\to Y$ there is an exact sequence
$0\to f^*\Omega^1_Y\to\Omega^1_X\to\Omega^1_{X|Y}\to0$ (with
$\Omega^1_X\simeq f^*\Omega^1_Y$ exactly when $f$ is \'etale).

\begin{defn}
The {\color{Bittersweet}{ramification divisor}} of a dominant morphism
$f\colon X\to Y$ of curves is
$\mf R=\sum_P\mr{length}(\Omega_{X/Y})_P\,[P]$, the sum over closed points;
the {\color{Bittersweet}{branch divisor}} is its image $B:=f(\mf R)$.
\end{defn}

When $f$ is a tame Galois quotient with stabilizers $G_v$,
$\mr{length}(\Omega_{X/Y})_v=|G_v|-1=e_v-1$. When $f$ is wild the length is
larger than $e_v-1$ and is computed by the higher ramification groups: with
$G_i(v)$ the lower-numbered filtration,
\begin{equation}\label{eq:RH}
2g_X-2\ =\ |G|\,(2g_Y-2)\ +\ \sum_{v\in X}\ \sum_{i\ge0}\big(|G_i(v)|-1\big),
\end{equation}
the wild Riemann--Hurwitz formula \cite{stich} (a consequence of Hurwitz Genus Formula + Hilbert's difference formula). 
For a $\Z/p$-cover with a single break $m$ at $v$ the inner sum is
$(m+1)(p-1)$, which is the case of \eqref{eq:RH} used in the sequel.

\begin{lemma}[Riemann--Hurwitz, sheaf form]\label{riemannhurwitz}
For $f\colon X\to Y$ a finite map of curves with ramification divisor $\mf R$,
$$\Omega^1_X\ \simeq\ \OO_X(\mf R)\otimes f^*\Omega^1_Y.$$
\end{lemma}
\begin{proof}
By definition $\OO_{\mf R}\simeq\Omega^1_{X|Y}$, so
$0\to f^*\Omega_Y\to\Omega_X\to\OO_{\mf R}\to0$ is exact. Twisting by the line
bundle $\Omega_X^\vee$ gives
$f^*\Omega_Y\otimes\Omega_X^\vee\to\OO_X\to\OO_{\mf R}$, which is the
sequence $\OO_X(-\mf R)\to\OO_X\to\OO_{\mf R}$ of the divisor; hence
$f^*\Omega_Y\otimes\Omega_X^\vee\simeq\OO_X(-\mf R)$, and dualizing gives the
claim.
\end{proof}

A practical tool for ramification of plane curves is projection from a point.

\begin{construction}\label{makeamap}
For $X\subset\PP^2$, a point $x\in\PP^2$, and a line $L\subset\PP^2$, let
$\pi_x\colon X\to L$ send $p$ to $L(x,p)\cap L$, where $L(x,p)$ is the line
through $x$ and $p$.
\end{construction}

\begin{lemma}\label{lem:projdeg}
If $X$ has degree $d$, then $\deg\pi_x=d$ for $x\notin X$ and $\deg\pi_x=d-1$
for $x\in X$.
\end{lemma}
\begin{proof}
For $x\notin X$ a line through $x$ meets $X$ in $d$ points. For $x\in X$ the
map is defined on $X\smallsetminus\{x\}$ and extends over $x$ along the
tangent line; each line through $x$ now meets $X$ in $d$ points \emph{one of
which is $x$ itself}, so the degree drops by one.

\begin{center}\includegraphics[width=8cm]{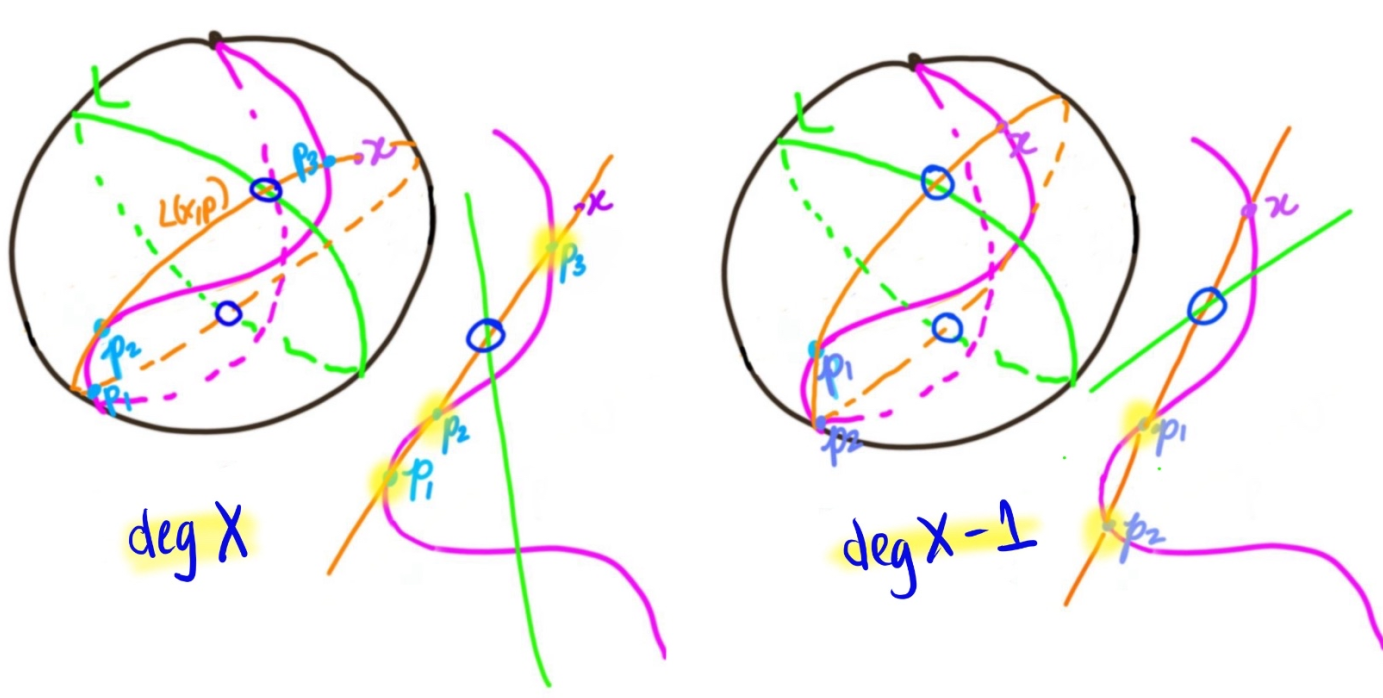}\end{center}

\end{proof}

\begin{defn}[Ramification break]
Let $L/K$ be a finite Galois extension of local fields with Galois group
$G = \operatorname{Gal}(L/K)$, valuation ring $\mathcal{O}_L$, and
uniformizer $\pi_L$. For $i \geq -1$ the {\color{Bittersweet}{lower-numbering ramification groups}} are
\[
  G_i := \bigl\{\, \sigma \in G \big|
    v_L\bigl(\sigma(\pi_L) - \pi_L\bigr) \geq i + 1 \,\bigr\},
\]
a decreasing filtration
$G = G_{-1} \supseteq G_0 \supseteq G_1 \supseteq \cdots$ with
$G_i = 1$ for $i \gg 0$. An integer $b \geq 0$ is a
{\color{Bittersweet}{ramification break}} (or {\color{Bittersweet}{jump}}) of $L/K$ in the lower
numbering if
\[
  G_b \;\neq\; G_{b+1}.
\]
Equivalently, writing $\varphi_{L/K}$ for the Herbrand function and
$G^u := G_{\psi_{L/K}(u)}$ for the upper-numbering filtration, a real
number $u \geq 0$ is a break in the upper numbering if
$G^{u} \neq G^{u+\varepsilon}$ for all $\varepsilon > 0$; the two sets
of breaks correspond under $u = \varphi_{L/K}(b)$.
\end{defn}

\subsubsection{Wild ramification: breaks, Herbrand, and Artin--Schreier--Witt towers}
\label{sec:wild}
The tame material above suffices for height $p-1$. At height $\varphi(p^k)$ with $k \geq 2$ the wild part of $G'$ is cyclic of order $p^k$ and the following standard facts are used \cite{higherlevels}; proofs are in \cite[Ch.\ IV]{Ser79}.

\begin{defn} Let $L/K$ be a finite Galois extension of local fields with group $G$ and lower-numbering filtration $G_i$ as in Definition \ref{sec:coversram}. The \emph{Herbrand function} is $\varphi_{L/K}(u) = \int_0^u \frac{dt}{[G_0 : G_t]}$, with inverse $\psi_{L/K}$, and the \emph{upper-numbering} filtration is $G^{v} := G_{\psi(v)}$. The upper numbering is compatible with quotients: $(G/N)^v = G^vN/N$. \end{defn}

\begin{lemma}[Hasse--Arf, conductor--discriminant, Schmid] \label{lem:hasseArf} Let $L/K$ be a totally ramified cyclic extension of local fields of characteristic $p$ with group $\Z/p^k$, and let $u_1 \le \dots \le u_k$ be its upper breaks, $u_j$ the break of the layer $L^{p^{k-j}}/L^{p^{k-j+1}}$.
\begin{enumerate}
\item[(i)] (Hasse--Arf) The $u_j$ are integers \cite[IV.3, V.7]{Ser79}.
\item[(ii)] (conductor--discriminant) A character of exact order $p^j$ has Artin conductor $u_j+1$, and since there are $p^j-p^{j-1}$ of them, the different of $L/K$ has valuation $v_L(\mathfrak{D}_{L/K})=\sum_{j=1}^k (p^j-p^{j-1})(u_j+1)$ \cite[VI.3]{Ser79}.
\item[(iii)] (Schmid) $u_{j+1}\ge p\,u_j$, and if $u_{j+1}>p\,u_j$ then $p\nmid u_{j+1}$ \cite{thomasASW}.
\end{enumerate}
In particular, among $\Z/p^k$-towers whose first break is $u_1$, the different is minimized exactly when $u_{j}=p^{j-1}u_1$ for all $j$. \end{lemma}

\begin{lemma}[wild Riemann--Hurwitz] \label{lem:wildRH} For a $G$-Galois cover of curves $q \colon X \to Y$ over $k$, $$2g_X - 2 = |G|(2g_Y - 2) + \sum_{x \in X} \sum_{i \geq 0} (|G_{x,i}| - 1),$$ where $G_{x,i}$ are the lower ramification groups at $x$. For $G = \Z/p$ totally ramified at a single point with break $m$ (so $G_0 = \dots = G_m = G$, $G_{m+1} = 1$), the local contribution is $(p-1)(m+1)$. The right-hand side is monotone in the breaks, so a bound on the breaks from below bounds the genus from below. \end{lemma}

\begin{defn}[Artin--Schreier--Witt towers] \label{def:ASW} Let $W_k$ denote Witt vectors of length $k$ and $\wp = F - 1$. For $f \in W_k(\OO(U))$ on an affine open $U$ of a curve over $k$, the equation $\wp(y_0, \dots, y_{k-1}) = f$ defines a $\Z/p^k$-Galois cover of $U$, with $\sigma$ acting by Witt addition of $(1, 0, \dots, 0)$; every $\Z/p^k$-cover of $U$ arises this way, and the covers with $f$ and $f'$ coincide iff $f - f' \in \wp W_k(\OO(U))$ \cite[Ch.\ IV]{Ser79}, \cite[\S 4]{garuti}. The intermediate quotient by $\sigma^{p^{j}}$ is the tower $\wp(y_0, \dots, y_{j-1}) = (f_0, \dots, f_{j-1})$; the $j$-th layer is an Artin--Schreier cover $y_j^p - y_j = (\text{Witt carry in } y_0, \dots, y_{j-1}) + f_j$. The upper breaks of the tower at a pole of $f$ are computed layer by layer from the pole orders of the carries after reduction (``reduced coordinates''); see \cite[Lemma 1.2]{higherlevels} for the towers $\wp(y_\bullet) = (z^{p-1}, 0, \dots, 0)$. \end{defn}

\subsection{Pushforwards Along Stacky Curves, and From Group Actions to Branch Points}
\subsubsection{Analysis of Pushforwards of Stacky Curves}
Let $X$ be a curve with an action of a finite group $G$. We will quickly set up some foundational lemmas about stacky curves to work with the map $q^G_*: \mr{QCoh}^G(X) \to \mr{QCoh}(X/G)$ induced by the map $\pi: X \stackyq G \to X/G.$ For a thorough foundational treatment, please consult Section 4 of \cite{kobin}. 

\begin{remark} The floor functions in Kobin \cite{kobin} and Bertin-Mezard \cite{bm} Section 4 should be replaced by ceiling functions, as noted in Remark 1.6.2 of \cite{katocornel}.\end{remark}

Let $\mc{Y}$ be a stacky curve, and $Y$ be its coarse counterpoint, in our case, $\mc{Y} := \mc{X} \stackyq G$, and $Y := X/G.$ Consider the map $h: \mc{Y} \to Y$. 

\begin{defn} \label{stackydegree} The degree of a divisor $D = \sum_P n_P[P]$ on a stacky curve is the formal sum $\deg(D) = \sum_P \frac{n_P}{|G_P|}$ where $|G_P|$ is the degree of the stabilizer group at $P$. \end{defn}

\begin{lemma} \label{koby}  \cite[Lemma 4.10]{kobin} Given a stacky curve $\mc{X}$ with coarse space morphism $\pi \colon \mc{X} \to X$, there's an isomorphism $$\pi_*\OO_{\mc{X}}(D) \to \OO_X(\lceil D \rceil).$$  \end{lemma}

The following result is easily extended to wild cases and cases not totally ramified, we discuss the tame totally ramified case for cleanliness of exposition.

\begin{lemma} \label{diamondbeneath} (diamond beneath) 
For $X$ a totally ramified curve with a tame action of $G$, let $\mf{R}$ be the unreduced ramification divisor, let $B$ be the the reduced branch divisor (i.e., image of the points with nontrivial stabilizer groups $G_x$ under the quotient map $q: X \to X/G$). Then, $$q_*^G\OO_X(-\mf{R}) \simeq \OO_Y(-B).$$
\end{lemma}

\begin{proof} The pushforward $q_*^G\OO_X(-\mf{R})$ is considering $\OO_X(-\mf{R})$ as a $G$-module, which is equivalent to considering $\OO_{X\stackyq G}(-\mf{R})$ under the equivalence $\QCoh^G(X) \simeq \QCoh(X \stackyq G)$. Let $\beta_P$ denote the local different at $P$ (for a tame curve $\beta_P = |G_P|-1$)

By Lemma \ref{koby}, $\deg \beta_P [P] = \frac{\beta_P}{|G_P|}$, thus $$q_*^G\OO_X(-\mf{R}) \simeq \OO_Y(-\sum_{P} \lceil \frac{\beta_P}{|G_P|} \rceil q(P)) \simeq \OO_Y(-\sum_{P} \lceil \frac{|G_P|-1}{|G_P|} \rceil q(P)).$$
The conclusion follows, as $\sum_P q(P)$ is the reduced branch divisor $B$ by definition. 
\end{proof}

\begin{remark} The statement of \ref{diamondbeneath} is usually written in the literature as $q_*^G\OO_X(-\mf{R}) \simeq \OO_{X/G} \cap q_*(\OO_X(-\mf{R}))$, where the ceiling function is listed as the next step in the evaluation of the right hand side (see \cite{katocornel} Prop 1.6, and \cite{bm} Prop 5.3.2). They are implicitly using stacky gerbes to calculate this cap product as a ceiling function. \end{remark} 

\begin{lemma} \label{dropweight} (drop weight) For $X$ a curve with a totally ramified tame action of $G$, let $B$ be the the reduced branch divisor (i.e., image of the points with nontrivial stabilizer groups $G_x$ under the quotient map $q: X \to X/G$). Then,
$$q_*^GT_X \simeq T_{X/G}(-B)$$
\end{lemma}

\begin{proof} 
By taking the dual of Riemann-Hurwitz, Lemma \ref{riemannhurwitz},
$$T_X \simeq q^*T_Y \otimes \OO_X(-\mf{R})$$ where $\mf{R}$ is the (unreduced) ramification divisor. Applying $q_*^G,$ we get 
$$q_*^GT_X \simeq T_Y \otimes q_*^G\OO_X(-\mf{R}),$$
so our anaysis is reduced to considering $q_*^G\OO_X(-\mf{R}).$ By lemma \ref{diamondbeneath}, $q_*^G\OO_X(-\mf{R}) \simeq O_Y(-B)$ and the result follows.
\end{proof}

\begin{cor} \label{kernelsanders} (kernel sanders)
Let $Y := X/G$, and let $B$ be the reduced ramification divisor of the map $q: X \to X/H$, then $$H^1(Y, q_*^GT_X) \simeq H^1(Y, T_Y(-B))$$ and $\dim H^1(Y, q_*^GT_X) = 3g_{Y} - 3 + |B|.$
\end{cor}

\begin{proof} 
The equivalence $H^1(Y, q_*^GT_X) \simeq H^1(Y, T_Y(-B))$ is an immediate corollary of Lemma \ref{dropweight} which states that $q_*^GT_X \simeq T_Y(-B)$. The dimension count follows from Riemann-Roch applied to $H^1(Y, T_Y(-B))$.
\end{proof}

The totally ramified assumptions in the lemma below can be removed by considering the more general case of Lemma \ref{diamondbeneath} and weighting $B$ accordingly. 

\begin{theorem} \label{grouptobranch} (group to branch) For $X$ a smooth curve over a field $k$ with a tame action by a group $G$ which is totally ramified, and $\Ext^2_G(\Omega_X, \OO_X) \simeq 0$, let $B$ be the set of branch points of the quotient map $q_*: X \to X/G$. Then, the functor $L$ is an isomorphism of formal moduli problems
$$L \colon \Def_{(X, G)} \simeq \Def_{(X/G; B)}$$
\end{theorem}

\begin{proof} This reduces to showing that the map $L$ defined in Defn \ref{LightYagami} induces an isomorphism on tangent spaces of deformations. When $G$ is a totally ramified tame action, $H^0(Y, R^1q_*^GT_X) \simeq 0$ by Lemma \ref{moodstabilizer}, thus $$H^1(Y, q_*^GT_X) \simeq \Ext^1_G(\Omega_X, \OO_X).$$ It remains to show that the map $\Ext^1_G(\Omega_X, \OO_X) \to H^1(Y, q_*^GT_X)$ induced by the isomorphism of the exact sequence is the map on tangent spaces $\delta(L)$. This follows from combining \cite{hurwitz} Theorem 5.5 and \cite{wewers} Prop 4.10.
\end{proof}

\subsection{Configuration Spaces as Symmetric Group Representations}
We analyze the symmetric group action on configuration stacks of points on $\PP^1$; these represent the marked deformation problems $\Def_{(\A^1; n)}$ of Section 
ef{gstacks}.

\begin{defn} \label{defn:configspace} We define the moduli stack $\mr{Conf}(X, k)$ for an object $X$ with a topology as 

$$\mr{Conf}(X, k) := \{ (x_1, ..., x_n) \in X^k | x_i \neq x_j, i \neq j \}.$$

This is an open space of $X^k$ obtained removing the closed diagonal subspaces 

$$\Delta_{i, j} :=  \{ (x_1, ..., x_n) \in X^k | x_i = x_j\}.$$

Each tuple is called an ordered configuration, and each entry is a point of the configuration.
\end{defn} 

\begin{defn} The symmetric group of $k$-letters acts on $\mr{Conf}(X, k)$ by permuting the points of the configuration. If $\sigma \in \Sigma_k$ and $(x_1, ..., x_k) \in \mr{Conf}_k(X)$ then $\sigma(x_1, ..., x_k) = (x_{\sigma(1)}, \cdots, x_{\sigma(k)}).$
\noindent The orbit space is denoted by 
$$C_k(X) = \{ \{x_1, ..., x_k \}| x_i \in X, x_i \neq x_j\}.$$
The symmetric group action is free and properly discontinuous, so the quotient map $\mr{Conf}_k(X) \to C_k(X)$
is a regular covering space of degree $k! = \Sigma_k$. 
\end{defn} 

The affine group $\G_a \rtimes \G_m \simeq \mr{Stab}_{\infty}(\bb{P}^1) \simeq \Aut(\A^1)$ is the subgroup of $\Aut(\P^1)$ fixing infinity (thus acting only 2-transitively), thus, $$M_{0, n+1} := \mr{Conf}(\bb{P}^1, n+1) \stackyq \Aut(\bb{P}^1) \simeq \mr{Conf}(\bb{A}^1, n) \stackyq \Aut(\bb{A}^1).$$ 

Considering the action of $\Sigma_n$ by permuting marked points, an $n$-marked stack represented by a ring becomes a $\Sigma_n$-representation. Let's quickly review a representation we will be using intensively. 

\begin{defn} \label{standard} The group $\Sigma_n$ acts on $R^n$ by permuting basis vectors; this defines a representation called the permutation representation. This representation has a $1$-dimensional invariant subspace, spanned by the vector $e_1 + e_2 + \cdots + e_n$ which is the trivial representation. A complementary subspace to this is $$V = \{a_1e_1 + \cdots + a_ne_n | a_1 + \cdots + a_n = 0 \}.$$ This is called the {\color{Bittersweet}{standard representation}} or irredudible $n-1$ dimensional representation $\lambda$ of $\Sigma_n$.
\end{defn} 

\subsubsection{Configuration Space as a $\Sigma_n$ Representation} 
\label{sec:config}
Let's consider the action of $\Sigma_n$ on the group ring $R[C_n]$ considered as a module $V := R\{g_1, ..., g_n\}$ with basis $B := \{g_1, ..., g_n\}$. Applying $\Sym(V)$ gives a polynomial ring with generators our basis $B.$ 

\begin{theorem} \color{blue}{ (simple gods) \label{simplegods}  (i) Let $\Delta$ be the fat diagonal, that is, all 2-planes defined by $x_i=x_j$ for $i \neq j$.  $$\mr{Conf}(\A^1, n) \simeq (\Sym(\A^1) - \Delta) \stackyq (\G_a \rtimes \G_m)$$
\noindent (ii) $\mr{Conf}(\A^1, n)$ is represented by the ring $\Sym(\lambda)[N^{-1}] \stackyq \G_m$ as a $\Sigma_n$-representation, where $\lambda$ is the standard representation $\lambda := W(k)\{y_0, .., y_{n-1}\}/(y_1 + \cdots y_{n-1})$ with $|y_i| = -2$, and $N=\prod_{i = 0}^{n-1} \sigma^i(y_0),$ where $\sigma$ generates $C_n$.} 
\end{theorem}

\begin{remark} Specifying the $\G_m$ action is equivalent to specifying a grading on $\Sym(\lambda)[N^{-1}]$. We could consider $\mr{Conf}(\A^1, n)$ with a different grading say, $\G_m$ acts by $\mu^{-k}$ rather that $\mu$ itself. In this case, we'd have $\Sym(\lambda)$ is a graded ring where $|x_i| = -k$. \end{remark}

\begin{proof} 
\noindent (i) Choosing $n$ points on $\A^1$ is equivalent to an $n$-tuple, that is a point of $\A^n \simeq \Sym(\A^n)$. To make sure points do not coincide, we remove the fat diagonal, which is the removal of all 2-planes defined by $x_i=x_j$ for $i \neq j$. We wish to consider our states up to automorphisms of $\A^1$ which are and $\Aut(\A^1) \simeq \G_m,$ where $\G_a$ is the translation action, and $\G_m$ is the scaling action, which acts on $\A^n$ by scaling all coordinates equally. \newline
\noindent (ii)
The ring representing $\Sym(\A^1)-\Delta$ is $Y := R[x_1, ..., x_n][N^{-1}]$, where $N := \prod_{i \neq j} (x_i - x_j)$. When we quotient by $\G_a$ on the stack, this translates to forcing translation invariance on functions. On coordinates this gives us
$$\widetilde{Y} := Y^{\G_a} := (R[x_1, ..., x_n][N^{-1}])^{\G_a} \simeq R[x_2-x_1, ..., x_n-x_1][N^{-1}].$$ Note that the latter is a basis of the reduced permutation representation. \newline 
\noindent By choosing the additive basis $x_i-x_{i+1}$ for the standard representation $\sigma$, this gives a map between the two rings,
$$\Sym(\lambda)[N^{-1}] \to \widetilde{Y}.$$
After change of basis, $N^{-1} := \prod \sigma^i(y_0)$. It remains to endow the left with an action of $\G_m$ and show that the map is $\G_m$-equivariant. The scaling action $x_i \mapsto \mu x_i$, coincides with the action on the right. This concludes the demonstration of their equivalence.
\end{proof}

\begin{remark} For reference, we include the projective version. For $n \geq 3$,
$$\mr{Conf}(\PP^1, n) \simeq (\PP^1 - \{ 0, 1, \infty \})^{n-3} - \Delta,$$
\noindent where $\Delta = \{(q_i)_i \colon \exists i \neq j \text{ with } q_i = q_j \}$ is the fat diagonal. The meat of why this holds comes from the isomorphism $(\PP^1)^3 - \Delta \simeq \mr{PGL}_2$.\end{remark} 

\begin{defn} \label{reduced} The reduced regular representation of $C_n$ is defined as $\overline{\rho} := \ker(R[C_n] \to R)$. This inherits an action of $\Aut(C_n),$ we also call the $\overline{\rho}$ corresponding representation of $C_n \rtimes \Aut(C_n).$
\end{defn}

\begin{remark} The restriction of the standard $\Sigma_n$-representation $\lambda$ to $\Z/n \rtimes \Z/(n-1)$ is $\overline{\rho}.$\end{remark}

\begin{cor} \label{set} Given a set of points $P := \{p_1, ..., p_n\}$ in $\A^1$, $$\Def{(\A^1; n)} \simeq \mr{Conf}(\A^1, n)^{\wedge}_P$$ is representable by the graded ring $$\Sym(\lambda)[N^{-1}]^{\wedge}_{I} \stackyq \G_m$$ as a $\Sigma_n$-representation, where $I = (x_1-p_1, ..., x_n-p_n).$
\end{cor}

\subsubsection{Symmetric Group Action on Cohomology of Configuration Spaces}
\label{sec:symmetriccohom}
This computation is what fixes the grading convention $|y_i| = -2$ when configuration spaces model graded Lubin--Tate actions (Remark after Lemma \ref{lem:gradedungraded}). 

Considering the action of $\Sigma_n$ by permuting marked points, the cohomology of any of stack of $n$-marked objects becomes a $\Sigma_n$-representation. Let us consider the cohomology of $\mr{Conf}(\A^1, n)$ as a $\Sigma_n$-representation. 

\begin{lemma} $$H^*(\mr{Conf}(\A^1, n)(R)) \simeq \begin{cases} \text{triv} & * = 0 \\ 
\lambda & * = -2 \\
0 & \text{ else} \end{cases}. $$
\noindent where $\lambda$ is the irreducible dimension $n-1$ representation from definition \ref{standard}. \end{lemma}

\begin{proof} Recall, the affine group $G := \mathbb \G_a \rtimes \mathbb \G_m$ is the subgroup of $\Aut(\PP^1)$ fixing the point $\infty \in \mathbb P^1$. As stated earlier, it follows that the quotient $\mathrm{Conf}(\A^1,n) \stackyq G$ is equal to the moduli space $M_{0,n+1}$ of $(n+1)$ points on the projective line modulo symmetries. In other words, $$M_{0, n+1} := \mr{Conf}(\bb{P}^1, n+1) \stackyq \Aut(\bb{P}^1) \simeq \mr{Conf}(\bb{A}^1, n) \stackyq \Aut(\bb{A}^1).$$ 

This makes $\mathrm{Conf}(\A^1,n)$ homotopic to a trivial $\G_a \rtimes \G_m$-bundle over $M_{0,n+1}$ (in char 0, this would be a trivial circle bundle since $\G_a$ is contractible, in general this is a $\G_m$ bundle), and so
$$H^\bullet(\mathrm{Conf}(\A^1,n)) \cong H^\bullet(M_{0,n+1}) \oplus H^{\bullet-1}(M_{0,n+1}).$$

Only needs to construct such an $\Sigma_n$-equivariant factorization for the cohomology $H^\bullet(M_{0,n+1})$. There is a locally trivial $\Sigma_n$-equivariant fiber bundle $$M_{0,n+1} \to M_{0,n}$$ by forgetting the last marking such that each fiber $F$ is $\mathbb P^1$ minus $n$ points \cite{FadNeu}(Theorem 3). This is because the fiber over a point $Q_n \in M_{0,n}$ is exactly the choices of an $n+1$th point distinct from $\Q_n$, which is $\bb{P}^1-\Q_n$.
By excision $H^0(F) \cong \text{triv}$, and $H^1(F) \cong \lambda$ as $\Sigma_n$-representations. The Leray-Serre spectral sequence for $M_{0,n+1}$ degenerates shifting degree, and this gives the claimed result. 
\end{proof}

\begin{remark} In the notation of \cite{FadNeu}, $F_{0,n}(\A^1) \stackyq \Aut(\A^1) \simeq F_{2, n-2}(\A^1-Q_2)$. Note that $F_{2, n-2}(\A^1- Q_2)$, that is, the configuration space of $\A^1$ minus 2 points is exactly the $j$ line.\end{remark}

\subsubsection{Surgeries: deformations of a pointed curve as a double coset}
\label{sec:surgeries}
Let $Y$ be a smooth proper curve over $k$ and $B \subset Y$ a finite set of closed points containing a distinguished point $Q_\infty$, with $B_0 := B \smallsetminus \{Q_\infty\}$; write $U := Y \smallsetminus B_0$ and choose a uniformizer $z_P$ at each $P \in B_0$. Throughout, $R$ is an Artinian local $W$-algebra with residue field $k$ and maximal ideal $\fm$. The following presentation of $\Def_{(Y; B)}$ (Definition \ref{localgstacks}) by regluing formal disks is Beauville--Laszlo's lemma \cite{BL95} read on curves rather than bundles; it is used in \cite{higherlevels} for the Artin--Schreier--Witt curves $Y_k$ with $B_0 = W_k(\F_p)$, and at $k=1$, when $Y = \PP^1$ and $B_0 = \F_p$, its order-zero part is the configuration space of Theorem \ref{simplegods}. We state it for a $\bar G$-action on $(Y, B)$ with $\bar G$ acting on the disks through automorphisms $u_{g,P}$ of the coordinates.

\begin{defn}[loop groups]\label{def:loop}
$\cA(R)=\prod_{P\in B_0}\{g_P=z_P+\varepsilon_P:\ \varepsilon_P\in\fm
R((z_P))\}$, composition by substitution. $\cA^{in}$: $\varepsilon_P\in
\fm z_PR[[z_P]]$. $\cA^{(m)}$: $\varepsilon_P\in\fm z_P^{-m}R[[z_P]]$.
$\cA^{[m]}=\langle\cA^{(m)},\cA^{in}\rangle$, the generated subgroup.
$\cA^{out}(R)$: automorphisms of $U\otimes R$ reducing to the identity
and fixing $Q_\infty$, embedded by Laurent expansion at each $P$.
\end{defn}

The subgroups must be generated, not merely filtered: composing
$z+az^{-m}$ with $z+bz^{-n}$ produces $-mab\,z^{-(m+n+1)}$ at
$\fm$-depth two. Poles spill.

\begin{theorem}[Beauville--Laszlo: glue and present]\label{thm:BL}
\color{blue}{Gluing $\OO(U\otimes R)$ to the disks $R[[z_P]]$ along $g\in\cA(R)$ is
an equivalence $[\cA^{out}\backslash\cA/\cA^{in}]\simeq\Def(Y_k,B)$,
strictly $\bar G$-equivariant for the twisted permutation action
$(\bar g\cdot h)_{\bar gP}=u_{g,P}\circ h_P\circ u_{g,P}^{-1}$ with
(in \cite{higherlevels}: $u_{\tau,P}(z)=\theta z$ and $u_{\sigma,P}=\mathrm{id}$). At first order
$\cA/(\cA^{out}+\cA^{in})=M_k:=H^1(Y_k,T_{Y_k}(-B))$, and at $k=1$ a
surgery of order $0$, $z_P\mapsto z_P+\nu_P$, moves $P$ to $P+\nu_P$.}
\end{theorem}
\begin{proof}
Glue along $g$ \cite{BL95}; conversely choose disk coordinates
(ambiguity $\cA^{in}$) and an outer trivialization ($H^1(U,T_U)=0$,
ambiguity $\cA^{out}$). Stabilizers $\cA^{out}\cap g\cA^{in}g^{-1}$
vanish at first order by $H^0(T(-B))=0$, and then at every $\fm$-depth
by induction. The coordinates are exactly compatible:
$\bar\tau^*z_{\bar\tau P}=\theta z_P$ and $\bar\sigma^*z_{\bar\sigma P}
=z_P$, because Witt addition of $(1,0,\dots)$ is plain addition in
$\tilde y_0$. For the curves of \cite{higherlevels}, the strictification of $\bar\sigma$ over $W$ is obstructed by the seam class $-p$: the stack is equivariant although none of its atlases is.
\end{proof}

\section{Toolkit: Representation-Theoretic Lemmas}
\label{sec:reptoolkit}
\subsection{Conventions on Pullback and Pushforward}
\begin{conv}[upper star, lower star, and the other star]\label{conv:stars}
Throughout, $f\colon Y\to Z$ is an isomorphism of curves ($g'$ on $X$, $\rho(g)$ on $\mc X$, $\iota\colon\mc X_s\to X$, or $\phi\colon\mc X\to\mc X'$).
\begin{itemize}
\item $f^*$ is pullback: on functions, differentials, cohomology, line bundles ($\Pic^0(Z)\to\Pic^0(Y)$), and vector fields ($f^*\vartheta=(df)^{-1}\vartheta f$, compatible with contraction). It reverses composition. Every character superscript is a pullback eigenspace, $V^\chi=\{v: g^*v=\chi(g)v\}$.
\item $f_*=(f^{-1})^*$ is pushforward, used only on $\Pic^0$ and its $p$-divisible group. It preserves composition, so $g'\mapsto g'_*$ is a left action of the nonabelian $G'$; that is why Jacobians carry $f_*$, not $f^*$. When $f_*$ commutes with $G$, $e_1(f_*)$ is its restriction to the $e_1$-summands.
\item The $\star$ of $\Def^\star$ means rigidified (Section \ref{gstacks}); a subscript $s$ means special fiber.
\end{itemize}
Since $\mr{AJ}\circ g=g_*\circ\mr{AJ}$ up to translation, $g_*$ acts as $g^*$ on $\omega_{\Pic^0}=H^0(\Omega^1)$ and $H^1_{\mr{dR}}$, and as $(g^{-1})^*$ on $\mr{Lie}\,\Pic^0=H^1(\OO_X)$. Consequently, for a block $F$ of $\Pic^0(X)[p^\infty]$ on which a central $G$ acts through a character $\psi$, $\omega_F = H^0(\Omega^1)^\psi$ and $\mr{Lie}(F) = H^1(\OO_X)^{\psi^{-1}}$: the Hodge line of the block is the $\psi$-eigenspace of differentials for the \emph{pullback} action.
\end{conv}

\subsection{Idempotent Decompositions}
\label{idempotent}

Our analysis of the formal group law $\Jac(X)^\wedge_e$ relies on constructing a splitting of this formal group law into smaller dimensional pieces. This is originates from a more general construction which we briefly outline. 

\begin{construction} \label{idempotentdecompconstruction}
Any $p$-complete additive category $Y$ with a $G$-action splits due to a splitting of the source of the following map,

$$k[G] \hookrightarrow \End(Y),$$

For $S_1, ..., S_n$ representatives of the isomorphism classes of irreps of $G$, there is an isomorphism of rings. 

$$k[G] \simeq \End_k(S_1) \times \cdots \times \End_k(S_n)$$

If we take $G$ to be coprime to $k$, thus the representatives are characters of $G$. Then for  $i = 1, ..., n$ the element $e_i = f^{-1}(id_{\End_k(S_i)})$ is a central idempotent (projector) in $k[G].$ Further, we can write $e_i$ as

$$e_i = \frac{1}{|G|} \sum_{g \in G} \chi_{S_i}(g)g^{-1}.$$

In other words, 
$$k[G] \simeq \bigoplus_{i=1}^n e_i Z_p[G].$$

This induces a splitting on $Y$, 
$$Y \simeq \bigoplus_{i=1}^n e_i Y.$$ 
\end{construction}

\begin{remark} This also works for an object $Y$ in an idempotently complete category. Let's go Maschke's theorem. \end{remark}

\begin{lemma} \label{equivarianceofidempotent} Consider a group $G'$ such that $G \subseteq Z(G')$ lies in its center. Let $\mc{C}$ be an additive category with a $G'$ action. 

Let $\mc{I}^i$ be the functor defined by projecting onto the $i$th idempotent component of a $G$-equivariant idempotent decomposition of $\mc{C}$, as in Construction \ref{idempotentdecompconstruction}.
$$\mc{I}^i \colon \mc{C} \to e_i\mc{C}$$

\noindent Then, $\mc{I}^i$ is $G'$-equivariant. 
\end{lemma}

\begin{proof} 
This is because $G$ sits inside of the center of $G'$, which means that for any $g \in G'$, $g(e_i\mc{C}) = e_ig(\mc{C})$.
\end{proof}

\subsection{Tate Cohomology of $C_p$ on Symmetric Algebras}

\begin{notation}
We establish some quick notation. We call $\rho := \bb{W}[C_p]$ the regular representation of $C_p$, and $\overline{\rho}$ the reduced regular representation. That is, $\overline{\rho}$ is the kernel of the augmentation map $\epsilon: \bb{W}[\F_p] \to \bb{W}$. 

Let $\Lambda := \Sym(\overline{\rho})$ and $A := \Sym(\rho)$. We consider the element $s_1 := ([1]+[\sigma]+\cdots+ [\sigma^{(p-1)}])(y_0) = y_0 + \cdots + y_{p-1}$ in $A$, and call $s_1A$ the ideal that $s_1$ generates in $A$. 
\end{notation}

The short exact sequence we will be working with is the following:
$$s_1A \to A \to A/s_1A \simeq \Lambda$$

We will now proceed to go through several lemmas that we will piece together to get the Tate $C_p \rtimes C_{(p-1)^2}$-cohomology of $\Lambda$ in the sequel.

\begin{lemma} \label{tateC_pA} The Tate cohomology of the $C_p$ action on $A$ is of the form $$\hat{H}^*(C_p, A) \simeq \Z/p[b^{\pm}, d],$$ \noindent where $b$ is in degree $(2, 0)$ and  $d$ is in degree $(0, -2p)$. \end{lemma}

\begin{proof}
The orbit of each monomial of $A$ under $\sigma$ is free except for $d := y_0 \cdots y_{p-1}$ which is fixed (thus the orbit is trivial), and its corresponding cohomology class $d$ lies in degree $(0,|y_i|p)=(0, -2p)$. Therefore, $A$ splits as a sum of a $G$-module $F$ with a free $C_p$-action and $\bb{W}[d]$ which has trivial $C_p$-action, $$F \oplus \bb{W}[d].$$ Thus, $H^q(C_p, \bb{W}[d]) \simeq H^q(C_p, A).$ 

Next, we take the ring map from $\bb{W} \to A$, this gives a map from $H^*(C_p, \bb{W})$ to $H^*(C_p, A)$. Recall that $$H^*(C_p, \bb{W}) = \begin{cases} \bb{W}& \text{if } * = 0 \\ \bb{W}/p & \text{if } * \text{ even} \\ 0 & \text{if } * \text{ odd}\end{cases},$$ i.e., $H^*(C_p, \bb{W}) \simeq \bb{W} \oplus \bb{W}/p[b]$, where $b$ has bidegree $(2, 0)$.

Let $N: A \to H^0(C_p, A)$ be the norm map. Then, we have an exact sequence $$A \xrightarrow{N} H^0(C_p, A) \to \bb{W}/p[b,d] \to 0.$$ 
The Tate cohomology of $A$ is then $$\hat{H}^*(C_p, A) \simeq H^*(C_p, A)[b^{-1}] \simeq \Z/p[b^{\pm}, d].$$
\end{proof}

\begin{lemma} \label{tateC_plambda} The Tate cohomology of the $C_p$ action on $\Lambda[d^{-1}] := \Sym(\overline{\rho})[y_0 \cdots y_{p-1}] $ is $$\hat{H}^*(C_p, \Lambda)[d^{-1}] \simeq \Z/p[b^{\pm 1}, c, d^{\pm 1}]/(c^2),$$ \noindent where $b$ is in degree $(2, 0)$, $d$ is in degree $(0, -2p)$, and $c$ is in degree $(1,|y_i|)=(1, -2).$
\end{lemma}

\begin{proof}
Let $\widetilde{b} := s_1 b$, and $\widetilde{d} := s_1 d$. Similarly, the Tate cohomology of $s_1A$ is $\hat{H}^*(C_p, s_1A) \simeq \Z/p[\widetilde{b}^{\pm}, \widetilde{d}].$

From our short exact sequence we get a long exact sequence, which is zero at the ends for degree reasons: 

$$0 \to \hat{H}^{2k-1}(C_p, \Lambda) \to \hat{H}^{2k}(C_p, s_1A) \to \hat{H}^{2k}(C_p, A) \to \hat{H}^{2k}(C_p, \Lambda) \to 0.$$

The middle map in this long exact sequence is zero, because it is induced by multiplication by $s_1$, which is in the image of the additive norm on $A$, and norms are modded out by Tate cohomology. 

It follows that $$\hat{H}^*(C_p, \Lambda) \simeq \Z/p[b^{\pm 1}, c, d]/(c^2)$$ where $c$ is the element of bidegree $(1,|s_1|)=(1, |y_i|)=(1, -2)$ which maps to $\widetilde{b} \in \hat{H}^2(C_p, s_1A)$.

The cohomology of $H^*(C_p,\Lambda[d^{-1}]) \simeq H^*(C_p,\Lambda)[d^{-1}]$ because finite group cohomology commutes with colimits.
\end{proof}

\begin{lemma} \label{tatefullreduces} Let $G' := C_p \rtimes C_{(p-1)^2}$. Let $R$ be a ring in which $(p-1)^2$ is a unit. For a $G$-$R$-module $M$, the spectral sequence collapses and $$H^*(G, M) \simeq H^*(C_p, M)^{C_{(p-1)^2}}.$$ 

\end{lemma}

\begin{proof}
Given a $G$-$\bb{W}$-module $M$, the short exact sequence $C_p \xrightarrow{\triangleleft} G \to C_{(p-1)^2}$ gives rise to the spectral sequence 
$$H^p(C_{(p-1)^2}; H^q(C_p; M)) \Rightarrow H^{p+q}(G; N).$$ 
The transfer map in group cohomology yields a commutative diagram
\[
\begin{tikzcd}
{H^p(C_{(p-1)^2}, H^q(C_p, M)) } \arrow[rd, "\mr{res}"'] \arrow[rr, "(p-1)^2"] &                                                  & {H^p(C_{(p-1)^2}, H^q(C_p, M)) } \\
                                                                               & {H^p(1, H^q(C_p, M)) } \arrow[ru, "\mr{cores}"'] &                                 
\end{tikzcd}
\]
Because $(p-1)^2$ is a unit in $\bb{W}$, the horizontal map is an isomorphism and the spectral sequence collapses.
\end{proof}

\begin{remark} Note that the same holds for Tate cohomology. \end{remark}

The last lemma of this subsection says when group cohomology ignores a completion. It is used in the sequel to pass from the Lubin--Tate ring, a completed and localized symmetric algebra, back to the symmetric algebra itself.

\begin{lemma} \label{lem:idealhop} \color{blue}{(ideal hop) Let $G$ be a finite group, $\Lambda$ a commutative ring with $G$-action, and $I \subseteq \Lambda$ a $G$-stable ideal, with $\Lambda^\wedge_I := \lim_k \Lambda/I^k$.
\begin{enumerate}
\item[(i)] For $d \in \Lambda^G$, $H^*(G, \Lambda[d^{-1}]) \simeq H^*(G, \Lambda)[d^{-1}]$, and the same for Tate cohomology $\hat H^*$.
\item[(ii)] For every $n$ there is a short exact sequence
$$0 \to \lim{}^1_k \hat H^{n-1}(G, \Lambda/I^k) \to \hat H^n(G, \Lambda^\wedge_I) \to \lim_k \hat H^n(G, \Lambda/I^k) \to 0,$$
and likewise for $H^*$.
\item[(iii)] If for every $n \in \Z$ and every $k \geq 2$ the inclusion $I^k \hookrightarrow I^{k-1}$ induces the zero map $\hat H^n(G, I^k) \to \hat H^n(G, I^{k-1})$, then the natural maps
$$\hat H^n(G, \Lambda) \xrightarrow{\ \simeq\ } \lim_k \hat H^n(G, \Lambda/I^k) \xleftarrow{\ \simeq\ } \hat H^n(G, \Lambda^\wedge_I)$$
are isomorphisms for every $n$: Tate cohomology does not see the completion.
\end{enumerate}}
\end{lemma}

\begin{proof} (i) Group cohomology of a finite group is computed by a complex of finitely generated free $\Z[G]$-modules, so it commutes with filtered colimits, and $\Lambda[d^{-1}] = \mr{colim}(\Lambda \xrightarrow{d} \Lambda \xrightarrow{d} \cdots)$ is a filtered colimit of $G$-modules because $d$ is invariant. The same complex computes Tate cohomology.

(ii) The standard cochain complex $C^\bullet(G, M) = \mr{Map}(G^\bullet, M)$ commutes with limits, so $C^\bullet(G, \Lambda^\wedge_I) = \lim_k C^\bullet(G, \Lambda/I^k)$, and the tower of complexes consists of termwise surjections because $\Lambda/I^{k+1} \to \Lambda/I^k$ is onto. The Milnor sequence for the cohomology of a limit of a tower of surjective cochain complexes is the displayed sequence. The same holds for the complete resolution computing $\hat H^*$.

(iii) The short exact sequences $0 \to I^k \to \Lambda \to \Lambda/I^k \to 0$ give a tower of long exact sequences. The tower $\{\hat H^n(G, \Lambda)\}_k$ is constant with identity maps, so its $\lim$ is $\hat H^n(G, \Lambda)$ and its $\lim^1$ vanishes; the tower $\{\hat H^n(G, I^k)\}_k$ has zero transition maps by hypothesis, so both its $\lim$ and its $\lim^1$ vanish, and moreover the image of $\hat H^n(G, I^k) \to \hat H^n(G, \Lambda)$ is zero for $k \geq 2$, since the map factors through $\hat H^n(G, I^{k-1})$. Splitting each long exact sequence into the short exact sequences of towers
$$0 \to \hat H^n(G, \Lambda) \to \hat H^n(G, \Lambda/I^k) \to \ker\big(\hat H^{n+1}(G, I^k) \to \hat H^{n+1}(G, \Lambda)\big) \to 0 \qquad (k \geq 2)$$
and taking the six-term $\lim$/$\lim^1$ sequence, the outer towers contribute $\hat H^n(G, \Lambda)$ with vanishing $\lim^1$ and $0$ with vanishing $\lim^1$ respectively, so $\lim_k \hat H^n(G, \Lambda/I^k) = \hat H^n(G, \Lambda)$ and $\lim^1_k \hat H^n(G, \Lambda/I^k) = 0$ for every $n$. Now (ii) identifies $\hat H^n(G, \Lambda^\wedge_I)$ with the same limit.
\end{proof}

\begin{remark} The argument is the one in Nave's proof of \cite[Thm.\ 5.7]{nave1999}, run with the hypothesis made explicit. In the sequel \cite[Lemma 6.39]{sequel} the hypothesis of (iii) is supplied by an explicit vanishing: multiplication by the generators of $I$ is zero on the Tate cohomology. \end{remark}

\subsection{Jacobians: Cartier, Frobenius, and the Dieudonn\'e Module}
\label{sec:cartier}
Both sequels locate a one-dimensional formal group inside the Jacobian of a curve and identify its Dieudonn\'e module. We collect the standard tools; references are \cite{oda}, \cite[\S 5]{katzcrys}, \cite{manin}.

\begin{defn} Let $X$ be a smooth proper curve over a perfect field $k$ of characteristic $p$. The \emph{Cartier operator} $\mc{C} \colon H^0(X, \Omega^1) \to H^0(X, \Omega^1)$ is the $p^{-1}$-linear map characterized by $\mc{C}(f^{p-1}df) = df$ and $\mc{C}(f^p \omega) = f \mc{C}(\omega)$; the \emph{Frobenius} $F \colon H^1(X, \OO) \to H^1(X, \OO)$ is the $p$-linear map induced by $f \mapsto f^p$. They are Serre dual: $\langle F\xi, \omega\rangle = \langle \xi, \mc{C}\omega\rangle^p$. The \emph{$a$-number} is $\dim \ker \mc{C} = \dim \ker F$, and the \emph{$p$-rank} is the rank of $F^{g}$, i.e.\ $\dim_{\F_p} \Jac(X)[p](\bar k)$. \end{defn}

\begin{lemma}[Oda] \label{lem:oda} The Dieudonn\'e module of $\Jac(X)[p]$ is $H^1_{\mr{dR}}(X)$ with $F$ and $V$ induced by Frobenius and Cartier, and $\bb{D}(\Jac(X)[p^\infty]) \simeq H^1_{\mr{cris}}(X/W)$ with its Frobenius, compatibly with the Hodge filtration $H^0(\Omega^1) \subset H^1_{\mr{dR}}$. In particular the Hodge line of a one-dimensional formal summand $F \subset \Jac(X)[p^\infty]$ is a line in $H^0(X, \Omega^1)$, and $\omega_F = H^0(\Omega^1) \cap \bb{D}(F)$. \end{lemma}

\begin{lemma}[isotypic splitting] \label{lem:isotypicsplit} If a finite group $G$ of order prime to $p$ acts on $X$, the idempotents of Lemma \ref{equivarianceofidempotent} split $\Jac(X)[p^\infty]$, $H^1_{\mr{cris}}$, $H^1_{\mr{dR}}$, $H^0(\Omega^1)$ and $H^1(\OO)$ compatibly into $\chi$-isotypic pieces, $F$ and $\mc{C}$ preserve the splitting, and the dimension of the $\chi$-block of the $p$-divisible group is $\dim H^0(\Omega^1)^{\chi}$ (for the pullback action; see Convention \ref{conv:stars} for the sign). Its height is $\dim_{W[1/p]} H^1_{\mr{cris}}(X)^{\chi}[1/p]$, and its slopes are the $p$-adic valuations of the Frobenius eigenvalues on that piece \cite{manin}. \end{lemma}

\begin{remark} For the Artin--Schreier curve $y^p - y = x^{p-1}$ the Frobenius eigenvalues are Gauss sums and Stickelberger gives the slopes $j/(p-1)$, $1 \leq j \leq p-2$, one block of height $p-1$ for each nontrivial character of $C_{p-1}$ \cite[Part 3]{sequel}. For the level-$k$ towers the Cartier operator alone locates the Lubin--Tate module at $k = 2$ \cite[Part 4]{higherlevels}. \end{remark}

\bibliographystyle{alpha}
\bibliography{Biblio}

\end{document}